\documentclass{lms}

\usepackage{amsmath,amssymb, amsfonts, amsbsy, latexsym, color,bbm}%,amsthm}
\usepackage{graphicx}
\usepackage{hyperref}

\newtheorem{thm}{Theorem}[section]

\newtheorem{cor}[thm]{Corollary}

\newtheorem{lemma}[thm]{Lemma}
\newtheorem{lem}[thm]{Lemma}
\newtheorem{prop}[thm]{Proposition}

\newnumbered{defn}[thm]{Definition}
\newnumbered{rem}[thm]{Remark}
\newnumbered{conclusion}[thm]{Conclusion}
\newnumbered{remarks}[thm]{Remarks}
\newnumbered{ex}[thm]{Example}
\newnumbered{exc}[thm]{Exercise}
\newnumbered{facts}[thm]{Facts}
\newnumbered{prob}{Problem}
\newnumbered{question}[thm]{Question}
\newnumbered{answer}[thm]{Answer}
\newnumbered{conj}[thm]{Conjecture}
\newunnumbered{notn}[thm]{Notation}

\numberwithin{equation}{section}

\newcommand{\cA}{\mathcal{A}}
\newcommand{\cG}{\mathcal{G}}

\newcommand{\cI}{\mathcal{I}}

\newcommand{\cB}{\mathcal{B}}

\newcommand{\cS}{\mathcal{S}}
\newcommand{\bZ}{\mathbb{Z}}
\newcommand{\bN}{\mathbb{N}}
\newcommand{\bR}{\mathbb{R}}
\newcommand{\bC}{\mathbb{C}}

\newcommand{\bF}{\mathbb{F}}

\newcommand{\rar}{\rightarrow}

\newcommand\up[1]{^{(#1)}}
\newcommand\BB[1]{\mathbb{#1}}
\newcommand\CAL[1]{\mathcal{#1}}
\newcommand\TR[1]{\mathop{\mathrm{tr}}{#1}}
\newcommand\TRQS[2]{{{\mathrm{tr}}_{#1}}{#2}}

\newcommand{\DSYM}{d_{\mathrm{sym}}}
\newcommand{\SYM}{\mathrm{Sym}}

\newcommand{\ip}[2]{\left\langle#1,#2\right\rangle}
\newcommand{\norm}[1]{\left\Vert#1\right\Vert}
\newcommand{\abs}[1]{\left\vert#1\right\vert}
\newcommand{\absip}[2]{\left\vert\langle#1,#2\rangle\right\vert}
\newcommand{\lspan}{\operatorname{span}}

\newcommand{\beq}{\begin{equation}}
\newcommand{\eeq}{\end{equation}}

\providecommand{\tr}{\mathop{\rm tr} \nolimits}
\providecommand{\spark}{\mathop{\rm spark} \nolimits}
\providecommand{\argmin}{\mathop{\rm arg\, min}}
\providecommand{\TP}{\mathop{\rm TP} \nolimits}
\providecommand{\diag}{\mathop{\rm diag} \nolimits}
\providecommand{\Sym}{\mathop{\rm Sym} \nolimits}

\newcommand{\ejk}[1]{{{#1}}}
\newcommand{\bgb}[1]{{{#1}}}

\title[Optimal arrangements of classical \& quantum states]%
{Optimal arrangements of classical and quantum states\\ with limited purity}%

\author{B.\ G.\ Bodmann and E.\ J.\ King}%Bernhard G.\ Bodmann and Emily J.\ King}

\dedication{In memory of John Haas}
    
\classno{81P45, 05B05 (primary), 42C15, 52A40, 52C17 (secondary)}
 
\extraline{The first author was supported by the National Science Foundation grant DMS-1715735. The second author was supported in part by the Zentrum f\"ur Forschungsf\"orderung der Universit\"at Bremen Explorationsprojekt ``Hilbert Space Frames and Algebraic Geometry.''  The second author also received financial support to participate in the Workshop on Finite Weyl-Heisenberg Groups which was part of the Hausdorff Trimester Program Mathematics of Signal Processing and the Summer of Frame Theory at the Air Force Institute of Technology.}

\begin{document}
\maketitle

    \begin{abstract}
We consider sets of trace-normalized non-negative operators in Hilbert-Schmidt balls that maximize their mutual Hilbert-Schmidt distance; these are optimal arrangements in the sets of purity-limited classical or quantum states on a finite-dimensional Hilbert space. Classical states are understood to  be represented by diagonal matrices, with the diagonal entries forming a probability vector. We also introduce the concept of spectrahedron arrangements which provides a unified framework for classical and quantum arrangements and the flexibility to define new types of optimal packings. Continuing a prior work, we combine combinatorial structures and line packings associated with frames to arrive at optimal arrangements of higher-rank quantum states.
One new construction that is presented involves generating an optimal arrangement we call a Gabor-Steiner equiangular tight frame as the orbit of a projective representation of the Weyl-Heisenberg group over any finite abelian group.  The minimal sets of linearly dependent vectors, the so-called binder, of the Gabor-Steiner equiangular tight frames are then characterized; under certain conditions these form combinatorial block designs and in one case generate a new class of block designs.  The projections onto the span of minimal linearly dependent sets in the Gabor-Steiner equiangular tight frame are then used to generate further optimal spectrahedron arrangements.
\end{abstract}

\section{Introduction}

The present paper investigates how packings of quantum states relate to their
classical counterpart. A classical packing problem is, for example, the selection of
a number of vertices in the Boolean cube such that the minimum number of edges between 
any pair of selected vertices is maximized \cite{BBFKKW}. The mutual distance has practical relevance when the 
points in the initial
selection are distorted, for example through noise in a communication system. 
Sets of points maximizing their mutual distance are also known as maximum-distance separable 
codes  \cite{BBFKKW}.
A similar problem is the selection of a number of points on a sphere such that
their respective minimum distance is maximized. Points realizing this distance
are called spherical codes \cite{BargMusin07,BachocVallentin09}.

When classical information is stored or transmitted in the form of a signaling set of quantum states \cite{Holevo98}, then 
there is an intrinsic source of uncertainty that is not related to an external  source
of noise coming from an environment. 
States are modeled by trace-normalized positive semidefinite operators on a finite-dimensional Hilbert space. 
A set of quantum states can only be 
distinguished reliably when they are orthogonal with respect to the Hilbert-Schmidt inner product. When the number of signaling states
grows beyond the dimension of the Hilbert space,
they may not be identified without allowing 
the possibility of an error \cite{Peres1988,scott2006tight}.  In that situation, one may wish to minimize
the probability of confusing one state with another in the signaling set, which amounts to minimizing the pairwise Hilbert-Schmidt inner
products. When the choice of states is unrestricted, this is equivalent to maximizing the mutual distance given by the Hilbert-Schmidt norm.
Although there is a conceptual issue with this as a distance between states \cite{Ozawa00} because a quantum channel may increase the distance between specific pairs of states,
we adopt it here because of its intuitive geometric, Euclidean character. A more refined analysis has been made based on hypothesis testing and 
smoothed conditional entropies, see \cite{DST14} and references therein.

When quantum states are used for information storage \bgb{or transmission}, they are produced by a device, which is modeled by a quantum channel. In recent
years, the maximum output purity has been discussed to characterize the performance 
of a channel \cite{ZanardiLidar04,AFKW04,King18}. If pure states cannot be realized by a channel, then one may ask what
the best packing is one can realize with a given number of quantum states
that have limited purity. This is precisely the question addressed hereafter.
In more concrete mathematical terms, we characterize
sets of trace-normalized non-negative operators in a Hilbert-Schmidt ball of a given radius
that maximize their mutual Hilbert-Schmidt distance.
The radius of the ball bounds the purity of a state. Classical states are understood to 
be represented by diagonal matrices, with the diagonal entries forming a probability vector.
We note that limiting the purity of a probability vector may be less common than for quantum states,
but it is instructive to compare packings of classical and quantum states on the same statistical footing. To unify the treatment, we introduce the concept of spectrahedron arrangements.
As special cases, spectrahedra include the set of trace-normalized positive semidefinite operators or the subset of diagonal matrices whose diagonal entries form probability vectors.

Earlier results on subspace packings~\cite{GrassPack,BodmannHaas18} can be regarded as non-convex optimization problems arising from restricting 
the larger search space of quantum states with a purity limit to multiples of orthogonal projections, a non-convex subset.
The results included here show cases in which optimal packings in the larger search space consist of extreme points, elements of minimal rank
that are multiples of orthogonal projections. Hence, we can deduce optimal subspace packings from our construction of quantum state arrangements.
This is also true when the states are restricted to quantum states obtained from specific group orbits, in which case the corresponding 
subspace packing
is optimal when only subspaces in the orbit under the group operation are included in the optimization.
The group actions we consider are the unitary group acting on quantum states, the permutation group
acting on classical states, or a projective unitary representation relating to the Heisenberg-Weyl group acting on a subset of quantum states.
\bgb{The optimal subspace packings resulting from our construction also have convenient properties when they are considered
as designs \cite{Zauner1999}, for example as tight or even equi-chordal tight fusion frames \cite{PKC08,BE13,FJMW17}.} 
% M. Fickus, J. Jasper, D. G. Mixon, C. E.Watson, A brief introduction to equi-chordal and equi-isoclinic tight fusion frames, Proc. SPIE 10394 (2017) 103940T/1?9.
%A. Pezeshki, G. Kutyniok and R. Calderbank, "Fusion frames and robust dimension reduction," 2008 42nd Annual Conference on Information Sciences and Systems, Princeton, NJ, 2008, pp. 264-268.
%This property of optimal arrangements in the convex set of purity-limited quantum states
%may be of help with the numerical search for subspace packings.

The special case of 
optimal
arrangements of rank-one projections may be viewed as collections of unit vectors whose maximal inner product is as small as possible.  Examples of such optimal arrangements include equiangular tight frames in harmonic analysis and mutually unbiased bases and symmetric informationally complete positive operator valued measures in quantum information theory (see, for example,~\cite{Waldron18}) \bgb{or other arrangements that augment certain equiangular tight frames with an orthonormal basis \cite{BodmannHaas16}}.  A key idea of this paper is to leverage combinatorial designs to create new optimal arrangements from optimal arrangements of rank one projections, generalizing an approach in~\cite{BodmannHaas18}. In some cases, these combinatorial designs arise from the structure of embedded simplices, which are minimal linearly dependent sets in an equiangular tight frame.  This structure is called a binder and was introduced in~\cite{FJKM17}.  In order to go beyond arrangements based on mutually unbiased bases, we construct an infinite class of group covariant frames which  we call Gabor-Steiner equiangular tight frames which are also exactly (not just unitarily equivalent to those) created using the Steiner equiangular tight frame construction~\cite{FMT12}, thus have vector support corresponding to a combinatorial block design.
These frames are interesting in their own right.
 In certain cases, they are equivalent to known indirect constructions which leverage combinatorial designs or group actions but not both \cite{BoEl10a,BoEl10b,IJM17,FJMPW17}.  We also provide the first complete characterization of the binders of an infinite class of equiangular tight frames.

The present  paper is organized as follows:

Section~\ref{sec:prelim} contains background information about classical and quantum state arrangements.  Classical and quantum states are both examples of spectrahedra, and in Definition~\ref{def:specarr} we introduce the new concept of spectrahedron arrangements.  After some statements generalizing known packing results to the spectrahedron arrangement case (e.g., Theorem~\ref{thm:rankin}), we recall facts from combinatorial design theory, in particular balanced incomplete block designs. Among the results that we discuss in the classical setting are balanced incomplete block designs
as the support set of distance maximizers. This was realized in a prior paper in a special case in  which probability vectors had a support that was half the number of outcomes/vertices~\cite{BodmannHaas16}.  Here, in Corollary~\ref{cor:affinedesigninnerproduct} we deduce that a family of designs on   $q$  vertices, where $q$ is a prime power, 
$n=q(q^m-1)/(q-1)$ blocks of size $k=q^{m-1}$, intersecting in at most
$q^{m-1}$ vertices gives rise to $n$ probability vectors that saturate a bound by Rankin for packings  in a simplex while purity is limited by $1/k$. 

%%\ejk{I broke things up because having a 22 page section seemed a bit unbalanced}

Mutually unbiased bases are introduced in Section~\ref{sec:MUB}. We use these
optimal packings of $q(q+1)$ rank one projections in $\bC^q$, where $q$ is a prime power
to increase the packing density relative to the classical setting. In Theorem~\ref{thm:MUBpack}, we obtain  $n=(q^m+1) \frac{q(q^m-1)}{q-1}$  quantum states with purity $q^{1-m}$ and the Hilbert-Schmidt distance between any two quantum states is bounded in the same way as in the packing of (classical) probability vectors.

In Section~\ref{sec:gabor}, a construction of an infinite class of equiangular tight frames which are generated as the orbit of a projective unitary representation over the Heisenberg-Weyl group over any finite abelian group of odd order is presented.  It is shown in Theorem~\ref{thm:Steiner} that the support set of the vectors corresponds to a balanced incomplete block design, making these equiangular tight frames the first class of group covariant frames which are also Steiner equiangular tight frames.  We thus call the generated frames Gabor-Steiner equiangular tight frames.

The binders of the Gabor-Steiner equiangular tight frames are characterized in Section~\ref{sec:BIBD} and \bgb{later} used to create spectrahedron arrangements. In particular, it is shown in Theorems~\ref{thm:AG}, ~\ref{thm:bibdsets}, and~\ref{thm:AG2} that if the underlying group used to generate the Gabor-Steiner equiangular tight frame is $\bZ_p$ or $\bZ_p \times \bZ_p$ for $p$ an odd prime, then the binders are balanced incomplete block designs.  In the former case, the designs are affine geometries but in the latter case, the designs seem to be a new class (Theorem~\ref{thm:bibddecomp}). 

\bgb{In Section~\ref{sec:SAfromB}, we leverage the binder structure from the preceding section} to create two different types of optimal arrangements. For example, one result of Theorem~\ref{prop:ff1} is that we can use a subset of the simplices in a binder to create for all odd integers $m$ a set of quantum states of purity $(m-1)^{-1}$ in a $m(m-1)/2$-dimensional complex Hilbert space that have a constant distance from each other; that is, they are equichordal.  By considering the projections onto the entire binders of the Gabor-Steiner equiangular tight frames corresponding to $\bZ_p$ or $\bZ_p \times \bZ_p$, we construct other spectrahedron arrangements (Corollaries~\ref{cor:ff2} and~\ref{cor:ff3}) which in the former case are optimal.

Finally, in Theorem~\ref{thm:gabMUB} in Section~\ref{sec:BinderMUB} we relate
maximal sets of mutually unbiased bases and a class of non-maximal equiangular tight frames.

\section{Preliminaries}\label{sec:prelim}

Let $\BB F$ be the field of real or complex numbers, $\BB R $ or $\BB C$. We equip the $d$-dimensional vector space $\BB{F}^d$ with the canonical
inner product to make it a Hilbert space. The standard orthonormal basis $\{e_i\}_{i=1}^d$ is the set of vectors with entries given by
Kronecker's $\delta$-symbol, $(e_i)_j = \delta_{i,j}$. The matrix units in $\BB{F}^{d \times d}$ are $d\times d$ matrices $\{E_{i,j}\}_{i,j=1}^n$ whose entries
are $(E_{i,j})_{l,m} = \delta_{i,l} \delta_{j,m}$ also written as $E_{i,j} = e_i \otimes e_j^*$.

\begin{defn}\label{def:states}
The space of \textit{quantum states} on $\BB{F}^d$ is the space of trace-normalized positive semidefinite matrices, $\{W \in \BB{F}^{d\times d}:
\TR{W}=1, W \succeq 0\}$. When speaking of \textit{classical states}, we mean the subset of diagonal quantum states, 
$\mathcal D = \{D \in \BB{F}^{d\times d}: \tr{D} = 1, D \succeq 0, 
D = \sum_{i=1}^d E_{i,i} D E_{i,i}\}$. 
For a classical or quantum state $W$, its \textit{purity} is the quantity $\TR{W^2}$. A state
given by a rank one projection $P$ is called \textit{pure}.
A set of states $\{W_j\}_{j=1}^n$ is called \textit{coherent} if its average equals the \textit{maximally mixed state}
$$
    \frac{1}{n} \sum_{j=1}^n W_j = \frac{1}{d} I \, .
$$
Given a group $G$ acting on $\BB{F}^{d}$ by a representation $\rho: G \to \BB{F}^{d\times d}$, we say that a subset of states $\mathcal W$ is \textit{$G$-invariant} 
if for each $W \in \mathcal W$ and $g \in G$, $\rho(g) W \rho(g^{-1}) \in \mathcal W$.
\end{defn}

If we consider the set of all quantum states, then this is invariant under conjugating with unitaries. The set of diagonal matrices is invariant under conjugation with permutation
matrices, and so are the classical states. In both cases these groups act irreducibly on $\mathbb{F}^d$. By versions of Schur's lemma,
if a set of states $\mathcal W = \{W_j\}_{j=1}^n$ is invariant under the orthogonal or unitary matrices, then
it is coherent. This property allows one to express the value of the trace of any symmetric or Hermitian matrix in terms 
of the average of the Hilbert-Schmidt inner products with the quantum states. Zauner
viewed this property as that of a (quantum) 1-design, in analogy with combinatorial design theory.
He also studied higher-degree design properties. A quantum $t$-design with respect to $\rho$ is present when $\sum_{j=1}^n \otimes_{j=1}^t W_j$
is invariant under conjugation with $\otimes_{j=1}^t \rho(g)$ \cite{Zauner1999,Zauner2011}. 

Instead of averaging properties, in this paper we focus on packing of quantum states,
in close analogy with codes in  the combinatorics literature. Although there is at first no direct connection between
the definition of codes and the averaging properties of designs, there is an interplay between the two properties,
especially when codes of maximal size or designs of minimal size are concerned~\cite{Levenshtein1992}.   A code is, in brief, a set of states that maximizes the minimal distance between any two elements. We choose the chordal distance as the relevant metric.

\begin{defn} For a pair of quantum states $W$ and $W'$, we define the \textit{chordal distance} $d_c(W,W')=(\TR{(W-W')^2})^{1/2}/\sqrt 2$.
An \textit{arrangement of quantum states} is a subset $\mathcal{W} = \{W_j\}_{j=1}^n$ of quantum states. We say that this set is an \textit{optimal 
packing of purity $\gamma$} if $\TR{W_j^2} \le \gamma$ for each $j \in \{1, 2, \dots, n\}$ and among all choices of $n$ such states, $\mathcal W$
maximizes the minimal distance occurring between any pair $W,W' \in \mathcal{W}$. 
\end{defn}

The chordal distance is, up to a normalization factor, a natural metric induced by the Hilbert-Schmidt norm when
states are embedded in the real Hilbert space $\SYM_d(\mathbb{F})$ of symmetric/Hermitian $d\times d$ matrices over $\BB{F}$.
For convenience, we abbreviate the dimension of this Hilbert space
as
 $$\DSYM = \left\{ \begin{array}{cc} d(d+1)/2, & \BB{F}= \BB{R} \\  d^2 , & \BB{F} =\BB{C} \end{array} \right. .$$

By the definition of chordal distance, we have for two quantum states $W, W'$ with purity at most $\gamma$ % in an optimal packing of purity $\gamma$ \ejk{This is true for any two states of purity $\le \gamma$} 
the bound
$$
   d_c(W,W')^2 = (\TR{W^2} - 2 \TR{W W'} + \TR{(W')^2} )/2 \le  \gamma -  \TR{W W'} \, . 
$$
Hence, if $\mathcal W$ saturates a lower bound for $\max_{W,W' \in \mathcal{W},W \ne W'} \TR{W W'}$
while $\TR{W^2}=\gamma$ for each $W \in \mathcal{W}$, then an optimal packing is achieved. In particular,
if $\gamma=1$, then a set of rank-one projections $\mathcal{P}=\{P_j\}$ minimizing $\max_{P, P' \in \mathcal{P} }
\TR{P P'}$ is optimal. The problem of finding optimal line packings, the subspaces associated with the rank-one projections,  has a long history
in real or complex Hilbert spaces \cite{Tammes30,ConwaySloane98}, see also \cite{Waldron18}.

In this paper, we wish to address the packing problem for states with a purity limit $\gamma \in [1/d,1]$.  Hence, we allow the purity limit to range between that of the maximally mixed state and a pure state. In addition, we choose groups other than the unitaries
and the permutations and corresponding convex subsets of quantum states that are invariant under these groups to pose and study packing problems.  The most general type of packing problem
covered here is to find states within a given spectrahedron, a type of set appearing in statistics, graph theory \cite{GoWi1995}, and quantum information theory \cite{Weis2011}.

\begin{defn}\label{def:spec}
A \textit{spectrahedron}  $\mathcal S$ in $\SYM_d(\mathbb{F})$ is the intersection of 
a real, affine-linear space of symmetric/Hermitian $d\times d$ matrices with the convex cone of positive 
semidefinite matrices. If $\mathcal S$ contains all positive semidefinite matrices with unit trace,
it is called a \textit{spectraplex} \cite{GaMa2012}. If $\mathcal S$ is a spectrahedron containing a positive multiple of the identity 
in its interior then we call it \textit{monic} \cite{EHKM2018}.
When referring to the dimension of a spectrahedron, we mean the (real) dimension
of the affine subspace in $\SYM_d(\mathbb{F})$.
\end{defn}

\begin{ex}
The trace-normalized elliptope $\mathcal{E}_3$ consists of all the real $3\times 3$ matrices
$$
   X = \left( \begin{array}{ccc} 1/3 & x & y\\ x & 1/3 & z\\ y & z & 1/3 \end{array} \right)            
$$
such that all of its principal minors are non-negative. Larger size matrices 
of this form appear in statistics as covariance matrices and also have significance for
the construction of graph cuts \cite{GoWi1995}. We note that this spectrahedron is invariant under
conjugation with $3\times 3$ permutation matrices.
\end{ex}
\begin{ex}
The Bloch ball in quantum information theory is the set of $2\times 2$ matrices
$$
  W = \frac{1}{2}\left( \begin{array}{cc} 1+z & x+iy\\ x-iy & 1-z  \end{array} \right)  
$$
with real parameters $x, y$, and $z$ satisfying $x^2+y^2+z^2 \le 1$. It is precisely the spectraplex in $\SYM_2(\mathbb{C})$
and models quantum states of a qubit. This set of states is invariant under conjugation
with $2\times 2$ unitaries. More generally, the spectraplex in $\SYM_d{(\mathbb{C})}$ as defined above is the convex set of quantum states
on the Hilbert space $\mathbb{C}^d$.
\end{ex}

We remark that after choosing a basis $\{A_j\}_{j=1}^D$ for the corresponding subspace of dimension $D$ obtained from all differences of pairs of vectors in $\mathcal S$,
a spectrahedron is necessarily of the form 
$$
  {\mathcal S} = \Bigl\{A(x) = A_0+\sum_{j=1}^D x_j A_j, x \in \mathbb{R}^D, A(x) \succeq 0  \Bigr\} \, .
$$
If $\mathcal S$ is monic, we can choose $A_0= \lambda I$, $\lambda> 0$, and $\TR(A_j)=0$ for $j \in \{1, 2, \dots, n\}$.

%
% 
% We wish to maximize the minimal chordal distance occurring
%among all pairs of subspaces.
% Because of the relationship between the chordal distance and the Hilbert-Schmidt inner product
%between projections, this is equivalent to minimizing the maximal
%inner product among pairs of projections onto the subspaces. 

Spectrahedra are the feasibility sets of semidefinite programs \cite{GaMa2012}. Here, we wish to study the problem of
packings in spectrahedra.

\begin{defn}\label{def:specarr} A subset of $n$ elements in a spectrahedron $\mathcal S \subset \SYM_d(\mathbb{F})$ is called an {\it $(n,d)$-spectrahedron arrangement}.
When the spectrahedron is a subset of the spectraplex, and $n$ elements are chosen in such a way that the set of states $\mathcal W = \{W_j\}_{j=1}^n$
maximizes $\min_{j,l \le n, j \neq l} d_c(W_j ,W_l)$
among all such arrangements of purity $\tr{W_j^2} \le \gamma$ in $\mathcal S$, then we say $\mathcal W$ is an {\it optimal spectrahedron arrangement} of purity $\gamma$.
\end{defn}
%\ejk{Mention Martin Ehler's unequal dimensional optimal packings as an example?}

%The Grassmannian $\GR{d}{k}$ over $\BB R$ or $\BB C$ can be viewed as the orbit of any rank-$k$ projection under the group of  orthogonal or unitary matrices, respectively. For an orthogonal or unitary matrix $U$, the group action
%is obtained from mapping $P \in \GR{d}{k}$ to $ U P U^*$. This action is an isometry on the Grassmannian.
%
%There are other subspaces of $\SYM_d(\mathbb{F})$ and other group actions that preserve the chordal distance. For example,
%one can consider the space of diagonal $d\times d$ projection matrices of rank $k$ and the symmetric group
%$\mathcal{S}_d$ acting by conjugation with permutation matrices.
%
%In both cases, the group action preserves the trace and positivity. It is natural to extend this
%to the convex hull of projections or even to the set of positive semidefinite matrices spanned by the projections.
%This type of set is called a spectrahedron.
%

With the help of the Euclidean metric induced by the Hilbert-Schmidt norm, a sphere packing bound given by 
Rankin~\cite{Ran55} 
can 
be reformulated in terms of inner products for elements in a spectrahedron, generalizing an approach in~\cite{GrassPack}. 
Rankin shows that $n$ points $\{v_j\}_{j=1}^n$ on the unit sphere in $\mathbb{R}^D$
provide an optimal packing if the maximal inner product saturates the bound
$\max_{j \ne l} \langle v_j, v_l\rangle \ge -1/(n-1)$, which requires $n \le D+1$.
This occurs precisely when the vectors form an equatorial simplex, meaning
the inner products are all equal to $-1/(n-1)$ and the vectors sum to zero. For $n>D+1$,
the lower bound improves to  $\max_{j \ne l} \langle v_j, v_l\rangle\ge 0$.

\begin{thm}\label{thm:rankin}
If $\{W_j\}_{j=1}^n$ are elements of a spectrahedron of dimension $D$ in $\SYM_d(\mathbb{F})$ with trace $\TR(W_j) = 1$ and each $W_j$ has purity $\TR(W_j^2) = \gamma$,
then 
\begin{equation}\label{eqn:Rankin1}
    \max\limits_{j,l \le n, j \neq l} \TR( W_j  W_l )  \geq \frac{n-\gamma d}{(n-1)d}
\end{equation}
and if equality is achieved then $n \le D+1$. In this case, the inner product $\TR( W_j  W_l )$
equals the lower bound for each $j \ne l$ and $\sum_{j=1}^n W_j = (n/d) I$.
If  $n >  D+1$, then
\begin{equation}\label{eqn:Rankin2}
\max\limits_{j,l \le n, j \neq l} \TR(W_j W_l) \geq \frac{1}{d}
\end{equation}
and if equality is achieved then $n \le 2 D $. 
%Moreover, if $n=2 \DSYM-2$, then
%equality in this bound implies that for each $j \le n$, $\TR(P_j P_l) =\frac{2(2k-d)}{d}$ for exactly one $l \ne j$  and $\TR(P_j P_l) = \frac{k^2}{d}$ for all other $l \ne j$.
\end{thm}
\begin{proof}
The proof is obtained from Rankin's bound by mapping $\{W_j\}_{j=1}^n$ to trace-zero matrices $\{Q_j\}_{j=1}^n$,
$Q_j = W_j - \frac{1}{d}I$. By orthogonality, the squared Hilbert-Schmidt norm of each $Q_j$ is $\TR(Q_j^2)
= \gamma - \frac{1}{d}$.
Normalizing each $Q_j$ and then applying Rankin's  bounds in the Euclidean space spanned by these zero-trace matrices
gives the claimed inequalities and characterization of cases of equality. In particular, the zero-summing property
of optimal zero-trace matrices implies that $\sum_j W_j = (n/d) I$, so if the first inequality is saturated, then the spectrahedron is monic.
\end{proof}
We call \eqref{eqn:Rankin1} the \emph{first Rankin  bound} and \eqref{eqn:Rankin2} the \emph{second Rankin  bound}  for spectrahedra.

We can convert between inner products and distances in Hilbert-Schmidt norm to derive a packing bound.
In that case, we can allow the purity of the states to be bounded above instead of fixing it.

\begin{cor}
If $\{W_j\}_{j=1}^n$ is a set of states contained in a  spectrahedron of dimension $D$  in $\Sym_d(\mathbb{F})$ and
the purity of each state is bounded by $ \TR{W_j^2} \le \gamma$,
then $\min_{j \ne l} \TR{(W_j - W_l)^2 } \le  \frac{(2\gamma d-2)n}{(n-1)d} $. If equality is achieved, then $n \le D+1$,
the distance is a constant between any pair, and the states are coherent. If $n>D+1$, then
$
  \min_{j \ne l} \TR{(W_j - W_l)^2 } \le 2 \gamma - \frac{2}{d} \, .
$
\end{cor}
\begin{proof}
This consequence is verified using the identity $\tr[(W_j - W_l)^2] = \tr[(Q_j - Q_l)^2]$
and then by applying the bounds as in the proof of Theorem~\ref{thm:rankin}, $\tr[ Q_j Q_l ] \ge -(\gamma - \frac 1 d)/(n-1)$ when $n \le D+1$ or
$ \tr[ Q_j Q_l ] \ge 0$ when $ n >D+1$.
\end{proof}

We note that as the purity limit decreases, so does the first upper bound.  This can be expected from the fact that
the Hilbert-Schmidt ball available for the packing shrinks as $\gamma$ decreases.
In the constructions of optimal packings described below, we specialize to the choice $\gamma = 1/k$, with $k \in \mathbb N$. 
The optimal states emerge
as scaled rank-$k$ orthogonal projections. This is the minimal rank that can be achieved for such a purity limit. It can also be verified
that such scaled projections are extreme points in the set of purity-limited quantum states, because a non-trivial convex combination of projections has a higher rank than each of its summands unless their ranges coincide, according to Weyl's inequalities~\cite{HornJohnson}.  
The same result holds when we consider scaled projections among the classical states. More generally, we consider 
packing problems in spectrahedra
that are invariant under group operations. The quantum states are invariant under conjugating with unitaries, the classical states under
the action of the permutation group, and a third result is obtained from the action of the Heisenberg-Weyl group.

The cases of optimal packings we derive are then also cases of optimal subspace packings 
\bgb{within an appropriate set of subspaces.
Identifying classical states with the probability vectors formed by their diagonal entries, we can reformulate the packing problem
as distance maximization in a convex set whose dimension is $D=d-1$. The optimal purity-limited packings we construct are 
formed by probability vectors that are constant on a support set of fixed size~$k$. 
When the support of each probability vector is identified with the corresponding coordinate subspace,
then this yields an optimal packing among $k$-dimensional coordinate subspaces.
%After rescaling the vectors, these can then be identified   
%with optimal constant-weight binary codes, see \cite[Inequality (14)]{Johnson72} or \cite[Table 8.4]{Levenshtein1992}.
In the case of purity-limited quantum states, the optimizers saturate
the bounds appearing already in the work by Conway, Hardin and Sloane~\cite{GrassPack}. 
The dimension of the convex set in which the states reside is then maximal, $D=\DSYM-1$.}
%In that case,
%the optimization is restricted  to quantum states of rank-$k$ and purity $1/k$, hence scaled orthogonal projections.

If equality holds in either of the bounds in Theorem~\ref{thm:rankin}, then an optimal packing is achieved. If $n \le D+1$, the operators $\{W_j - \frac{1}{d} I\}_{j=1}^n$ obtained from the optimal set
$\{W_j\}_{j=1}^n$ form a simplex in a subspace of the zero-trace Hermitians. 
Since the Hilbert-Schmidt inner product between any two such states is then a constant,
 they are called equiangular.

\begin{defn}A subset of $n$ elements in the spectraplex of trace-normalized positive semidefinite matrices in $\SYM_d(\mathbb{F})$ 
with purity bound $\gamma$ is 
called an $(n,d,\gamma)$-spectraplex arrangement.
Such an arrangement $\mathcal W$ is called {\it equiangular} or {\it equichordal} if 
each element $W$ has purity $\TR(W^2) = \gamma$ and
there exists $c \ge 0$ such that 
for each  $W, W' \in \mathcal W$ with $W \ne W'$, $\TR(W W') = c$. 
%\ejk{We either need to add back in ``subspace arrangement'' or update all of the uses of it in the paper}
%If the elements in the spectraplex are projections onto subspaces, then we call this
%an equiangular/equichordal $(n,k,d)$-subspace arrangement.
\end{defn}

%From the fact that the states are coherent, one then has $1 = \TR{W_j} = \frac d n \sum_{l=1}^n \TR{ W_j W_l} = \gamma + (n-1) \frac{n-\gamma d}{(n-1)d} = \frac{n}{d}$. 

In the special case of purity $\gamma=1$,
equality in the first lower bound in Theorem~\ref{thm:rankin} is equivalent to the existence of an equiangular tight frame.
There are many examples of Rankin-bound achieving spectraplex-arrangements
given by rank-one projections. In this case, unit vectors are often chosen as representatives of the projections.

\begin{defn}\label{def:ETF}
Let $\mathbb{F} = \mathbb{R}$ or $\mathbb C$. A {\it frame} for $\mathbb{F}^d$ is a spanning sequence $\Phi = \{\varphi_j\}_{j=1}^n$. If the frame vectors have a constant norm $\|\varphi_j\|=\nu$ for each $j \in \{1, 2, \dots, n\}$, if the orthogonal projections $\{(1/\nu^2)\varphi_j \otimes \varphi_j^*\}_{j=1}^n$ form an equiangular spectraplex arrangement, and if $\sum_{j=1}^n \varphi_j \otimes \varphi_j^*$ is a constant multiple of the identity, then $\Phi$ is called and {\it equiangular tight frame},
often abbreviated as {\it ETF}.
\end{defn}

We will abuse notation slightly by also allowing $\Phi$ to denote the matrix with $\{\varphi_j\}_{j=1}^n$ as the ordered set of column vectors.  Further, we note that an arrangement of rank-one projections saturates the first Rankin spectrahedron bound  \eqref{eqn:Rankin1} if and only if it is associated with an equiangular tight frame (see, for example~\cite{StH03}).
For sufficiently small $n$, it is known that if such frames  exist, then the rank-one orthogonal projections whose range is the span of each frame vector
provide such projections. These frames can only exist up to $n=\DSYM$ (private correspondence with Gerzon cited in \cite{LS1973}). In the complex case, the existence of such extremal cases is an open question posed by Zauner.

\begin{conj}[\cite{Zauner1999,Zauner2011}]\label{conj:Zau}
For each $d\in \mathbb N$, an equiangular tight frame
of $n=d^2$ vectors can be obtained from the orbit of a vector under the irreducible action of the Heisenberg-Weyl group
on $\mathbb{C}^d$.
\end{conj}

There is a growing body of evidence supporting Zauner's conjecture (see, e.g., \cite{ScGr10,GrSc17}) but as of yet, it has not even been proven that there are infinitely many $d$ yielding an equiangular tight frame of $d^2$ vectors in $\bC^d$, regardless of group covariance.  For a general overview of the state of the problem, as well as the connections between such maximal equiangular tight frames and algebraic number theory, Lie and Jordan algebras, $2$-designs, finite groups, stochastic matrices, and quantum information, see \cite{FHS17,AFZ15} and the references therein.

Equality in the second bound in Theorem~\ref{thm:rankin} can only hold if $n\le 2 \DSYM-2$.
A maximal set $\{W_j\}_{j=1}^n$  of size $n=2\DSYM-2$ achieving the second lower bound
is related to a set of operators $\{W_j - \frac{1}{d} I\}_{j=1}^n$ that form an orthoplex in the space of zero-trace Hermitians.  This
explains why the \emph{second Rankin bound}  \eqref{eqn:Rankin2} is also called the orthoplex bound.

\begin{defn}
Let $\mathcal S$ be a spectrahedron of dimension $D$ in $M_d(\mathbb{F})$.
An arrangement $\{W_j\}_{j=1}^n$ of states in $\mathcal S$ is called {\it orthoplex-bound achieving} if 
$n>D+1$ and 
$$
   \max_{j,l \le n, j \ne l} \TR{(W_j W_l)} = \frac{1}{d} \, .
$$
\end{defn}

  There are also examples that saturate \eqref{eqn:Rankin2}, like
the projections associated with maximal sets of mutually unbiased bases.
%\ejk{Notation here and with the MUB construction from the binder (Section~\ref{sec:BinderMUB}) should be made uniform.}
\begin{defn} Let $\mathbb{F} = \mathbb{R}$ or $\mathbb C$.
A family of vectors $\Phi=\{\eta^{(j,i)}\}$ in $\mathbb{F}^d$ indexed by $i \in \BB K = \{1, 2, \dots d\}$ and 
$ j \in \BB J=\{1, 2, \dots r\} $
is said to form $r$ \textit{mutually unbiased bases} if 
for all $j,j' \in \BB J$ and $k,k' \in \BB K$
the magnitude of the inner product between $\eta\up{j,i}$ and $\eta\up{j',i'}$ is given by
$$
   |\langle \eta\up{j,i} , \eta\up{j',i'} \rangle | =  \delta_{i,i'} \delta_{j,j'} + \frac{1}{\sqrt d} (1-\delta_{j,j'}) \, ,                              
$$
where Kronecker's $\delta$ symbol is one when its indices are equal and zero otherwise.
If $\mathbb F = \mathbb R$ and $r=d/2+1$ or $\mathbb F = \mathbb C$ and $r=d+1$ then
we say that $\Phi$ is a {\it maximal set of mutually unbiased bases}.
\end{defn}

%\begin{defn} Let  $d, k \in \BB N$. When the ground field $\BB F$  is chosen as $\BB R $ or $\BB C$ and $d \in \BB N$, we equip the vector space $\BB{F}^d$ with the canonical basis
%and the associated bilinear or sesquilinear inner product.
%  The Grassmannian $\GR{d}{k}$ over $\BB F$ is the set of subspaces 
%of dimension $k$ in $\BB{F}^d$. We identify each such subspace $V$ with the associated orthogonal $d\times d$ projection matrix of rank $k$
%whose range is $V$
%and write
%$$
%    \GR{d}{k} = \{P \in {\BB F}^{d\times d}: P = P^* P, \TR(P)= k\} \, .
%$$
%The {\it chordal distance} between $P, P' \in \GR{d}{k}$ is defined by 
%$
%   d_c(P,P') = \frac{1}{\sqrt 2} \| P - P'\|_{HS} = (k - \TR( P P' ) )^{1/2} \, . 
%$
%\end{defn}

%We will in particular be interested in minimal sets of linearly dependent vectors in a given ETF.  These are not only of interest in compressed sensing, but characterizing their structure within an ETF helps us understand the ETF itself better.
%\begin{thm}[\cite{DE03,BDE09}]\label{thm:spark}
%Let $\Phi = \{ \varphi_j\}_{j = 1}^n$ be a collection of vectors in $\bF^d$. Define the {\it spark} to be the size of the smallest linearly dependent subset of $\Phi$.  Then
%\beq\label{eqn:spark}
%\spark \Phi \geq 1 + 1/\max_{j,l\leq n, j \neq l} \absip{\varphi_j}{\varphi_l}.
%\eeq
%\end{thm}
%If $\spark (\Phi) = d+1$, then $\Phi$ has no non-trivial linear dependence relationships between the vectors and is called \emph{full spark}.

A maximal set of mutually unbiased bases saturates the orthoplex bound and contains 
$n=(d/2+1)d$ vectors if $\mathbb F = \mathbb R$ and $n=d(d+1)$ if $\mathbb{F}=\mathbb{C}$.
In both cases, $n>\DSYM$, so the corresponding (pure) states form an optimal packing
in the spectraplex of $\SYM_d({\mathbb F})$.

For the following, we introduce the notion of affine unbiased quantum state arrangements.
These arrangements generalize mutually unbiased bases to the purity-limited case.

\begin{defn}
An {\it affine unbiased quantum state arrangement} of purity $\gamma$ is a collection of states $\mathcal A = \{W_j\}_{j \in J}$ on $\mathbb{F}^d$ such that each
state $W_j$ satisfies  $\TR{W_j^2} = \gamma$ and
$\mathcal A$  can be partitioned into subsets $\{\mathcal{A}_j\}_{j=1}^\rho$ whose states are mutually orthogonal
and sum to a multiple of the identity, and there is $\mu>0$ such that if two states belong to different subsets, $W \in \mathcal{A}_j$, $W' \in \mathcal{A}_{j'}$, $j \ne j'$, then $\TR{W W' } = \mu$.
\end{defn}

Affine unbiased quantum state arrangements can be constructed with diagonal projection matrices
associated with affine block designs, as we see below.

\begin{lemma}
If each subset $\mathcal{A}_j$ in the partition of an affine unbiased quantum state arrangement $\mathcal A$ 
of purity $\gamma$ has size $\sigma$, then
$\gamma = \sigma/d$. Moreover, if the partition has $\rho$ subsets, then  $\mu = 1/d$.
\end{lemma} 
\begin{proof}
If the subset $\mathcal{A}_j$ contains $\sigma$ states summing to a multiple of the identity, then 
$$
   \sum_{W \in \mathcal{A}_j} W = \kappa I 
$$
with $\kappa=\sigma/d$, obtained from taking the trace on both sides. From the orthogonality
and the normalization we can then conclude 
$$
   \sum_{W \in \mathcal{A}_j} \TR{W^2} = \sigma \gamma = \kappa^2 \TR{I} = \frac{\sigma^2}{d} \, ,
$$ 
which implies $\gamma=\sigma/d=\kappa$.
Next, summing $\TR{ W W' } = \mu$ over all $W' \in \mathcal{A}_{j'}$ with $W \in \mathcal{A}_j$ fixed and $j \ne j'$, 
gives $\mu \sigma = \kappa = \sigma/d$.
\end{proof}

\begin{prop}
If an affine unbiased quantum state arrangement $\mathcal A = \{W_j\}_{j=1}^n$ contains $n$ states in $\mathbb{F}^d$
and is partitioned into sets of mutually orthogonal states of size $\sigma$
and $n > \DSYM$, then $\mathcal A$ forms an optimal $(n,d)$-spectraplex arrangement.
\end{prop}
\begin{proof}
If two states $W_j$ and $W_{j'}$ belong to the same subset $\mathcal{A}_j$ in the partition of $\mathcal A$ and $j \ne j'$,
then $\TR{W_j W_{j'}} = 0$. If they belong to different subsets, then $\TR{W_j W_{j'}}= 1/d$, so
$
  \max_{j \ne j'} \TR{ W_j W_{j'} } = 1/d \, .
$
The orthoplex bound is saturated
and by the condition on the size $n$, the states form an optimal packing in the spectraplex.
\end{proof}

Choosing a spectrahedron embedded in an affine subspace of dimension $D < \DSYM-1$ implies that the second Rankin spectrahedron bound
already applies for sets of smaller size $n$.
For example, when restricting to the spectrahedron of diagonal matrices in $\Sym_d(\mathbb{F})$, then $D=d-1$ and \eqref{eqn:Rankin2} is relevant for $d < n \le 2d-2$. The same conclusion holds for the set
of positive semidefinite, trace-normalized circulant matrices with $D=(d-1)/2$ if $d$ is odd and $\mathbb F = \mathbb R$, $D=d/2$ if $d$ is even and $\mathbb F = \mathbb R$, and $D=d-1$ if $\mathbb F = \mathbb C$.

\begin{cor}\label{cor:diagonalsimplex}
Let $\mathcal{D}$ be the convex set of all classical states, diagonal $d\times d$ matrices with non-negative entries and trace $1$. For states $\{W_j\}_{j=1}^n$ in $\mathcal D$ with purity $\TR(W_j^2) = \gamma$, we have
$ \max_{j \ne l} \TR W_j W_l \ge (n-\gamma d)/((n-1)d)$.  If $n>d$, then $\max_{j \ne l}\tr W_j W_l  \ge 1/d$.
\end{cor}

In the following, we will construct such optimal arrangements. The key idea to realize them is to first consider
unbiased arrangements of diagonal projections. These can be obtained with combinatorial design theory (see, e.g., \cite{Stinson2004,Handbook}).

\begin{defn} A {\it balanced incomplete block design} (BIBD) $\mathcal B$ is a collection of subsets of $\{1,2, \dots, v\}$, also called blocks, such that
each block $\beta \in \mathcal B$ has the same size $k$ and any pair of distinct elements in $\{1,2, \dots, v\}$ is contained in precisely $\lambda$ blocks.
We also call $\mathcal B$ a $(v,k,\lambda)$-BIBD. Such a collection of subsets $\mathcal B$ is called an {\it affine block design} if it can be partitioned into
{\it parallel classes},
subsets $\{\mathcal{B}_j\}_{j=1}^\rho$ such that the blocks in each $\mathcal{B}_j$ partition $\{1,2, \dots, v\}$, and there is $\mu \in \mathbb N$ such that 
if $\beta \in \mathcal{B}_j$ and $\beta' \in \mathcal{B}_l$ with $j \ne l$,
then $|\beta \cap \beta'| = \mu$.
If $\cB$ can be partitioned into two nonempty collections $\cB_1$ and $\cB_2$ such that $\cB_1$ is a $(v,k,\lambda_1)$-BIBD and $\cB_2$ is a $(v,k,\lambda_2)$-BIBD, then we call $\cB$ {\it decomposable}. If no blocks are repeated, then $\cB$ is {\it simple}.  Finally, if there exists an $m \in \bN$ such that a $(v,k,\lambda/m)$-BIBD exists, then a $(v,k,\lambda)$-BIBD is an {\it $m$-multiple} or {\it quasimultiple}.
%
%has parameters $(v,b,r,k,\lambda)$. 
%Because the parameters are dependent the structure of the BIBD is determined by $v$, $k$ and $\lambda$. We often refer to such a BIBD as $(v,k,\lambda)$-BIBD.
%A BIBD that can be grouped into $\rho$ subsets $\{\mathcal{X}_i\}_{i=1}^\rho$ such that each
%$\mathcal{X}_i$ forms a partition of $\{1,2, \dots, v\}$
%is called resolvable. A resolvable BIBD whose number of blocks is $b=v+\rho-1$
%is called an affine resolvable BIBD.
\end{defn}

We summarize well-known relationships between the parameters of an affine block design.
To this end, we identify each block $\beta$ with the corresponding $v \times v$ diagonal projection matrix $D_\beta$, with $(D_\beta)_{j,j} = 1 $ if $j \in \beta$ and 
all other entries vanishing. In terms of these matrices, a $(v,k,\lambda)$-design is associated with a set of diagonal $v\times v$ projection matrices $\{D_\beta\}_{\beta \in \mathcal B}$
with $ \TR(D_\beta) = k$ and if $j \ne l$, then $\sum_{\beta \in \mathcal B} (D_\beta)_{j,j} (D_\beta)_{l,l} = \lambda $. 
In this matrix notation, the relationships between the parameters \cite{Stinson2004} are obtained in a convenient way.

%\ejk{Wasn't this a new proof of yours, although the result was known?}
\begin{prop}\label{prop:pcsize}
If a $(v,k,\lambda)$-BIBD $\mathcal B$ is an affine block design for which the blocks between different parallel classes
intersect in $\nu$ elements
and each parallel class contains $\sigma$ blocks, then $k = \nu \sigma$, $v = \nu \sigma^2$
and the number of parallel classes is $\rho = \lambda \frac{\sigma^2 \nu-1}{\sigma\nu-1}$.
\end{prop}
\begin{proof}
If $\mathcal B$ is an affine design, then the trace-normalized diagonal $v \times v$ matrices $\{(1/k)D_\beta\}_{\beta \in \mathcal B}$ associated with the blocks form
an affine unbiased quantum state arrangement with purity $\gamma=1/k$. Thus, if $D_\beta$ and $D_{\beta'}$ belong to different parallel classes,
we get $\mu = \tr[D_\beta D_{\beta'}]/k^2=\nu/k^2 = 1/v$, $\nu = k^2/v$ and 
$\sigma=\gamma v$, hence
$v=k \sigma = \nu \sigma^2$
from the lemma.

If one element of $\{1, 2, \dots v\}$ is fixed, then each block containing this element allows to pair it with $k-1$ others from the same block.
Counting the number of blocks $r$ containing this fixed
element then gives $r(k-1) = \lambda(v-1)$, so
\begin{equation}
   \sum_{\beta \in \mathcal{B}} D_\beta = \frac{\lambda (v-1)}{k-1} I \, .
\end{equation}
Inserting the expressions for $v$ and $k$ then gives the claimed identity for $\rho$. %Summing the above identity
 %over $\beta$ and $\beta'$ gives
 %$$
 %     \sum_{\beta \in {\mathcal B}_i}   \sum_{\beta' \in {\mathcal B}_l}   \TR( D_\beta D_{\beta'} ) = \sigma^2 \mu = \TR( I) = v \, .
 %$$
 %Finally, summing over all pairs $\beta, \beta' \in \mathcal B$ with $\beta \ne \beta'$ gives
 %$$
 %   \sum_{\beta \ne \beta'} \TR( D_\beta D_{\beta'} ) = \rho (\rho-1) \sigma^2 \mu + \rho \sigma k
 %   = \sum_{\beta ,  \beta' \in \mathcal B } \TR( D_\beta D_{\beta'} ) - \rho \sigma k = \rho^2 v - \rho \sigma k
% $$
 %$$
 %  \rho(\mu - 1) = \sigma^2 \mu - 1
% $$
 %summing both sides of this equation over all $i$ gives after exchanging the trace and the summation
% $$
 %    \rho \sigma k =  (\rho-1)\mu \sigma + k = (\rho-1)k + k 
% $$
 %$\sum_{\beta \in \mathcal{B}_i} D_\beta$ over all $i$ and comparing with the expression for $\sum_{\beta \in \mathcal{B}} D_\beta$
% gives  $\rho = \frac{\lambda (v-1)}{k-1} = \lambda \frac{\sigma^2 \mu-1}{\sigma\mu-1}$.
\end{proof}

Using these relationships allows us to re-express the size of the intersection between two blocks and the parameter $\lambda$  in terms of $k$ and $v$.

\begin{cor}\label{cor:affint}
An affine $(v,k,\lambda)$-BIBD has a total of  $n = \frac{v(v-1)}{k(k-1)}\lambda $ blocks and any two blocks $\beta$ and $\beta'$
intersect in a set of size
$$
   | \beta \cap \beta'  | \in \{k, k^2/v, 0\} \, .
$$
\end{cor}

Standard examples of such affine block designs are based on affine geometries.
We will notice that the number of blocks is not large enough to put the associated diagonal projection matrices
in the regime where the orthoplex bound is optimal.

\begin{defn}\label{defn:affgeo}
Consider the Galois field $\BB F = GF(q)$, with $q$ a prime power, and $m \ge 2$. The affine geometry $AG(m,q)$
is the set of affine subspaces of ${\BB F}^m$.
\end{defn}

\begin{ex}[(Stinson~\cite{Stinson2004}, Theorem 5.18; \cite{Handbook}, Proposition VII.2.44)] \label{ex:affgeo2}
Let $q$ be a prime power, $m\ge 2$ and $\BB{F}^m$ the vector space of dimension $m$
over the Galois field $\BB F = GF(q)$. Let $\mathcal A$ be the set of 
hyperplanes in $\BB{F}^m$, i.e., the
affine subspaces of dimension $m-1$, then each $\beta \in \mathcal A$ has size
$q^{m-1}$ and enumerating the $v=q^m$ elements in the vector space gives that $\mathcal A$ 
 is an affine  $(q^m,q^{m-1},\lambda)$-BIBD with 
$ \lambda = (q^{m-1}-1)/(q-1)$ containing $\rho = (q^m-1)/(q-1)$ parallel classes
of $\sigma = q$ blocks each.
\end{ex}

In this case, the number of blocks and the size of the intersection are given in terms of the prime power $q$.

\begin{cor}When $m\ge 2$ and $q$ is a prime power,
the affine  $(q^m,q^{m-1},\lambda)$-BIBD
consists of $n=q(q^m-1)/(q-1)$ blocks for which any pair $\beta$ and $\beta'$ satisfy
$$
 | \beta \cap \beta'  | \in \{q^{m-1}, q^{m-2}, 0\} .
$$
In particular, if $\beta \ne \beta'$, then $| \beta \cap \beta'  | \le q^{m-2}$.
\end{cor}

%We associate with each block $\beta$, the diagonal projection matrix
%$$
 %  D_\beta = \sum_{j \in \beta} E_{j,j} \, .
%$$
The trace-normalized diagonal matrices belonging to an affine block design satisfy the orthoplex bound.
For $q\ge 2$,  $\frac{q}{q-1} (q^m-1) > q^m$, so according to Corollary~\ref{cor:diagonalsimplex}, they form an optimal packing in the simplex
of classical states. 

\begin{cor}\label{cor:affinedesigninnerproduct}
When $m\ge 2$ and $q$ is a prime power, an affine  $(d,k,\lambda)$-BIBD $\mathcal A$ having parameters $d=q^m$ and $k=q^{m-1}$
provides us with $n=q(q^m-1)/(q-1)$ associated
classical states  $\{(1/k)D_\beta\}_{\beta \in \mathcal A}$ 
that are orthoplex-bound achieving in the simplex of non-negative, trace-normalized diagonal matrices.
\end{cor}

The choice of $\mathcal A$ that achieves the orthoplex bound yields a block design whose packing properties
can be reformulated in the terminology of binary codes. If each $\beta \in \mathcal A$ is associated with
a binary sequence whose support is $\beta$, then these sequences form a constant-weight code
of weight $k=q^{m-1}$ that maximizes the Hamming distance, see \cite[Inequality (14)]{Johnson72} or \cite[Table 8.4]{Levenshtein1992}.
This implication has been pointed out in an earlier work \cite{BodmannHaas18} for the special case when $q$ is a power of $4$ in the real case or a power of $2$ in the complex
case.

%\section{Main results}

\section{Mutually unbiased bases and affine designs}\label{sec:MUB}
We now show how the orthoplex bound can be achieved by combining optimal line packings with 
affine block designs. The first construction builds on maximal sets of mutually unbiased bases. 

Zauner~\cite{Zauner1999,Zauner2011} describes a construction of mutually unbiased bases that is closest
to the structure of affine block designs. 

\begin{defn}\label{def:trace} Let $q= p^s$ with $p$ a prime, $s \in \mathbb N$. Consider $\BB{F} = GF(q)$ as a vector space over $\BB{Z}_p = GF(p)$, and let
for $a \in \BB F$, the trace be defined by
$$
  \TRQS{q/p}{a} = a + a^p + a^{p^2} + \cdots + a^{p^{s-1}} \, .
  $$
 Two bases $\{a_j\}_{j=1}^s$ and $ \{b_j\}_{j=1}^s$ are called dual if
 $$
   \TRQS{q/p}{a_j b_l} = \delta_{j, l} \, .
 $$
\end{defn} 

\begin{defn}
Let $m \geq 2$ be an integer and $\zeta_m \in \bC$ a primitive $m$th root of unity. We denote the $m \times m$ identity matrix by $I_m$. Furthermore, we define the \emph{(cyclic) translation} $T_m$,  \emph{modulation} $M_m$, and \emph{discrete Fourier transform (DFT)} $F_m$ operators as
\[
T_m = \left( \begin{array}{ccccc} 0 & 0  & \hdots & 0 & 1 \\ 1 & 0  & \hdots & 0& 0 \\ 0 & 1 & \hdots & 0& 0 \\  \vdots  & \vdots & \ddots & \vdots & \vdots \\ 0 & 0  & \hdots& 1 & 0 \end{array}\right), \quad M_m =\left( \begin{array}{cccc} 1 & 0 & \hdots & 0 \\ 0 & \zeta_m & \hdots & 0 \\ \vdots & \vdots & \ddots & \vdots \\ 0 & 0 & \hdots & \zeta_m^{m-1} \end{array}\right),  \quad \textrm{and} \quad F_m = \left( \zeta_m^{ij} \right)_{i,j=0}^{m-1}.
\]
%When $m$ is clear from context, we will omit the subscript. 
Further, if $m=(m_0, m_1, \hdots, m_s)$ is a vector of integers $\geq 2$, the group of translations over $\bigoplus_{\ell=0}^s \bZ_{m_\ell}$ is
\[
\left\{T_m^{(k)} = T_{m_0,\hdots, m_s}^{(k_0,\hdots,k_s)}:=  \bigotimes_{\ell=0}^s T_{m_\ell}^{k_\ell}: k=(k_0, \hdots k_s) \in \bigoplus_{\ell=0}^s \bZ_{m_\ell}\right\},
\]
where $\otimes$ is the Kronecker product.  Similarly, the group of modulations is
\[
\left\{M_m^{(\kappa)}=M_{m_0,\hdots, m_s}^{(\kappa_0,\hdots,\kappa_s)}:= \bigotimes_{\ell=0}^s M_{m_\ell}^{\kappa_\ell}: \kappa=(\kappa_0, \hdots \kappa_s) \in \bigoplus_{\ell=0}^s \bZ_{m_\ell}\right\}.
\]
The Fourier transform is $F_m=F_{m_0,\hdots,m_s}=\otimes_{\ell=0}^s F_{m_\ell}$.  We further define the \emph{quantum displacement operators} as $D_m^{(k, \kappa)} = (-\tau_m)^{-k \cdot \kappa} M_m^{(\kappa)} T_m^{(k)}$, where 
$\tau_m=\prod_\ell \tau_{m_\ell}$ and
each $\tau_{m_\ell}$ is a primitive $2m_\ell$th root of unity such that $\tau_{m_\ell}^2 = \zeta_{m_\ell}$. %\ejk{Should double check that the tau factor is correct for non-cyclic groups}

\end{defn}

%\ejk{add note about labeling of the rows here or later??}
%\ejk{do we really need the following sentence?? If so, fix for non-cyclic case} We note that the following identities also hold: 
%\[
%T_m^m = I_m, \enskip M_m^m = I_m, \enskip T_mM_m = \zeta_m^{-1} M_mT_m, \enskip \textrm{and} \enskip F_m^\ast = m F_m^{-1}.
%\]
%Further, when $m$ is odd, we will define $\tilde{M}$ to be the $\left((m-1)/2\right) \times \left((m-1)/2\right)$ principal submatrix of $M$; that is,
%\[
%\tilde{M} = \tilde{M}_m= \left( \begin{array}{cccc} 1 & 0 & \hdots & 0 \\ 0 & \zeta & \hdots & 0 \\ \vdots & \vdots & \ddots & \vdots \\ 0 & 0 & \hdots & \zeta^{(m-3)/2} \end{array}\right).
%\]
We will always order elements of $\oplus_{\ell = 0}^s \bZ_{m_\ell}$ lexicographically. We also denote the all-ones-vector by $\mathbbm{1}$, with the dimension being clear from context.

\begin{defn}
For two dual bases $\{a_l\}_{l=1}^s$ and $ \{b_l\}_{l=1}^s$ of $\BB F=GF(q)$ and $c,d \in \BB F$
with $c=\sum_{j=1}^s c_j a_j $ and $d=\sum_{j=1}^s d_j b_j$, we define the Weyl displacement operator
as $D_{m}^{(k,\kappa)}$ as above, with $m=(p, p, \dots, p)$, acting 
on the Hilbert space $({\BB C}^{p})^{\otimes s}$.
\end{defn}

\begin{ex}\label{ex:weyl}
%\ejk{How to best incorporate this with (temporarily put into separate sections) Section~\ref{sec:BinderMUB} about creating MUBs from binders and Section~\ref{sec:dispGroup} about BIBD binders corresponding to subgroups of displacement operators?}
Let $\{\mathcal V_j\}_{j=0}^q$ be the collection of one-dimensional subspaces of $\BB{F}^2$,
then it can be verified that for each $j$, the set $\{D_m^{(c,d)}: (c,d) \in \mathcal V_j\}$ is pairwise commuting
and contains the identity $D_m^{(0,0)}$. By simultaneous diagonalization, for each $j$ there are linear combinations
of these operators that give $q$ mutually orthogonal rank-one projection operators $\{P_l\up j\}_{l=0}^{q-1}$.
Using the orthogonality among the Weyl operators, $\TR{(D^{(c,d)}(D^{(c',d')})^*))}=q \delta_{c,c'} \delta_{d,d'}$,
if two rank-one orthogonal projections $P$ and $P'$  are associated to two different subspaces $\mathcal{V}_j$
and $\mathcal{V}_{j'}$  of $\BB{F}^2$, $j \ne j'$, then
$$
  \TR{(P - \frac 1 q I)(P'- \frac 1 q I)} = 0 \, ,
$$
and hence 
$$
  \TR{PP'} = \frac 1 q \, .
$$
This means that the $q+1$ one-dimensional subspaces of $\BB{F}^2$ provide
mutually unbiased bases containing a total of $q(q+1)$ vectors in $\BB{C}^q$.
\end{ex}

\begin{defn}
For the real case, define $X = \left(\begin{array}{cc} 0 & 1 \\ 1 & 0 \end{array}\right)$,  
$Z= \left(\begin{array}{cc} 1 & 0 \\ 0 & -1 \end{array}\right)$, 
$T\up 0 = I \otimes I$, $T\up 1=X \otimes I$, $T\up 2 = I \otimes X$, $T\up 3=X \otimes X$;
$M\up 0 = I \otimes I$,
$M\up 1 = Z\otimes I$, $M\up 2 = I \otimes Z$, and $M\up 3= Z \otimes Z$.
Let $q=4^s$ and $\BB{F}= GF(q)$.
As before, we define displacement operators $\widetilde{D}_m^{(c,d)}$ in terms of tensor products of $T\up j$ and $M\up j$, $j \in \{0, 1, 2, 3\}$. 
\end{defn}

\begin{ex} Let $q=2^{2s}$ and $\BB{F}= GF(q)$.
Calderbank et al.~\cite{CalderbankCameronKantorSeidel1997}
find $q/2+1$ one-dimensional subspaces $\{\mathcal V_j\}_{j=0}^{q/2}$ of $\BB{F}^2$ for which
 $\{\widetilde{D}_m^{(c,d)}: (c,d) \in \mathcal V_j\}$ consists of  pairwise commuting symmetric matrices.
 Again by simultaneous diagonalization and Hilbert-Schmidt orthogonality,
 each subspace $\mathcal{V}_j$ produces $q$ mutually orthogonal rank-one projection operators $\{P_l\up j\}_{l=0}^{q-1}$
 and if
 two rank-one orthogonal projections $P$ and $P'$  are associated to two different subspaces $\mathcal{V}_j$
and $\mathcal{V}_{j'}$  of $\BB{F}^2$, $j \ne j'$, then
%$$
%  \TR{(P - \frac 1 q I)(P'- \frac 1 q I)} = 0 \, ,
%$$
%and hence 
$
  \TR{PP'} = \frac 1 q \, .
$
Consequently, we obtain mutually unbiased bases with a total of 
 $q(q/2+1)$ vectors in $\BB{R}^q$.
\end{ex}

Next, we group the one-dimensional subspaces obtained with the mutually unbiased bases 
with the help of affine block designs in order to produce optimal quantum state arrangements.

\begin{defn}
Given a subset of an orthonormal basis $\mathcal{B}=\{b_i\}_{i=1}^d$ indexed by $\beta \subset \{1, 2, \dots, d\}$, the orthogonal projection onto the span of $\{b_i\}_{i \in \beta}$ is  called {\em coordinate projection $P_\beta$.}
\end{defn}

%\begin{defn}
%Given an orthonormal basis $\CAL{B} = \{b_j\}_{j=1}^d$ for $\BB{F}^d$ and a subset $ \beta \subset \{1,2,\dots, d\}$, the {\it $\beta$-coordinate projection} with respect to $\CAL{B}$ is
%$$P_{\beta} =\sum\limits_{j \in \beta} b_j \otimes b_j^*.$$
%\end{defn}

\begin{prop}\label{prop_MUBs_give_MUFFs}
If $\CAL{B} = \{b_j\}_{j=1}^d$ and ${\CAL{B}'}=\{{b'}_j\}_{j=1}^{d}$ are a pair of mutually unbiased bases for $\BB{F}^d$ and $\beta, \beta' \subset \{1, 2, \dots, d\}$ with $k=|\beta|=|\beta'|$, then
$
\TR(P_{\beta} {P'_{\beta'}}) = \frac{k^2}{d},$
where
$P_{\beta}$ is the $\beta$-coordinate projection with respect to $\CAL{B}$ and  ${P'}_{\beta'}$ is $\beta'$-coordinate projection with respect to ${\CAL{B}'}.$  
\end{prop}

\begin{proof}
We have
$$
\TR(P_{\beta} {P'}_{\beta'}) 
= %\sum\limits_{j \in \beta, j' \in \beta'} \TR
%{(
%b_j \otimes {b_j}^* \, {b'}_{j'} \otimes { {b'}_{j'}}^* 
%)}
%=  
\sum\limits_{j \in \beta, j' \in \beta'} \left|\left\langle b_j, {b'}_{j'} \right\rangle \right|^2= \frac{k^2}{d}
.
$$
\end{proof}

Combining the preceding proposition with Corollary~\ref{cor:affinedesigninnerproduct} yields the following bound
for the Hilbert-Schmidt inner products between coordinate projections.

\begin{cor} \label{cor:opb}
Let $m\ge 2$ and $q$ be a prime power, and $\mathcal A$ be  
an affine  $(d,k,\lambda)$-BIBD with $d=q^m$, $k = q^{m-1}$ and $\lambda = (q^{m-1}-1)/(q-1)$. Let $\{\mathcal{B}_{j}\}_{j=1}^r$ be a set of mutually unbiased bases
and for $\beta\in \mathcal A$, $P\up j_\beta$ the $\beta$-coordinate projection corresponding to basis $ \mathcal{B}_j$, then
for any pair $(j,\beta)\ne (j',\beta')$, we have
$$
    \TR(P\up j_\beta P\up{j'}_{\beta'} ) \le k^2/d \, .
$$
\end{cor}

Counting the number of blocks shows that the orthoplex bound is saturated, so the projections give an optimal quantum state arrangement.
This holds in the real as in the complex case when the dimension is appropriately restricted.

\begin{thm}\label{thm:MUBpack}
Let $\BB{F} = \BB{R}$, $m\ge 2$ and $q$ be a power of $4$. % if $\BB F = \BB C$ or a power of $4$ if $\BB F = \BB R$. 
Let $\mathcal A$ be  
an affine  $(d,k,\lambda)$-BIBD with $d=q^m$, $k = q^{m-1}$ and $\lambda = (q^{m-1}-1)/(q-1)$ and $\{\mathcal{B}_{j}\}_{j=1}^r$
be a set of $r=d/2+1$ mutually unbiased bases %if $\BB F = \BB C$ and
%otherwise with $r=d/2+1$ 
and for $\beta\in \mathcal A$, $P\up j_\beta$ the $\beta$-coordinate projection corresponding to basis $ \mathcal{B}_j$, then
$\mathcal W = \{(1/k)P_{\beta}\up j : j \in \{1, 2, \dots, d/2+1\}, \beta \in \mathcal A\}$ is an optimal arrangement
of $n=\frac{(d-1)d(d/2+1)}{d-k}$ quantum states with purity $\gamma=q^{1-m}$.
\end{thm}
\begin{proof}
If $\mathbb F =\mathbb R$, $d=q^m$ with $q=4^i$ and $m\ge 2$, and $r=d/2+1$,  then we have
$n$ subspaces whose associated projections observe the bound in Corollary~\ref{cor:opb}
with $n=r \rho\sigma = (q^m/2+1) \frac{q(q^m-1)}{q-1} = \frac{(d/2+1) d (d-1)}{d-k} )\ge (d/2+1)d>d(d+1)/2$,
so the resulting states are orthoplex-bound achieving.
\end{proof}

\begin{thm}
Let $\BB{F} = \BB{C}$, $m\ge 2$ and $q$ be a prime power. % if $\BB F = \BB C$ or a power of $4$ if $\BB F = \BB R$. 
Let $\mathcal A$ be  
an affine  $(d,k,\lambda)$-BIBD with $d=q^m$, $k = q^{m-1}$ and $\lambda = (q^{m-1}-1)/(q-1)$ and $\{\mathcal{B}_{j}\}_{j=1}^r$
be a set of $r=d+1$ mutually unbiased bases %if $\BB F = \BB C$ and
%otherwise with $r=d/2+1$ 
and for $\beta\in \mathcal A$, $P\up j_\beta$ the $\beta$-coordinate projection corresponding to basis $ \mathcal{B}_j$, then
$\mathcal P = \{(1/k)P_{\beta}\up j : j \in \{1, 2, \dots, d+1\}, \beta \in \mathcal A\}$ is an optimal arrangement
of $n=\frac{(d-1)d(d+1)}{d-k}$ quantum states with purity $\gamma=q^{1-m}$.
\end{thm}
\begin{proof}
If $\mathbb F =\mathbb C$ and $r=d+1$, with $d=q^m$, $m\ge 2$, then we have
$n$ subspaces with
$n=r \rho\sigma = (q^m+1) \frac{q(q^m-1)}{q-1} = \frac{(d-1)d(d+1)}{d-k} \ge (d+1)d$,
so the resulting states are orthoplex-bound achieving.
\end{proof}

\section{Gabor-Steiner equiangular tight frames}\label{sec:gabor}

The preceding section combined optimal line packings, associated with maximal sets of mutually unbiased bases, 
with combinatorial designs in order
to construct quantum states that form optimal arrangements. Hereafter, we show that a similar type of result can be obtained starting
from equiangular line packings. The line packings used for this purpose are given by orbits under the action of the Heisenberg-Weyl group.
To prepare this result, we first fix some notation.

For a ring with unity $R$, $\operatorname{Mat}_n(R)$ denotes the $n \times n$ matrices with entries in $R$, $\SYM_n(R)$ the symmetric matrices with entries in $R$. %For $a_0, a_1, a_2 \in R$, we define $\Sym(a_0,a_1,a_2) \in \operatorname{Mat}_2(R)$ to be the symmetric matrix
%\[
%\Sym(a_0,a_1,a_2) = \left( \begin{array}{cc} a_0 & a_1 \\ a_1 & a_2 \end{array} \right). 
%\]%$GL(R,n)$ denotes the invertible $n\times n$ matrices with entries in $R$.

\begin{defn}
Let $G=\bigoplus_{\ell=0}^s \bZ_{m_\ell}$ be an arbitrary finite abelian group. 
%\ejk{OPTION 1}: 
%The \emph{finite Gabor-Weyl-Heisenberg group over $\bigoplus_{\ell=0}^s \bZ_{m_\ell}$} is the semidirect product $G \rtimes (G \times G)$ equipped with the operation
%\[
%(\lambda, k, \kappa) \cdot (\tilde{\lambda}, \tilde{k}, \tilde{\kappa}) = (\lambda + \tilde{\lambda} - k \tilde{\kappa}, k+\tilde{k}, \kappa + \tilde{\kappa}),
%\]
%where $k \tilde{\kappa} \in G$ is the component-wise product of the elements $G$. Set $\abs{m}=\prod_{\ell=0}^s m_\ell$, and for each $\ell \in \{ 0, \hdots, s\}$, fix a primitive $m_\ell$ root of unity $\zeta_{m_\ell}$. A standard unitary representation of the finite Gabor-Weyl-Heisenberg group is
%\[
%(\lambda, k, \kappa) \mapsto \left(\prod_{\ell=0}^s \zeta_{m_\ell}^{\lambda_\ell}\right)M_m^{(\kappa)}T_m^{(k)}  \in U(\bC,\abs{m}).
%\]
%\ejk{OPTION 2}: 
The \emph{(finite) Weyl-Heisenberg group over $\bigoplus_{\ell=0}^s \bZ_{m_\ell}$} is the semidirect product $G \rtimes (G \times G)$ equipped with the operation
\[
(\lambda, k, \kappa) \cdot (\tilde{\lambda}, \tilde{k}, \tilde{\kappa}) = (\lambda + \tilde{\lambda} + \tilde{k}\kappa - k \tilde{\kappa}, k+\tilde{k}, \kappa + \tilde{\kappa}),
\]
where $k \tilde{\kappa} \in G$ is the sum of the component-wise product of the elements in $G$. Set $\abs{m}=\prod_{\ell=0}^s m_\ell$, and for each $\ell \in \{ 0, \hdots, s\}$, fix a primitive $2m_\ell$th root of unity $\tau_{m_\ell}$ where $\zeta_{m_\ell} = \tau_{m_\ell}^2$ is the primitive $m_\ell$th root of unity used to define the modulation operator $M_{m_\ell}$. A standard unitary representation of the finite Weyl-Heisenberg group is
\[
(\lambda, k, \kappa) \mapsto \left(\prod_{\ell=0}^s (-\tau_{m_\ell})^{\lambda_\ell}\right) D_m^{(k,\kappa)} = \left(\prod_{\ell=0}^s (-\tau_{m_\ell})^{\lambda_\ell - k \kappa}\right)M_m^{(\kappa)}T_m^{(k)}  \in U(\abs{m})\, ,
\]
where $U(\abs{m})$ is the set of all $\abs{m}\times\abs{m}$ unitaries.
%\ejk{if using this, then need to make notation match throughout for the displacement operators\\}
%\ejk{end option}

We define $\sigma: G \times G \rightarrow U(\abs{m})$ by $\sigma(k,\kappa)= M_{m}^{(\kappa)} T^{(k)}_{m}$.  If $\abs{m}$ is odd, we further define the mapping $\pi: G \times G \rightarrow U(\abs{m}(\abs{m}-1)/2)$ as
\[
\pi(k, \kappa) = I_{(\abs{m}-1)/2} \otimes \left( M_{m}^{(\kappa)}T^{(k)}_{m}\right).
\] 
\end{defn}
The mapping $\pi$ may be viewed as a lifting of $\sigma$. We note that $\pi$ is dependent upon the primitive $m_\ell$th roots of unity used to define the $M_m^{(\kappa)}$.  It is easy to see that $\sigma$ and $\pi$ are faithful projective representations of $G \times G$, that is, they are injective and the product of two image points is, up to constant, the image of the product. A frame formed as the orbit of a vector under $\sigma(G \times G)$ is called a \emph{(finite) Gabor frame}. (We note that \emph{Gabor-Weyl-Heisenberg frame} would be a more appropriate name, but for simplicity we will use the name which is standard in frame theory.) We may now rephrase Zauner's conjecture (Conjecture~\ref{conj:Zau}): For all $m \in \bN$ there exists a Gabor equiangular tight frame of $m^2$ vectors in $\bC^m$.

%An important conjecture in quantum information theory concerning Gabor frames is known as {\it Zauner's conjecture}.
%\begin{conj}[\cite{Zauner1999,Zauner2011}]
%For all $m \in \bN$ there exists an equiangular tight frame of $m^2$ vectors in $\bC^m$.
%\end{conj}
Such an equiangular tight frame is maximal with respect to the number of vectors by Theorem~\ref{thm:rankin} and is also called a \emph{symmetric informationally complete positive operator-valued measure} (\emph{SIC}).  Up to equivalence, each known SIC is a Gabor frame over $\bZ_m$ for some $m \in \bN$ or over $\bZ_2 \oplus \bZ_2 \oplus \bZ_2$ \cite{Hog98,FHS17}. It is known that if $p$ is a prime, then any SIC in $\bC^p$ generated by a (projective) unitary representation is equivalent to one formed as the orbit of the Weyl-Heisenberg group \cite{Zhu10}.  
In addition, SICs obtained from the orbit of a vector in $\mathbb{C}^{d}$ under the irreducible representation of the Heisenberg-Weyl group 
form a maximal orbit in the sense that the group action leads to an orbit of $d^2$ states, which is equal to the size of
the projective equivalence classes of the group representation, when identifying unitaries that differ by a unimodular factor.

The linear dependencies of vectors in equiangular tight frames are of particular interest in compressed sensing and other sparsity-based methods, since a linear dependence of $v$ vectors in an equiangular tight frame $\Phi$ corresponds to a non-zero vector lying in the kernel of $\Phi$ with at most $v$ non-zero entries.  Furthermore, detecting the non-trival linear dependencies of group-generated equiangular tight frames yields information about additional symmetries possessed by the equiangular tight frame.  The linear dependences of frame vectors are encoded in their associated matroid structure, and one may leverage results about matroids to (re)prove results in frame theory \cite{King15,CahMS}. There is also the surprising result that vectors that are optimally geometrically spread apart like those which compose an equiangular tight frame often have non-trivial linear dependences \cite{King15,FJKM17}. We have the following relationship between geometric and algebraic spread of vectors.

\begin{thm}[\cite{DE03,BDE09}]\label{thm:spark}
Let $\Phi = \{ \varphi_j\}_{j = 1}^n$ be a collection of vectors in $\bF^d$ which are not all pairwise orthogonal. Define the {\it spark} to be the size of the smallest linearly dependent subset of $\Phi$.  Then
\beq\label{eqn:spark}
\spark \Phi \geq 1 + 1/\max_{j,l\leq n, j \neq l} \absip{\varphi_j}{\varphi_l}.
\eeq
\end{thm}
Spark is known in matroid theory as \emph{girth}.  If $\spark (\Phi) = d+1$, then $\Phi$ has no non-trivial linear dependence relationships between the vectors and is called \emph{full spark}.  In matroid terminology, a full spark frame corresponds to the uniform matroid of the correct parameters. Calculating the spark of a matrix is in general NP Hard \cite{TiPf14}, although there are special cases which can be solved using exact (mixed-)integer programming models and linear programming heuristics \cite{Til18}.  In what follows, we will calculate the spark of an infinite class of frames and for an infinite subclass of these frames, determine all of the subsets of $\spark \Phi$ vectors which are linearly dependent. Most work about linear dependencies in Gabor frames focuses on full spark conditions in general Gabor frames (not necessarily SICs) \cite{LPW05,KPR08,Mal15,Mal17}.  However, \cite{ABDF17,DBBA2013} (inspired in part by \cite{Hugh07}) deal with finding special subsets of vectors in certain SICs which form equiangular tight frames for their span. In particular, during the talk \cite{Hugh07}, it is pointed out that when one chooses an appropriate $\psi \in \bC^3$, one gets a SIC with linear dependencies corresponding to the Hesse configuration, which is $AG(2,3)$ (see Example~\ref{ex:hesse}). We will see in Theorem~\ref{thm:AG} that for all odd primes $p$, there is a $\psi$ such that the orbit of $\psi$ under $\pi(\bZ_p \times \bZ_p)$ is an equiangular tight frame $\Phi$ for which the minimal linearly dependent sets (i.e., those of size $\spark(\Phi)$) correspond to $AG(2,p)$.

\begin{defn}[\cite{GKK01,ChWa16,FJKM17}] \label{defn:simpl}
Let $\Phi = \{ \varphi_j \}_{j \in J}$, $\Psi = \{ \psi_j\}_{j \in J}$ be two sequences of vectors. If there exists a unitary matrix $U$, a permutation matrix $P$, and a unimodular diagonal matrix $D$ such that $\Psi = U \Phi P D$, then we say $\Phi$ and $\Psi$ are {\it switching equivalent}.  Let $\beta \subseteq J$.  If $\Phi_\beta = \{ \varphi_i\}_{j \in \beta}$ is switching equivalent to an equiangular tight frame with span of dimension $\abs{\beta}-1$ which has only real, negative inner products, then we call $\Phi_{\beta}$ a {\it simplex}.  The {\it (simplex) binder} $\cB =\cB(\Phi)$ is the set of simplices of an equiangular tight frame, or equivalently, their corresponding index sets, namely
\[
\cB = \cB(\Phi) = \left\{ \beta : \textrm{$\Phi_\beta \subset \Phi$ is a simplex} \right\}.
\]
\end{defn}
When $\Phi$ and $\Psi$ are switching equivalent they represent -- up to permutation, change of basis, and choice of spanning vector in each line -- the same collection of lines.  We note that some authors do not allow permutations in their definition of switching equivalence.

\begin{ex}[\cite{Hugh07,DBBA2013}] \label{ex:hesse}
We consider the linear dependencies of the orbits of two different vectors, namely $\psi=(\begin{array}{ccc} 1& 0& 1\end{array})^\top$ and $\tilde{\psi} =(\begin{array}{ccc}1 &0& -1\end{array})^\top$ under $\pi(\bZ_3 \times \bZ_3) =\sigma(\bZ_3 \times \bZ_3)$.  Fix $\zeta$ to be a primitive $3$rd root of unity. We note that by Definition~\ref{def:ETF}
\begin{align*}
\Phi_1 &= \left(\begin{array}{ccc|ccc|ccc} \! M_3^0 T_3^0 \psi  \! &  \! M_3^1 T_3^0 \psi  \! & \!  M_3^2 T_3^0 \psi \!  &  \! M_3^0 T_3^1 \psi  \! &  \!  M_3^1 T_3^1 \psi  \! & \!   M_3^2 T_3^1 \psi \!  &  \!  M_3^0 T_3^2 \psi \!  &  \! M_3^1 T_3^2 \psi \!  & \!  M_3^2 T_3^2 \psi \! \end{array}\right)\\
&= \left( \begin{array}{ccc|ccc|ccc} 1 & 1 & 1 & 1 & 1 & 1 & 0 & 0 & 0\\ 0 & 0 & 0 & 1 & \zeta & \zeta^2 &  1 & \zeta & \zeta^2 \\  1 & \zeta^2 & \zeta & 0 & 0 & 0 &  1 & \zeta^2 & \zeta\end{array} \right) \enskip \textrm{and}\\
\Phi_2 &= \left(\begin{array}{ccc|ccc|ccc}  \! M_3^0 T_3^0 \tilde{\psi}  \! &  \! M_3^1 T_3^0 \tilde{\psi} \!  &  \! M_3^2 T_3^0\tilde{\psi}  \! & \!  M_3^0 T_3^1 \tilde{\psi} \! &  \!  M_3^1 T_3^1 \tilde{\psi} \!  &   \! M_3^2 T_3^1 \tilde{\psi} \!  &  \!  M_3^0 T_3^2\tilde{\psi} \!  &  \! M_3^1 T_3^2\tilde{\psi} \!  & \!  M_3^2 T_3^2\tilde{\psi} \! \end{array}\right)\\
&= \left( \begin{array}{ccc|ccc|ccc} 1 & 1 & 1 & -1 & -1 & -1 & 0 & 0 & 0\\ 0 & 0 & 0 & 1 & \zeta & \zeta^2 &  -1 & -\zeta & -\zeta^2 \\  -1 & -\zeta^2 & -\zeta & 0 & 0 & 0 &  1 & \zeta^2 & \zeta\end{array} \right) 
\end{align*}
are both equiangular tight frames since the length of each vector is $\sqrt{2}$ and for $\varphi, \tilde{\varphi} \in \Phi_i$ with $\varphi \neq \tilde{\varphi}$ and $i \in \{1,2\}$,
\[
\tr\left(\left( (1/(\sqrt{2})^2\varphi) \otimes \varphi^\ast\right)\left(  \tilde{\varphi}^\ast   \otimes (1/(\sqrt{2})^2\tilde{\varphi})\right)\right) = \frac{1}{4} \absip{\varphi}{\tilde{\varphi}}^2 = \frac{1}{4} =\frac{9-3}{(9-1)3}.
\]
Since both $\Phi_1$ and $\Phi_2$ are equiangular tight frames of $9$ vectors spanning a $3$ dimensional space, they are SICs.  By performing an exhaustive search \cite{FJKM17}, one can see that the incidence matrices of the binders of $\Phi_1$ and $\Phi_2$ are %\begin{align*}
%\cB(\Phi_1) &\sim \left( \begin{array}{ccccccccc}1 & 1 & 1 & 0 & 0 & 0 & 0 & 0 & 0  \\  0 & 0 & 0  &1 & 1 & 1 & 0 & 0 & 0  \\0 & 0 & 0  & 0 & 0 & 0 & 1 & 1 & 1 
%\end{array} \right) \enskip \textrm{and}\\
%\cB(\Phi_2) &\sim \left( \begin{array}{ccccccccc}1 & 1 & 1 & 0 & 0 & 0 & 0 & 0 & 0  \\  0 & 0 & 0  &1 & 1 & 1 & 0 & 0 & 0  \\0 & 0 & 0  & 0 & 0 & 0 & 1 & 1 & 1 \\\hline 1 & 0 & 0 &  1 & 0 & 0 & 1 & 0 & 0 \\ 0 &1 & 0 & 0 &  1 & 0 & 0 & 1 & 0  \\ 0 & 0 &1 & 0 & 0 &  1 & 0 & 0 & 1   \\ \hline
%1 & 0 & 0 & 0 & 1 & 0 & 0 & 0 & 1\\ 0 & 1 & 0 & 0 & 0 & 1 &  1 & 0 & 0\\ 0 & 0 & 1 & 1 &  0 & 0 & 0 & 1  & 0\\ \hline
%1 & 0 & 0 & 0 & 0 & 1 & 0 & 1 & 0\\ 0 & 1 & 0 & 1 & 0 & 0 & 0 & 0 & 1\\ 0 & 0 & 1 & 0 &1  & 0 & 1 & 0 & 0
%\end{array} \right)
%\end{align*}
\[
\cB(\Phi_1)\! \sim \!\left( \begin{array}{ccccccccc}1 & 1 & 1 & 0 & 0 & 0 & 0 & 0 & 0  \\  0 & 0 & 0  &1 & 1 & 1 & 0 & 0 & 0  \\0 & 0 & 0  & 0 & 0 & 0 & 1 & 1 & 1 
\end{array} \right)  \, \textrm{and} \,\,\,
\cB(\Phi_2) \!\sim \!\left( \begin{array}{ccccccccc}1 & 1 & 1 & 0 & 0 & 0 & 0 & 0 & 0  \\  0 & 0 & 0  &1 & 1 & 1 & 0 & 0 & 0  \\0 & 0 & 0  & 0 & 0 & 0 & 1 & 1 & 1 \\\hline 1 & 0 & 0 &  1 & 0 & 0 & 1 & 0 & 0 \\ 0 &1 & 0 & 0 &  1 & 0 & 0 & 1 & 0  \\ 0 & 0 &1 & 0 & 0 &  1 & 0 & 0 & 1   \\ \hline
1 & 0 & 0 & 0 & 1 & 0 & 0 & 0 & 1\\ 0 & 1 & 0 & 0 & 0 & 1 &  1 & 0 & 0\\ 0 & 0 & 1 & 1 &  0 & 0 & 0 & 1  & 0\\ \hline
1 & 0 & 0 & 0 & 0 & 1 & 0 & 1 & 0\\ 0 & 1 & 0 & 1 & 0 & 0 & 0 & 0 & 1\\ 0 & 0 & 1 & 0 &1  & 0 & 1 & 0 & 0
\end{array} \right).
\]
Each vector in $\Phi_1$ or $\Phi_2$ is parameterized by a unique $(k,\kappa) \in \bZ_3 \times \bZ_3$ corresponding to $M_3^\kappa T_3^k$, which we may view as points in the finite field plane $(GF(3))^2$.  The embedded simplices in $\Phi_1$ correspond precisely to the three vertical affine lines $\{(k, \kappa):\kappa \in GF(3)\}$. All other subsets of $3$ vectors in $\Phi_1$ are linearly independent. There are $12$ simplices in $\Phi_2$. These correspond to the vertical affine lines in $(GF(3))^2$ plus the affine lines with slopes $\{0, 1, 2\}$.  Following Definition~\ref{defn:affgeo}, this is precisely $AG(2,3)$. Up to equivalence, all of the SICs in $\bC^3$ belong to a parameterized family \cite{DBBA2013}.  With the exception of a finite set of points in the parameterization which have the same linear dependencies as $\Phi_2$, these have the same linear dependencies as $\Phi_1$.  $AG(2,3)$ is also known as the \emph{Hesse configuration} and thus $\Phi_2$ is also referred to as the \emph{Hesse SIC}.
\end{ex}

Another equiangular tight frame construction which we will be concerned with appears in \cite{FMT12} and consists of a tensor-like product of the adjacency matrix of a certain class of balanced incomplete block designs with particularly nice simplices. We first comment on a simple construction of simplices in arbitrary dimension. For any $m = (m_0, \hdots, m_s)$, a vector of integers $\geq 2$, the inner product of any two columns of $F_m$ --  which has all unimodular entries -- is zero.  Hence, if we set $\Phi$ to be any $(m-1) \times m$ submatrix of $F_m$ with the rows possibly multiplied by unimodulars, then the columns have norm $\sqrt{m-1}$ and the distinct columns have inner product $-1$, where $(-1)^2/(\sqrt{m-1})^4 = (m-(m-1))/((m-1)(m-1))$ (the first Rankin bound, see Theorem~\ref{thm:rankin}) and thus form a simplex by Definition~\ref{defn:simpl}.

\begin{thm}[\cite{FMT12}] \label{thm:SteinerETF}
Every $(v,k,1)$-BIBD generates an equiangular tight frame $\Phi$ called a \emph{Steiner equiangular tight frame} consisting of $n=v(k+v-2)/(k-1)$ vectors in $d = v(v-1)/[k(k-1)]$-dimensional space via a particular construction. Specifically, a special case of the construction is as follows:
\begin{enumerate}
\item Let $A^\top$ be the $\frac{v(v-1)}{k(k-1)} \times v$ transpose of the \ejk{incidence} matrix of a $(v,k,1)$-BIBD.
\item Fix $m=(m_0,m_1, \hdots, m_s)$ with $\abs{m} = \frac{k+v-2}{k-1}$.  For each $j \in \{1, \hdots, v\}$, define $\tilde{F}_j$ to be a $\left( \frac{k+v-2}{k-1}-1\right) \times \frac{k+v-2}{k-1}$ matrix which results from removing one row from $F_m$ and multiplying the rows with unimodulars. 
\item For each $j \in \{1, \hdots, v\}$, let $\Phi_j$ be the $\frac{v(v-1)}{k(k-1)} \times  \frac{k+v-2}{k-1}$ matrix obtained from the $j$th column of $A^\top$ by replacing each of the one-valued entries with a distinct row of $\tilde{F}_j$ and every zero-valued entry with a row of zeros.
\item Concatenate the $\Phi_j$'s horizontally to form $\Phi = \left( \Phi_1 \hdots \Phi_v\right)$.
\end{enumerate}
\end{thm}
\begin{proof}
The basic idea of the proof is as follows.  We note that $F_m$ is, up to scaling by $1/\sqrt{\abs{m}}$, a unitary matrix with all entries of modulus $1$.  Thus, if the columns of $\tilde{F}_j$ are denoted by $\{f_\kappa\}_{\kappa \in \bigoplus_{\ell = 0}^s \bZ_{m_\ell}}$, $\absip{f_\kappa}{f_{\widehat{\kappa}}}=1$ for all $\kappa \neq \widehat{\kappa}$.  Furthermore, the rows of $\tilde{F}_j$ have norm $\sqrt{\abs{m}}$ and are orthogonal, thus $\sum_{\kappa \in \bigoplus_{\ell = 0}^s \bZ_{m_\ell}} f_\kappa \otimes f_\kappa^\ast$ is $\abs{m}$ times the identity. The columns of the corresponding $\Phi_j$ also have unimodular inner product and the outer products of the columns sum to $\abs{m}$ times the diagonal projection onto the $j$th block of the balanced incomplete block design.  Since $A$ is an \ejk{incidence} matrix of a balanced incomplete block design, any two vectors from $\Phi_j$ and $\Phi_{\hat{j}}$ with $j \neq \hat{j}$ only share one common entry in support and thus have a unimodular inner product.  Further, due to a similar argument as in the proof of Proposition~\ref{prop:pcsize}, one can see that the sum of the outer products of all of the vectors in $\Phi$ is $\abs{m} (v-1)/(k-1)$ times the identity.  Hence by Definition~\ref{def:ETF}, $\Phi$ is a an equiangular tight frame.
\end{proof}
A  $(v,k,1)$-BIBD is also called a \emph{$(2,k,v)$-Steiner system}, hence the name of the construction.

\begin{ex}
Both $\Phi_1$ and $\Phi_2$ in Example~\ref{ex:hesse} can be generated as Steiner equiangular tight frames.  Consider the $(3,2,1)$-BIBD of all subsets of size $2$ of a set of size $3$.  One \ejk{incidence} matrix $A$ of this balanced incomplete block design is
\[
A = \left( \begin{array}{ccc} 1 & 0 & 1 \\ 1 & 1 & 0 \\ 0 & 1 & 1 \end{array} \right) \enskip \Rightarrow \enskip A^\top =  \left( \begin{array}{ccc} 1 & 1&  0  \\ 0 & 1 & 1  \\ 1& 0 & 1  \end{array} \right). 
\]
We fix a primitive $3$rd root of unity $\zeta$, label the rows of the corresponding $3 \times 3$ Fourier matrix $F_3$ as $\{ w_0, w_1, w_2\}$:
\[
F_3 = \left( \begin{array}{c} w_0 \\ w_1 \\ w_2 \end{array} \right) = \left(\begin{array}{ccc} 1 & 1& 1 \\ 1 & \zeta & \zeta^2 \\ 1 & \zeta^2 & \zeta \end{array} \right),
\]
and define $z = (\begin{array}{ccc} 0 & 0 & 0\end{array})$. Then 
\[
\Phi_1 = \left( \begin{array}{ccc} w_0 & w_0 & z \\ z & w_1 & w_1 \\ w_2 & z & w_2\end{array}\right) \enskip \textrm{and} \enskip \Phi_2 = \left( \begin{array}{ccc} w_0 & -w_0 & z \\ z & w_1 & -w_1 \\ -w_2 & z & w_2\end{array}\right), 
\]
as desired. \ejk{Instead of thinking about signing the rows of $F_3$, as in the construction of Steiner equiangular tight frames, one may consider this construction to arise in part from signing the elements of the incidence matrix in order to induce linear dependencies.  Following Example 3.2 in~\cite{FJMPW17}, we see that while $A$ is full rank, 
\[
\tilde{A} = \left( \begin{array}{ccc} 1 & 0 & -1 \\ -1 & 1 & 0 \\ 0 & -1 & 1 \end{array} \right) 
\]
is not.  In~\cite{FJMPW17}, the authors phase incidence matrices (without additionally expanding them with rows of a unitary matrix) to construct equiangular tight frames; however, their approach is different since the constructed vectors are only frames for their spans.} 
\end{ex}

%For any $m = (m_0, \hdots, m_s)$ a vector of integers $\geq 2$ and any $\kappa \in \bigoplus_{\ell = 0}^s \bZ_{m_\ell}$, the inner product of any two columns of $M_m^{(\kappa)}$ is zero.  Hence, if we set $\Phi$ to be any $(m-1) \times m$ submatrix of $M_m^{(\kappa)}$ with the rows multiplied by unimodulars, then the distinct columns have $-1$, where $(-1)^2 = (m-(m-1))/((m-1)(m-1))$ (the first Rankin bound), as an inner product and thus form a simplex.

We now generalize this example and construct an infinite family of equiangular tight frames which have both Gabor symmetry, i.e. are Weyl-Heisenberg invariant, and can be generated as Steiner equiangular tight frames.

\begin{defn}\label{defn:psi}
Let $m=(m_0, \hdots, m_s)$ be a vector of odd integers $\geq 3$ and $\abs{m} = \prod_{\ell=0}^s m_\ell$. Define $\psi$ to be the block vector defined by
\[
\psi =  \phi_{(0, 0, \hdots, 0)} \oplus \phi_{(0, 0, \hdots, 1)} \oplus  \hdots \oplus \phi_{((m_0-1)/2, (m_1-1)/2, \hdots, (m_s-3)/2)}  \in \bR^{\abs{m}(\abs{m}-1)/2},
\]
where the \ejk{$\phi_i \in \bR^{\abs{m}}$} are ordered lexicographically and for 
\[
i \in \mathcal{I}=\left\{ (0,0, \hdots, 0) , \hdots, ((m_0-1)/2, (m_1-1)/2, \hdots, (m_s-3)/2)\right\},
\]
\[
\left( \phi_{i}\right)_{j}=\left\{ \begin{array}{lr} 1; & j=i\\ -1; &j= m-i-\mathbbm{1}\\ 0; & \textrm{o.w.}\end{array}\right.
\]
We further define $\cG(m)$ to be the orbit of $\psi$ under $\pi(\bigoplus_{\ell=0}^s \bZ_{m_\ell} \times \bigoplus_{\ell=0}^s \bZ_{m_\ell})$.   We call such a $\cG(m)$ a {\it Gabor-Steiner equiangular tight frame}; the full name will be justified below. \footnote{Matlab code to construct $\cG(m_1)$ or $\cG(m_1,m_2)$ for odd integers $m_1, m_2 \geq 3$ may be found at \url{http://www.math.uni-bremen.de/cda/software.html}.}\end{defn}
\begin{ex}
$\Phi_2$ in Example~\ref{ex:hesse} is $\cG(3)$.
\end{ex}
\begin{ex}
Fix a primitive $10$th root of unity $\tau$.  Then $\zeta=\tau^2$ is a primitive $5$th root of unity and for all $k \in \bZ$, $\tau^{2k} = \zeta^k$ and $\tau^{2k+1}= -\zeta^{k-2}$.  Thus we may represent the Gabor-Steiner equiangular tight frame  $\cG(5)$ as
%{\small
\[
\left(\begin{array}{ccccc|ccccc|ccccc|ccccc|ccccc} 0 & 0 & 0 & 0 & 0 & 5 & 5 & 5 & 5 & 5 & \ast & \ast & \ast & \ast & \ast & \ast & \ast & \ast & \ast & \ast & \ast & \ast & \ast & \ast & \ast  \\
\ast & \ast & \ast & \ast & \ast & 0 & 2 & 4 & 6 & 8 & 5 & 7 & 9 & 1 & 3 &\ast & \ast & \ast & \ast & \ast & \ast & \ast & \ast & \ast & \ast  \\ 
\ast & \ast & \ast & \ast & \ast & \ast & \ast & \ast & \ast & \ast & 0 & 4 & 8 & 2 & 6 & 5 & 9 & 3 & 7 & 1 & \ast & \ast & \ast & \ast & \ast  \\
\ast & \ast & \ast & \ast & \ast & \ast & \ast & \ast & \ast & \ast & \ast & \ast & \ast & \ast & \ast &0 & 6 & 2 & 8 & 4 & 5 & 1 & 7 & 3 & 9   \\
5 & 3 & 1 & 9 & 7 & \ast & \ast & \ast & \ast & \ast & \ast & \ast & \ast & \ast & \ast & \ast & \ast & \ast & \ast & \ast & 0 & 8 & 6 & 4 & 2  \\
\hline
\ast & \ast & \ast & \ast & \ast & \ast & \ast & \ast & \ast & \ast &  5 & 5 & 5 & 5 & 5 & \ast & \ast & \ast & \ast & \ast & 0 & 0 & 0 & 0 & 0\\
0 & 2 & 4 & 6 & 8 & \ast & \ast & \ast & \ast & \ast & \ast & \ast & \ast & \ast & \ast &  5 & 7 & 9 & 1 & 3 & \ast & \ast & \ast & \ast & \ast\\
\ast & \ast & \ast & \ast & \ast & 0 & 4 & 8 & 2 & 6 & \ast & \ast & \ast & \ast & \ast & \ast & \ast & \ast & \ast & \ast &5 & 9 & 3 & 7 & 1 \\
5 & 1 & 7 & 3 & 9 & \ast & \ast & \ast & \ast & \ast &  0& 6 & 2 & 8 & 4 & \ast & \ast & \ast & \ast & \ast & \ast & \ast & \ast & \ast & \ast\\
\ast & \ast & \ast & \ast & \ast &5 & 3 & 1 & 9 & 7 & \ast & \ast & \ast & \ast & \ast & 0 & 8 & 6 & 4 & 2 & \ast & \ast & \ast & \ast & \ast
\end{array} \right).
\]%}
where $\ast$ represents  $0$ and a number $k$ represents $\tau^k$.
\end{ex}
\ejk{
\begin{ex} Let $m = (3,3)$, then 
\[
\cI = \{ (0,0), (0,1), (0,2), (1,0) \},
\]
and
{\small
\[
\phi_{(0,0)} = \left( \begin{array}{c} 1 \\ 0 \\ 0 \\ 0 \\ 0 \\ 0 \\ 0 \\ 0 \\ -1 \end{array} \right), \enskip 
\phi_{(0,1)} = \left( \begin{array}{c} 0 \\ 1 \\ 0 \\ 0 \\ 0 \\ 0 \\ 0 \\ -1 \\ 0 \end{array} \right), \enskip 
\phi_{(0,2)} = \left( \begin{array}{c} 0 \\ 0 \\ 1 \\ 0 \\ 0 \\ 0 \\ -1 \\ 0 \\ 0 \end{array} \right), \enskip \textrm{and} \enskip
\phi_{(1,0)} = \left( \begin{array}{c} 0 \\ 0 \\ 0 \\ 1 \\ 0 \\ -1 \\ 0 \\ 0 \\ 0 \end{array} \right).
\]
}
In particular, the $\psi$ for $\cG(3,3)$ is the same as the $\psi$ for $\cG(9)$ but with different parameterization.
\end{ex}}

\begin{thm}\label{thm:Steiner}
Let $m=(m_0, \hdots, m_s)$ be a vector of odd integers $\geq 3$. Then for each $k \in \bigoplus_{\ell=0}^s \bZ_{m_\ell}$,
\[
\beta_k = \left\{ \pi(k, \kappa) \psi : \kappa \in \bigoplus_{\ell=0}^s \bZ_{m_\ell} \right\}
\]
is a simplex, and further $\cG(m)$ is a Steiner equiangular tight frame.
\end{thm}
Thus it is justified to call $\cG(m)$ a Gabor-Steiner equiangular tight frame.
\begin{proof}
We first claim that the matrix with columns
\beq\label{eqn:shifts}
\left\{I_{(\abs{m}-1)/2} \otimes T_{m}^{(k)}\abs{\psi} : k \in \bigoplus_{\ell=0}^s \bZ_{m_\ell}\right\},
\eeq
where $\abs{\psi}$ yields the component-wise absolute values, is an incidence matrix of a $(\abs{m},2,1)$-BIBD. By construction, the difference in $\bigoplus_{\ell=0}^s \bZ_{m_\ell}$ of the support locations in each $\phi_i$ is different, namely, $\pm (2i+\mathbbm{1})$  Thus the rows of \eqref{eqn:shifts} are the indicator functions of each pair in $\bigoplus_{\ell=0}^s \bZ_{m_\ell}$ of difference $\pm (2i+\mathbbm{1})$ for 
\[
(i_0,\hdots,i_s) \in \left\{ (0,0, \hdots, 0) , \hdots, ((m_0-1)/2, (m_1-2)/2, \hdots, (m_s-3)/2)\right\},
\]
representing all possible differences between paired elements in $\bigoplus_{\ell=0}^s \bZ_\ell$. Hence \eqref{eqn:shifts} is the indicator matrix of an $(\abs{m},2,1)$-BIBD.

We now want to prove that the non-zero rows of the matrix with columns
\[
\left\{I_{(\abs{m}-1)/2} \otimes M_{m}^{(\kappa)}\abs{\psi} : \kappa \in \bigoplus_{\ell=0}^s \bZ_{m_\ell}\right\}
\]
are $\abs{m}-1$ distinct rows of $F_{m}$, namely all but the row $(m-\mathbbm{1})/2$. However, this follows immediately from the support set of $\psi$.  Thus, for a fixed $k \in \bigoplus_{\ell=0}^s \bZ_{m_\ell}$,  the non-zero rows of the matrix with columns
\[
\left\{I_{(\abs{m}-1)/2} \otimes (M_{m}^{(\kappa)}T_{m}^{(k)})\abs{\psi} : \kappa \in \bigoplus_{\ell=0}^s \bZ_{m_\ell}\right\}
\]
are the rows of $F_{m}$ except for the row $k+(m-\mathbbm{1})/2$.

We have thus proven the claim.
\end{proof}

The Gabor-Steiner equiangular tight frames are -- to the best of our knowledge -- the very first class of equiangular tight frames discovered which are simultaneously Steiner equiangular tight frames and are generated as the orbit of a group action (Theorem~\ref{thm:Steiner}). However, there is a class of equiangular tight frames which are not Gabor equiangular tight frames but may be viewed as being generated by a group action and are equivalent to certain Steiner equiangular tight frames.  Namely, let $G = \oplus_{\ell=0}^s \bZ_{m_\ell}$ be an arbitrary  finite abelian group.  Then the Fourier matrix $F_{m} = F_{m_0, \hdots, m_s}$ is the character table of $G$. If $D$ is a difference set~\cite{Handbook} $G$, then the submatrix $\Phi$ of $F_{m}$ consisting of the rows corresponding to $D$ and the columns corresponding to all of $\widehat{G}$ generates an equiangular tight frame \cite{DiFe07,StH03,XZG05,Ding06,GoRo09}. Inspired by \cite{BoPa15}, one may write this as a group action by letting $\psi \in \bC^{\abs{m}}$ be the vector that is $1$ on $D$ and $0$ on $G \backslash D$ and letting 
\[
 \left\{ M_m^{(\kappa)}: \kappa = (\kappa_0, \hdots, \kappa_s)\in \bigoplus_{\ell=0}^s \bZ_{m_\ell} \right\} \cong G
\]
act on $\psi$.  This results in a collection of vectors in $\bC^{\abs{m}}$ which look like the vectors in $\Phi$ padded with $\abs{m} - \abs{D}$ zero rows.  Thus the vectors form an equiangular tight frame for their span \cite{FMJ16}.  It was further shown in~\cite{JMF13} that equiangular tight frames formed from so-called McFarland difference sets are equivalent to certain Steiner equiangular tight frames.  The equivalence in~\cite{JMF13} between the two classes of equiangular tight frames is proven via a change of basis by an explicit unitary matrix.  However, in general, it is not necessary to define an explicit mapping to verify switching equivalence as one can perform a simple test on the inner products instead, as seen in Theorem~\ref{thm:TP}.

\begin{thm}[\cite{AFF11,ChWa16,FJKM17}] \label{thm:TP}
Let $m=(m_0, \hdots, m_s)$ be a vector of odd integers $\geq 3$ with $\abs{m} = \prod_{\ell=0}^s m_\ell$ and $\Phi = \left\{\varphi_{k,\kappa}: (k,\kappa) \in \left(\bigoplus_{\ell=0}^s\bZ_{m_\ell}\right)^2\right\}$, $\Psi= \left\{\psi_{k,\kappa}: (k,\kappa) \in \left(\bigoplus_{\ell=0}^s\bZ_{m_\ell}\right)^2\right\}$ be equiangular tight frames.  Further set the {\it triple products} to be
\[
\TP_\Phi((k,\kappa),(\tilde{k},\tilde{\kappa}),(\hat{k},\hat{\kappa})) = \ip{\varphi_{k,\kappa}}{\varphi_{\tilde{k},\tilde{\kappa}}}\ip{\varphi_{\tilde{k},\tilde{\kappa}}}{\varphi_{\widehat{k},\widehat{\kappa}}}\ip{\varphi_{\widehat{k},\widehat{\kappa}}}{\varphi_{k,\kappa}},
\]
where the subscript $\Phi$ is omitted when the equiangular tight frame is clear from context.

Then $\Phi$ and $\Psi$ are switching equivalent if and only if 
\[
\TP_\Phi((k,\kappa),(\tilde{k},\tilde{\kappa}),(\hat{k},\hat{\kappa})) = \TP_\Psi(\tau(k,\kappa),\tau(\tilde{k},\tilde{\kappa}),\tau(\hat{k},\hat{\kappa}))
\] 
for some permutation $\tau: \left(\bigoplus_{\ell=0}^s\bZ_{m_\ell}\right)^2 \rightarrow \left(\bigoplus_{\ell=0}^s\bZ_{m_\ell}\right)^2$ and all distinct $(k,\kappa),(\tilde{k},\tilde{\kappa}),(\hat{k},\hat{\kappa}) \in \left(\bigoplus_{\ell=0}^s\bZ_{m_\ell}\right)^2$.

For $\beta \subseteq \left(\bigoplus_{\ell=0}^s\bZ_{m_\ell}\right)^2$, $\left\{\varphi_{k,\kappa}: (k,\kappa) \in \beta\right\}$ is a simplex if and only if $\abs{\beta} = \abs{m}$ and all of the triple products are real and negative. An equiangular tight frame $\Phi$ contains a simplex if and only if the spark of $\Phi$ saturates the bound in \eqref{eqn:spark}. 
\end{thm}
We note that triple products are also known as \emph{$3$-vertex Bargmann invariants} \cite{RAMS99}.  In order to study the structure of simplices in equiangular tight frames, we will analyze the triple products. 

\section{A family of Gabor equiangular tight frames associated with balanced incomplete block designs} \label{sec:BIBD}
With the help of Theorem~\ref{thm:TP} we will be able to prove the existence of infinite classes of equiangular tight frames which have binders that are balanced incomplete block designs.

%%% MT lemma
\begin{lem}\label{lem:ip}
Let $m=(m_0, \hdots, m_s)$ be a vector of odd integers $\geq 3$, $\abs{m} = \prod_{\ell =0}^s m_{\ell}$, with $\zeta_{m_\ell}$ the primitive $m_\ell$th roots of unity generating $\cG(m)$. Further let $k, \tilde{k}, \kappa, \tilde{\kappa} \in \bigoplus_{\ell=0}^s \bZ_{m_\ell}$.  Then
\[
\ip{\pi(k,\kappa)\psi}{\pi(\tilde{k},\tilde{\kappa})\psi} = 
\abs{m}\delta_{k,\tilde{k}} \delta_{\kappa,\tilde{\kappa}}  -\prod_{\ell=0}^s \zeta_{m_\ell}^{(\kappa_\ell-\tilde{\kappa}_\ell)(\tilde{k}_\ell+k_\ell-1)/2},
\]
where for each $\ell$, $1/2$ in the exponent refers to the multiplicative inverse of $2$ modulo $m_\ell$ and $\delta_{x,y}$ is one if $x= y$ and zero otherwise.
\end{lem}
\begin{proof}
Set $\mathcal{I}=\left\{ (0,0, \hdots, 0) , \hdots, ((m_0-1)/2, (m_1-1)/2, \hdots, (m_s-3)/2)\right\}$. 

We first consider the case that $\tilde{k} = k$.  Then both $\pi(k,\kappa)\psi$ and $\pi(\tilde{k},\tilde{\kappa})\psi$ are vectors in $\beta_k$ and the non-zero entries correspond to the $\kappa$ and $\tilde{\kappa}$ columns of $F_{m}$ without row $(m-\mathbbm{1})/2+k$.  If $\kappa = \tilde{\kappa}$, then the inner product is precisely $\abs{m}-1$.  Otherwise, the inner product is 
\[
- \prod_{\ell=0}^s \zeta_{m_\ell}^{[(m_\ell-1)/2 + k_\ell](\kappa_\ell-\tilde{\kappa}_\ell)} =   -\prod_{\ell=0}^s \zeta_{m_\ell}^{(\kappa_\ell-\tilde{\kappa}_\ell)(\tilde{k}_\ell+k_\ell-1)/2}.
\]

For $k \neq \tilde{k}$ there is a unique $i^\ast \in \mathcal{I}$ such that $\tilde{k} - k = 2i^\ast+ \mathbbm{1}$ or $-(2i^\ast + \mathbbm{1})$.  We compute
\begin{align*}
\lefteqn{\ip{\pi(k,\kappa)\psi}{\pi(\tilde{k},\tilde{\kappa})\psi}= \sum_{i \in \mathcal{I}} \ip{M_m^{(\kappa)}T_m^{(k)}\phi_i}{M_m^{(\tilde{\kappa})}T_m^{(\tilde{k})}\phi_i }}\\
&= \sum_{i \in \cI} \ip{\left(\begin{array}{lr}1; & j=i+k\\[1mm]  -1; & j=m-i-\mathbbm{1}+k\\[1mm] 0; & \textrm{o.w.} \end{array}\right)_j}{M_m^{(\tilde{\kappa}-\kappa)}\left(\begin{array}{lr}1; & j=i+\tilde{k}\\[1mm]  -1; & j=m-i-\mathbbm{1}+\tilde{k}\\[1mm] 0; & \textrm{o.w.} \end{array}\right)_j}\\
&=\! \sum_{i \in \cI}\!\! \ip{\!\!\left(\begin{array}{lr}\!1;\! \!& \!\! j=i+k\!\\[1mm]  \!-1;\!\! &\!\! j=m-i-\mathbbm{1}+k\!\\[1mm]\! 0;\! \!& \!\!\textrm{o.w.} \!\end{array}\right)_j\!\!}{\!\!\left(\begin{array}{lr}\!\prod_{\ell=0}^s \zeta_{m_\ell}^{(\tilde{\kappa}_\ell - \kappa_\ell)(i_\ell + \tilde{k}_\ell)};\! \!& \!\! j=i+\tilde{k}\!\\[1mm] \! -\prod_{\ell=0}^s \zeta_{m_\ell}^{(\tilde{\kappa}_\ell - \kappa_\ell)(-i_\ell -1+ \tilde{k}_\ell)};\! \!&\!\! j=m-i-\mathbbm{1}+\tilde{k}\!\\[1mm]\! 0; \!\!&\!\! \textrm{o.w.} \!\end{array}\right)_j\!\!}\\
&= \sum_{i \in \cI} \left\{\begin{array}{lr}  -\prod_{\ell=0}^s \zeta_{m_\ell}^{(\kappa_\ell - \tilde{\kappa}_\ell)(-i_\ell - 1 + \tilde{k}_\ell)}; & \tilde{k}-k =2i + \mathbbm{1} \\[1 mm]
-\prod_{\ell=0}^s \zeta_{m_\ell}^{(\kappa_\ell - \tilde{\kappa}_\ell)(i_\ell  + \tilde{k}_\ell)} ; & \tilde{k}-k =-(2i + \mathbbm{1}) \\  0; & \textrm{o.w.} \end{array}\right.\\
&=\left\{\begin{array}{lr}  -\prod_{\ell=0}^s \zeta_{m_\ell}^{(\kappa_\ell - \tilde{\kappa}_\ell)(-i^\ast_\ell - 1 + \tilde{k}_\ell)}; & \tilde{k}-k =2i^\ast + \mathbbm{1} \\[1 mm]
-\prod_{\ell=0}^s \zeta_{m_\ell}^{(\kappa_\ell - \tilde{\kappa}_\ell)(i^\ast_\ell  + \tilde{k}_\ell)} ; & \tilde{k}-k =-(2i^\ast + \mathbbm{1})\end{array}\right.\\
   &= \left\{\begin{array}{lr}  -\prod_{\ell=0}^s \zeta_{m_\ell}^{(\kappa_\ell -
 \tilde{\kappa}_\ell)(-(\tilde{k}_\ell -k_\ell -1)/2- 1 + \tilde{k}_\ell)}; &  \tilde{k}-k \in 2\cI + \mathbbm{1}\\[1 mm]
-\prod_{\ell=0}^s \zeta_{m_\ell}^{(\kappa_\ell - \tilde{\kappa}_\ell)((-\tilde{k}_\ell + k_\ell -1)/2 + \tilde{k}_\ell)} ; &  \tilde{k}-k \in -(2\cI + \mathbbm{1}) \end{array}\right.\\
 &=  -\prod_{\ell=0}^s \zeta_{m_\ell}^{(\kappa_\ell-\tilde{\kappa}_\ell)(\tilde{k}_\ell+k_\ell-1)/2}.
\end{align*}
\end{proof}

\begin{lem}\label{lem:TP}
Let $m=(m_0, \hdots, m_s)$ be a vector of odd integers $\geq 3$, $\abs{m} = \prod_{\ell =0}^s m_{\ell}$, with $\zeta_{m_\ell}$ the primitive $m_\ell$th roots of unity generating $\cG(m)$. Further let $(k, \kappa), (\tilde{k}, \tilde{\kappa}), (\widehat{k},  \widehat{\kappa}) \in \left(\bigoplus_{\ell=0}^s \bZ_{m_\ell}\right) \times \left( \bigoplus_{\ell=0}^s \bZ_{m_\ell}\right)$ be distinct.  Then
\begin{align}
\lefteqn{\hspace{-60mm}\TP((k,\kappa),(\tilde{k},\tilde{\kappa}),(\widehat{k},\widehat{\kappa})) = \ip{\pi(k,\kappa)\psi}{\pi(\tilde{k},\tilde{\kappa})\psi}\ip{\pi(\tilde{k},\tilde{\kappa})\psi}{\pi(\widehat{k},\widehat{\kappa})\psi}\ip{\pi(\widehat{k},\widehat{\kappa})\psi}{\pi(k,\kappa)\psi}} \nonumber\\
&\hspace{-60mm}= -\prod_{\ell=0}^s \zeta_{m_\ell}^{[k_\ell (\widehat{\kappa}_\ell-\tilde{\kappa}_\ell) + \tilde{k}_\ell( \kappa_\ell-\widehat{\kappa}_\ell) + \widehat{k}_\ell(\tilde{\kappa}_\ell-\kappa_\ell)]/2} \label{eqn:TP}
 \\
&\hspace{-60mm}= -\prod_{\ell=0}^s \zeta_{m_\ell}^{[(k_\ell \widehat{\kappa}_\ell - \widehat{k}_\ell \kappa_\ell) + (\tilde{k}_\ell \kappa_\ell - k_\ell \tilde{\kappa}_\ell) + (\widehat{k}_\ell \tilde{\kappa}_\ell -  \tilde{k}_\ell \widehat{\kappa}_\ell)]/2} \label{eqn:TP2}
\end{align}
\end{lem}
\begin{proof}
We compute using Lemma~\ref{lem:ip}: 
\begin{align*}
\lefteqn{\ip{\pi(k,\kappa)\psi}{\pi(\tilde{k},\tilde{\kappa})\psi}\ip{\pi(\tilde{k},\tilde{\kappa})\psi}{\pi(\widehat{k},\widehat{\kappa})\psi}\ip{\pi(\widehat{k},\widehat{\kappa})\psi}{\pi(k,\kappa)\psi}}\\
&= -\prod_{\ell=0}^s \zeta_{m_\ell}^{[(\kappa_\ell-\tilde{\kappa}_\ell)(\tilde{k}_\ell + k_\ell -1) +  (\tilde{\kappa}_\ell-\widehat{\kappa}_\ell)(\widehat{k}_\ell + k_\ell -1) +(\widehat{\kappa}_\ell-\kappa_\ell)(k_\ell + \widehat{k}_\ell-1) ]/2}.
\end{align*}
Isolating and manipulating twice the exponent for the $\ell$ factor, we obtain
\begin{align*}
\lefteqn{\hspace{-30mm}(\kappa_\ell-\tilde{\kappa}_\ell)(\tilde{k}_\ell + k_\ell -1) +  (\tilde{\kappa}_\ell-\widehat{\kappa}_\ell)(\widehat{k}_\ell + k_\ell -1) +(\widehat{\kappa}_\ell-\kappa_\ell)(k_\ell + \widehat{k}_\ell-1)}\\
&\hspace{-30mm}= k_\ell (\widehat{\kappa}_\ell-\tilde{\kappa}_\ell) + \tilde{k}_\ell( \kappa_\ell-\widehat{\kappa}_\ell) + \widehat{k}_\ell(\tilde{\kappa}_\ell-\kappa_\ell),
\end{align*}
as desired.
\end{proof}

\begin{prop}\label{prop:ab}
Let $m=(m_0, \hdots, m_s)$ be a vector of odd integers $\geq 3$. For any $a, b \in \bigoplus_{\ell=0}^s \bZ_{m_\ell}$,
%\beq\label{eqn:beta_ab}
\[
\beta_{a,b} = \left\{ (k, \diag(a)\,k+b): k \in \bZ_m\right\},
\]
%\eeq
where $\diag(a)\, k = (a_0 k_0, \hdots a_s  k_s)$, is a simplex.  Thus $\cB(\cG(m))$ contains at least $\abs{m}(\abs{m}+1)$ simplices.
\end{prop}
\begin{proof}
Every set of 3 distinct points in $\beta_{a,b}$ have the form
\[
(k, \diag(a)\,k +b),\, (\tilde{k}, \diag(a)\,\tilde{k} +b),\, (\widehat{k}, \diag(a)\,\widehat{k} +b)
\]
with $k\neq \tilde{k} \neq \widehat{k} \neq k$. Thus, for each $\ell \in \{0,\hdots, s\}$
%\begin{align*}
%k_\ell (\tilde{\kappa}_\ell - \widehat{\kappa}_\ell) + \tilde{k}_\ell(\widehat{\kappa}_\ell - \kappa_\ell) + \widehat{k}_\ell(\kappa_\ell - \tilde{\kappa}_\ell) &= k_\ell a_\ell (\tilde{k}_\ell - \widehat{k}_\ell) + \tilde{k}_\ell a_\ell(\widehat{k}_\ell - k_\ell) + \widehat{k}_\ell a_\ell (k_\ell - \tilde{k}_\ell) = 0.
%\end{align*}
\begin{align*}
k_\ell (\widehat{\kappa}_\ell-\tilde{\kappa}_\ell ) + \tilde{k}_\ell(\kappa_\ell-\widehat{\kappa}_\ell) + \widehat{k}_\ell(\tilde{\kappa}_\ell-\kappa_\ell) &= k_\ell a_\ell (\widehat{k}_\ell- \tilde{k}_\ell ) + \tilde{k}_\ell a_\ell(k_\ell-\widehat{k}_\ell ) + \widehat{k}_\ell a_\ell (\tilde{k}_\ell-k_\ell ) = 0.
\end{align*}
Using Lemma~\ref{lem:TP}, we obtain
\[
\TP((k,\kappa),(\tilde{k},\tilde{\kappa}),(\widehat{k},\widehat{\kappa})) =-1.
\]
Thus is follows from Theorem~\ref{thm:TP} that $\beta_{a,b}$ is a simplex. There are $\abs{m}^2$ such $\beta_{a,b}$, and it follows from Theorem~\ref{thm:Steiner} that the binder also contains $\abs{m}$ $\beta_k$'s.
\end{proof}

\begin{prop}\label{prop:prod}
Let $m=(m_0, \hdots, m_s)$ be a vector of pairwise relatively prime odd integers $\geq 3$. If for each $\ell \in \{ 0, \hdots, s\}$, $\cB_\ell$ is the binder for $\cG(m_\ell)$, then the binder of $\cG(m)$ is
\[
\cB = \Big\{ \big\{ \left((k_0, \hdots, k_s),(\kappa_0, \hdots, \kappa_s)\right): \forall \ell \in \{0, \hdots, s\} \, (k_\ell,\kappa_\ell) \in \beta_\ell \big\}: \forall \ell \in \{0, \hdots, s\} \, \beta_\ell \in \cB_\ell\Big\}.
\] 
Thus if $m$ is an odd integer $\geq 3$ with prime factorization $\prod_i p_i^{e_i}$, $\cG(m)$ has at least $\prod_i p_i^{e_i} (p_i^{e_i} +1)$ simplices.
\end{prop}
\begin{proof}
If $m_i,\, m_j$ are relatively prime and $\zeta_{m_i}$, $\zeta_{m_j}$ are primitive roots of unity, then 
\[
\zeta_{m_i}^{k_i} \zeta_{m_j}^{k_j} = 1 \enskip \textrm{if and only if} \enskip m_i \big\vert k_i, \, m_j \big\vert k_j.
\]
Thus, with the hypotheses above,
%\[
%-\prod_{\ell=0}^s \zeta_{m_\ell}^{[k_\ell (\tilde{\kappa}_\ell - \widehat{\kappa}_\ell) + \tilde{k}_\ell(\widehat{\kappa}_\ell - \kappa_\ell) + \widehat{k}_\ell(\kappa_\ell - \tilde{\kappa}_\ell)]/2}
%\]
\[
-\prod_{\ell=0}^s \zeta_{m_\ell}^{[k_\ell ( \widehat{\kappa}_\ell-\tilde{\kappa}_\ell) + \tilde{k}_\ell(\kappa_\ell-\widehat{\kappa}_\ell) + \widehat{k}_\ell(\tilde{\kappa}_\ell-\kappa_\ell)]/2}
\]
can only be equal to negative one if each of the terms is equal to one and the simplices in $\cB$ must be Cartesian products of the simplices in the $\cB_\ell$.

The lower bound on $\abs{\cB(\cG(m))}$ for $m=\prod_i p_i^{e_i}$ follows from the fact that $\bZ_m \cong \oplus_{i} \bZ_{p_i^{e_i}}$ since the prime factors are trivially relatively prime.
\end{proof}

\begin{prop}\label{prop:Ab}
Let $m$ an odd integer $\geq 3$ and $(m,\hdots,m)$ length $s+1$. Then each $A \in \SYM_{s+1}(\bZ_m)$ and $b \in \bZ_m^{s+1}$ defines a simplex in $\cG((m, \hdots, m))$, namely
\[
\beta_{A,b} = \Big\{\left( k, Ak + b\right): k \in   \bZ_m^{s+1}\Big\}.
\]
Thus $\cB(\cG(m,m, \hdots, m))$ has at least $\abs{m}\left(\abs{m}^{(s+2)(s+1)/2}+1\right)$ simplices.
\end{prop}
\begin{proof}
We first note that $\bigoplus_{\ell=0}^s \bZ_m \cong \bZ_m^{s+1}$. Letting $(k,\kappa)$, $(\tilde{k},\tilde{\kappa})$, $(\widehat{k},\widehat{\kappa})$ be arbitrary distinct points in $\beta_{A,b}$, we compute the triple product:
%\begin{align*}
%\TP((k,\kappa),(\tilde{k},\tilde{\kappa}),(\widehat{k},\widehat{\kappa})) &= -\prod_{\ell=0}^s \zeta_{m}^{[k_\ell (\tilde{\kappa}_\ell - \widehat{\kappa}_\ell) + \tilde{k}_\ell(\widehat{\kappa}_\ell - \kappa_\ell) + \widehat{k}_\ell(\kappa_\ell - \tilde{\kappa}_\ell)]/2} = - \zeta_m^{\textrm{EXP}},
%\end{align*}
\begin{align*}
\TP((k,\kappa),(\tilde{k},\tilde{\kappa}),(\widehat{k},\widehat{\kappa})) &= -\prod_{\ell=0}^s \zeta_{m}^{[k_\ell (\widehat{\kappa}_\ell-\tilde{\kappa}_\ell ) + \tilde{k}_\ell( \kappa_\ell-\widehat{\kappa}_\ell) + \widehat{k}_\ell(\tilde{\kappa}_\ell-\kappa_\ell )]/2} = - \zeta_m^{\textrm{EXP}},
\end{align*}
where 
%\begin{align*}
%2 \cdot \textrm{EXP} &= \sum_{\ell=0}^s \left( k_\ell (\tilde{\kappa}_\ell - \widehat{\kappa}_\ell) + \tilde{k}_\ell(\widehat{\kappa}_\ell - \kappa_\ell) + \widehat{k}_\ell(\kappa_\ell - \tilde{\kappa}_\ell) \right)\\
%&= \ip{k}{A(\tilde{k} - \widehat{k})} +  \ip{\tilde{k}}{A(\widehat{k}-k)} +  \ip{\widehat{k}}{A(k - \tilde{k})}\\
%&= \ip{k}{A\tilde{k}} - \ip{Ak}{\widehat{k}} +  \ip{\tilde{k}}{A\widehat{k}}- \ip{A\tilde{k}}{k} +  \ip{\widehat{k}}{Ak} - \ip{A\widehat{k}}{\tilde{k}}= 0,
%\end{align*}
\begin{align*}
2 \cdot \textrm{EXP} &= \sum_{\ell=0}^s \left( k_\ell (\widehat{\kappa}_\ell-\tilde{\kappa}_\ell ) + \tilde{k}_\ell( \kappa_\ell-\widehat{\kappa}_\ell) + \widehat{k}_\ell(\tilde{\kappa}_\ell-\kappa_\ell )\right)\\
&= \ip{k}{A(\widehat{k}-\tilde{k} )} +  \ip{\tilde{k}}{A(k-\widehat{k})} +  \ip{\widehat{k}}{A( \tilde{k}-k)}\\
&=  \ip{Ak}{\widehat{k}}-\ip{k}{A\tilde{k}} +   \ip{A\tilde{k}}{k} -\ip{\tilde{k}}{A\widehat{k}}+  \ip{A\widehat{k}}{\tilde{k}}- \ip{\widehat{k}}{Ak}= 0,
\end{align*}
as desired.  The lower bound on $\abs{\cB(\cG(m,m,\hdots,m))}$ comes from counting the symmetric affine transformations over $\bZ_m^{s+1}$.
\end{proof}

In the coming proofs, we will frequently need to fit a line to two points in the plane $\bZ_p^2$.  Thus for $(k, \kappa), (\tilde{k},\tilde{\kappa}) \in \bZ_p^2$ with $k \neq \tilde{k}$, we define
\[
(a,b) = SI((k,\kappa), (\tilde{k},\tilde{\kappa})) \in \bZ_p^2
\]
to be the unique (since $\bZ_p$ is a field) slope and intercept such that $\kappa = ak +b$ and $\tilde{\kappa} = a\tilde{k} +b$. For $a_0, a_1, a_2 \in R$, we define $\Sym(a_0,a_1,a_2) \in \SYM_2(R)$ to be % \ejk{Should we change the notation $\Sym(a_0,a_1,a_2)$ since it is very similar to $\SYM_2(R)$?}
\[
\Sym(a_0,a_1,a_2) = \left( \begin{array}{cc} a_0 & a_1 \\ a_1 & a_2 \end{array} \right).
\]

\begin{thm}\label{thm:AG}
Let $p$ be an odd prime. Then the binder $\cB$ of $\cG(p)$ is precisely $AG(2,p)$ (see Definition~\ref{defn:affgeo}).
\end{thm}
\begin{proof}
It follows from Theorem~\ref{thm:Steiner} that for each $k \in \bZ_p$, 
\[
\beta_k = \{ (k, \kappa): \kappa \in \bZ_p\} \in \cB.
\]
These are the vertical lines.  It follows as a corollary to either Proposition~\ref{prop:ab} or~\ref{prop:Ab} that for each $a, b \in \bZ_p$
\[
\beta_{a,b} = \{ (k, ak+b): k \in \bZ_p\} \in \cB,
\]
that is, the non-vertical lines.  Hence $AG(2,p) \subseteq \cB$.

To prove inclusion the other direction, let $\beta \in \cB$. We will actually prove a stronger statement than $\cB \subseteq AG(2,p)$, namely that all triple products which yield a $-1$ belong to a block in $AG(2,p)$. To this end choose distinct $(k, \kappa), (\tilde{k}, \tilde{\kappa}), (\widehat{k},  \widehat{\kappa}) \in \beta$. There are two cases: $k = \tilde{k}$ and $k \neq \tilde{k}$.

We plug $k = \tilde{k}$ into \eqref{eqn:TP} from Lemma~\ref{lem:TP}:
%\[
%-1 = \TP((k,\kappa),(\tilde{k},\tilde{\kappa}),(\widehat{k},\widehat{\kappa})) = - \zeta_{p}^{[k(\tilde{\kappa} - \widehat{\kappa}) + k(\widehat{\kappa}- \kappa) + \widehat{k}(\kappa - \tilde{\kappa})]/2} = -\zeta_p^{\textrm{EXP}},
%\]
\[
-1 = \TP((k,\kappa),(\tilde{k},\tilde{\kappa}),(\widehat{k},\widehat{\kappa})) = - \zeta_{p}^{[k(\widehat{\kappa}-\tilde{\kappa} ) + k(\kappa-\widehat{\kappa}) + \widehat{k}(\tilde{\kappa}-\kappa )]/2} = -\zeta_p^{\textrm{EXP}},
\]
where 
\begin{align*}
2 \cdot \textrm{EXP} &= k(\tilde{\kappa} - \kappa) + \widehat{k}(\kappa - \tilde{\kappa})= (k-\widehat{k})(\tilde{\kappa}-\kappa).
\end{align*}
%\begin{align*}
%2 \cdot \textrm{EXP} &= k(\kappa-\tilde{\kappa}) + \widehat{k}(\tilde{\kappa}-\kappa)= (k-\widehat{k})(\kappa-\tilde{\kappa}).
%\end{align*}
Since the triple product must equal $-1$ and since $p$ is a prime, either $k=\widehat{k}$ or $\tilde{\kappa}=\kappa$.  However, if the latter were true, then $(k, \kappa), (\tilde{k}, \tilde{\kappa})$ would not be distinct points.  Thus the three points are a subset of $\beta_k$.

If $k \neq \tilde{k}$, we set $(a,b) = SI((k,\kappa),(\tilde{k},\tilde{\kappa}))$.  Plugging this into \eqref{eqn:TP}, we obtain $-1 = -\zeta_p^{\textrm{EXP}}$, where
%\begin{align*}
%2 \cdot \textrm{EXP} & = k(a \tilde{k} +b - \widehat{\kappa}) + \tilde{k}(\widehat{\kappa}- ak - b) + \widehat{k}a (k - \tilde{k}) = (k - \tilde{k})(a + b\widehat{k} - \widehat{\kappa}).
%\end{align*}
\begin{align*}
2 \cdot \textrm{EXP} & = k(\widehat{\kappa}-a \tilde{k} -b ) + \tilde{k}(ak +b-\widehat{\kappa} ) + \widehat{k}a ( \tilde{k}-k) = (k - \tilde{k})(\widehat{\kappa}-a - b\widehat{k}).
\end{align*}
Since $k \neq \tilde{k}$, it must hold that $\widehat{\kappa} = a + b \widehat{k}$, and the three points lie in $\beta_{a,b}$.
\end{proof}

\begin{cor}\label{cor:prodAG}
Let $p=(p_0, \hdots, p_s)$ be a vector of distinct odd primes. Then the binder of $\cG(p)$,
\[
\cB\!\! = \!\!\Big\{\! \big\{ \left((k_0, \hdots, k_s),(\kappa_0, \hdots, \kappa_s)\right)\!:\! \forall \ell \in \{0, \hdots, s\} \, (k_\ell,\kappa_\ell) \in \beta_\ell \big\}\!:\! \forall \ell \in \{0, \hdots, s\} \, \beta_\ell \in AG(2,p_\ell)\!\Big\},
\] 
%which we write as $\bigtimes_{\ell=0}^\ell AG(p_\ell,2)$ and 
is not a BIBD.
\end{cor}
\begin{proof}
The structure of the binder follows from Proposition~\ref{prop:prod} and Theorem~\ref{thm:AG}.

For $p$ prime any point lies in $p+1$ blocks of $AG(2,p)$ and any distinct pair of points lies in $1$ block. Thus for $p_i \neq p_j$, $(k_i,\kappa_i), (\tilde{k}_i,\tilde{\kappa}_i) \in \bZ_{p_i}^2$ with $(k_i,\kappa_i) \neq (\tilde{k}_i,\tilde{\kappa}_i)$ and  $(k_j,\kappa_j),(\tilde{k}_j,\tilde{\kappa}_j) \in \bZ_{p_j}^2$, the pair
\[
((k_i, k_j),(\kappa_i,\kappa_j)), ((\tilde{k}_i, \tilde{k}_j),(\tilde{\kappa}_i,\tilde{\kappa_j}))
\]
in $(\bZ_{p_i} \oplus \bZ_{p_j})^2$ lies in $1$ block of $AG(p_i) \times AG(p_j)$ if $(k_j,\kappa_j) \neq (\tilde{k}_j,\tilde{\kappa}_j)$ and $p_j+1$ blocks otherwise.  In general no such product $AG(2,p_\ell)$ will again yield a BIBD.
\end{proof}

We would like to show that for an odd prime $p$, the binder of $\cG(p,p)$ forms a balanced incomplete block design which seems to be new.  We will first characterize the balanced incomplete block design and prove that it is, in fact, a balanced incomplete block design.

\begin{thm}\label{thm:bibdsets}
Let $p$ be an odd prime. The collection of blocks
\[
\widetilde{AG}(2,p^2) =  \bigcup_{i=0}^4 \cB_i
\]
with 
\begin{align*}
\cB_0 &= \Big\{ \beta_k^{(0)} = \big\{ (k, \kappa) : \kappa \in \bZ_p^2 \big\} : k \in \bZ_p^2 \Big\},\\
\cB_1 &= \Big\{\beta_{A,b}^{(1)} = \big\{ (k, Ak+b): k \in \bZ_p^{2}\big\} : A \in \SYM_2(\bZ_p), b \in \bZ_p^{2}\Big\},\\
\cB_2 &= \Big\{\beta_{a,b,c_0}^{(2)} = \big\{ ((c_0,k_1),(\kappa_0,ak_1 + b)): k_1,\kappa_0 \in \bZ_p\big\}: a,b,c_0 \in \bZ_p\Big\},\\
\cB_3 &= \Big\{\beta_{a,b,c_1}^{(3)} = \big\{ ((k_0,c_1),(ak_0 + b,\kappa_1)): k_0,\kappa_1 \in \bZ_p\big\}: a,b,c_1 \in \bZ_p\Big\}, \textrm{ and}\\
\cB_4 &= \Big\{\beta_{\alpha, \beta, a, b}^{(4)} \!\!= \!\!\big\{ ((k_0, \alpha k_0 + \beta),(\kappa_0, (-1/\alpha)\kappa_0 + ak_0 +b )\!\!:\!\! k_0, \kappa_0 \in \bZ_p \big\}:\alpha, \beta, a, b \in \bZ_p, \alpha \neq 0\Big\}
\end{align*}
is a $(p^4,p^2,p+1)$-BIBD.
\end{thm}
\begin{proof}
We would like to show that if $(k,\kappa), (\tilde{k},\tilde{\kappa}) \in (\bZ_p^2)^2$ are distinct, then they lie in exactly $p+1$ blocks of $\widetilde{AG}(2,p^2)$. We will consider four such cases.

\noindent\textbf{Case 1}: If $(k_0,k_1) = (\tilde{k}_0, \tilde{k}_1)$, the points clearly lie in exactly one block in $\cB_0$, namely, $\beta_k^{(0)}$ and cannot lie in a block in $\cB_1$. Otherwise $\kappa = \tilde{\kappa}$ would hold, and the points would not be distinct.
\begin{enumerate}
\item If further $\kappa_0 \neq \tilde{\kappa}_0$ and $\kappa_1 \neq \tilde{\kappa}_1$, then the points cannot lie in blocks in $\cB_2$ or $\cB_3$. Let $(a^\ast, b^\ast) = SI((\kappa_0,\kappa_1),(\tilde{\kappa}_0,\tilde{\kappa}_1))$, where  $a^\ast$ is necessarily non-zero.  Set $\alpha = -1/a^\ast$ and $\beta= k_1 - \alpha k_0$.  For any $a \in \bZ_p$, define $b=b^\ast - a k_0$.  Then for each of the $p$ values of $a$, the pair of points lies in $\beta^{(4)}_{\alpha, \beta, a, b}$.
\item If $\kappa_0 = \tilde{\kappa}_0$ and $\kappa_1 \neq \tilde{\kappa}_1$, then the points cannot lie in blocks in $\cB_2$ or $\cB_4$.  For any $a \in \bZ_p$, set $b = \kappa_0 - a k_0$. Thus, the pair of points lies in $p$ different blocks $\beta_{a,b,k_1}^{(3)}$.
\item Similarly, if $\kappa_0 \neq \tilde{\kappa}_0$ and $\kappa_1 = \tilde{\kappa}_1$, then the pair of points lies in $p$ different blocks $\beta^{(2)}_{a,b,k_0}$ and no blocks in either $\cB_3$ or $\cB_4$.
\end{enumerate}
\textbf{Case 2}: If $k_0 = \tilde{k}_0$ and $k_1 \neq \tilde{k}_1$, then the points cannot lie in a block in $\cB_0$, $\cB_3$, or $\cB_4$.  Let $(a, b) = SI((k_1,\kappa_1),(\tilde{k}_1,\tilde{\kappa}_1))$.  Then the pair of points lies in the block $\beta_{a,b,k_0}^{(2)}$.   Now let $(a_1,b^\ast) = SI((k_1,\kappa_0),(\tilde{k}_1,\tilde{\kappa}_0))$. For $a_0 \in \bZ_p$, let $b_0 = b^\ast - a_0 k_0$.  Now set $(a_2,b_1) = SI((k_1,\kappa_1),(\tilde{k}_1,\tilde{\kappa}_1))$ and $A = \Sym(a_0,a_1,a_2)$.
%\[
%A = \left( \begin{array}{cc}a_0 & a_1 \\ a_1 & a_2 \end{array} \right).
%\]
Then for each of the $p$ choices of $a_0$, the pair of points lies in $\beta_{A,(b_0,b_1)}^{(1)}$.\\
\textbf{Case 3}: By symmetry, if $k_0 \neq \tilde{k}_0$ and $k_1 = \tilde{k}_1$, then the pair of points cannot lie in a block in $\cB_0$, $\cB_2$, or $\cB_4$, but the pair does lie in precisely one block of $\cB_3$ and in $p$ blocks of $\cB_1$.\\
\textbf{Case 4}: If $k_0 \neq \tilde{k}_0$ and $k_1 \neq \tilde{k}_1$, then the points cannot lie in a block in $\cB_0$, $\cB_2$, or $\cB_3$.  Let $(\alpha, \beta) = SI((k_0,k_1),(\tilde{k}_0,\tilde{k}_1))$ with $\alpha$ necessarily nonzero.  Then set $(a,b^\ast) = SI((k_0,\kappa_1),(\tilde{k}_0,\tilde{\kappa}_1))$ and $b=b^\ast + (1/\alpha)\kappa_0$.  Thus the points lie in $\beta_{\alpha,\beta,a,b}^{(4)}$.  Now we will show the the points lie in $p$ blocks in $\cB_1$.  Let $a_1 \in \bZ_p$ and set $(a_0,b_0) = SI((k_0,\kappa_0-a_1 k_1),(\tilde{k}_0,\tilde{\kappa}_0-a_1\tilde{k}_1))$ and then $(a_2,b_1) = SI((k_1,\kappa_1-a_1k_0),(\tilde{k}_1,\tilde{\kappa}_1-a_1\tilde{k}_0))$, implying that for $A = \Sym(a_0,a_1,a_2)$, % and $b = ( \begin{array}{cc}b_0 & b_1 \end{array} )^\top$, 
%\[
%A=\left( \begin{array}{cc}a_0 & a_1 \\ a_1 & a_2 \end{array}\right), \quad b= \left(\begin{array}{c} b_0 \\ b_1 \end{array} \right),
%\]
the points lie in $\beta_{A,(b_0, b_1)}^{(1)}$.
\end{proof}
The structure of $\widetilde{AG}(2,p^2)$ is very similar to the affine geometry $AG(2,p^2)$ (cf., Definition~\ref{defn:affgeo}, Example~\ref{ex:affgeo2}), but it is a non-affine BIBD (proven below in Theorem~\ref{thm:bibddecomp}) and contains $p+1$ as many blocks as $AG(2,p^2)$.  In particular, $AG(2,p^2)$ may be represented as
\begin{align*}
\lefteqn{AG(2,p^2)}\\
&= \Big\{\big\{ ak + b\kappa + c = 0: k, \kappa \in GF(p^2) \big\}: (a,b) \in \{(0,1)\} \cup \{(a,0): a \in GF(p^2)\}, c \in GF(p^2) \Big\}\\
%&\subset \Big\{\big\{ ak + b\kappa + c = 0: k, \kappa \in GF(p^2) \big\}: a,b \in \operatorname{Hom}(GF(p^2)), c \in GF(p^2) \Big\} ,
&\subset \Big\{\big\{ ak + b\kappa + c = 0: k, \kappa \in GF(p^2) \big\}: a,b, c \in GF(p^2)\Big\} ,
\end{align*}
%where $\operatorname{Hom}(GF(p^2))$ denotes the additive homomorphisms of $GF(p^2)$, 
while
%$\widetilde{AG}(2,p^2)$ is a subset of
\[
\widetilde{AG}(2,p^2) \subset \Big\{\big\{ Ak + B\kappa + c = 0: k, \kappa \in \bZ_p^2 \big\}: A, B \in \operatorname{Mat}_2(\bZ_p), c \in \bZ_p^2 \Big\}.
\]

For any odd prime $p$ one may obtain a $(p^4,p^2,p+1)$-BIBD by gluing together $p+1$ $(p^4,p^2,1)$-BIBDs like $AG(2,p^2)$ \cite{MaRo85,Jung86,Handbook}.  However, the $(p^4,p^2,p+1)$-BIBDs presented in Theorem~\ref{thm:bibdsets} are not built up from smaller balanced incomplete block designs in a sense made clear in the following theorem and, to the best of our knowledge, represent a new class of balanced incomplete block design construction.  In \cite{Bill82,Ebe04}, a selection of simple, irreducible, quasimultiple $(n,k,\lambda)$-BIBDs are presented, but each construction has a restriction on the parameters, like $k=\lambda$ or $\lambda=3$, that does not apply to $\widetilde{AG}(2,p^2)$.
%in fact by leveraging combinatorial arguments about partitions of integers and sizes of automorphism groups, one can show that there are at least $10^{108}$ non-isomorphic 
In general, it is NP-Complete to determine whether a $(v,k,\lambda)$-BIBD with $\lambda > 1$ contains a $(v,k,\tilde{\lambda})$-BIBD with $\tilde{\lambda} < \lambda$ \cite{CCS85}; however, we are able to prove the irreducibilty of $\widetilde{AG}(2,p^2)$ using a simple pigeonhole principle argument.

\begin{thm} \label{thm:bibddecomp}
Let $p$ be an odd prime. Then $\widetilde{AG}(2,p^2)$ is a simple, non-affine balanced incomplete block design.  Also, $\widetilde{AG}(2,p^2)$ is not decomposable into a union of $p+1$ balanced incomplete block designs isomorphic to $AG(2,p^2)$.  Further, any pair of blocks $\beta, \beta' \in \widetilde{AG}(2,p^2)$ satisfy
\[
 | \beta \cap \beta'  | \in \{p^2, p, 1, 0\} .
\]
\end{thm}
\begin{proof}
It follows from the proof of Theorem~\ref{thm:bibdsets} that $\widetilde{AG}(2,p^2)$ contains no repeated blocks and is thus simple.  

Let $\beta, \beta' \in \widetilde{AG}(2,p^2)$ be distinct.  Then there exist $A, B, A', B' \in \operatorname{Mat}_2(\bZ_p)$ and $c, c' \in \bZ_p^2$ such that
\begin{align*}
\beta &= \big\{ Ak + B\kappa + c = 0: k, \kappa \in \bZ_p^2 \big\} \quad \textrm{and} \quad \beta' = \big\{ A'k + B'\kappa + c' = 0: k, \kappa \in \bZ_p^2 \big\}.
\end{align*}
Thus $\beta \cap \beta'$ is a solution to a system of linear equations over $(\bZ_p^2)^2$ and must be an affine $\bZ_p$-subspace. Since $\beta, \beta'$ are distinct and of size $p^2$, this means that $\abs{\beta\cap\beta'}\in \{p,1,0\}$, as desired.

By Corollary~\ref{cor:affint}, an affine $(p^4,p^2,\lambda)$-BIBD has blocks that intersect in sets of size $0$, $1$, or $p^2$; however, for any $a, b, a', b' \in \bZ_p$ with $a \neq a'$
%\[
%\beta_0^{(0)} \cap \beta_{a,b,0}^{(3)} = \{ ((0,0),(b,\kappa_1)) : \kappa_1 \in \bZ_p\},
%\]
\[
\beta_{a,b,0}^{(3)} \cap \beta_{a',b',0}^{(3)} = \left\{ \left(\left(\frac{b'-b}{a-a'},0\right),\left(\frac{ab'-a'b}{a-a'},\kappa_1\right)\right) : \kappa_1 \in \bZ_p\right\},
\]
which has size $p$.  Thus $\widetilde{AG}(2,p^2)$ is not affine.  Further, since there are $p^3(p-1)/2$ possible $\beta_{a,b,0}^{(3)}$, $\beta_{a',b',0}^{(3)}$ with $a \neq a'$, it follows from the pigeonhole principle that $\widetilde{AG}(2,p^2)$ is also not decomposable into a disjoint union of $p+1$ affine $(p^4,p^2,1)$-BIBDs like $AG(2,p^2)$.
%To prove the remaining claims, we analyze the size of the intersections of the blocks in $\widetilde{AG}(2,p^2)$.
%\noindent\textbf{Case $\mathbf{0 \cap 0}$}: Clearly $\beta^{(0)}_k \cap \beta^{(0)}_{\tilde{k}} = \emptyset$ for $k \neq \tilde{k}$.
%
%
%\noindent\textbf{Case $\mathbf{1 \cap 1}$}: Clearly $\beta^{(1)}_{A,b} \cap \beta^{(0)}_{\tilde{A},\tilde{b}} = \emptyset$ for $A = \tilde{A}$ and $b \neq \tilde{b}$.
\end{proof}

\begin{thm}\label{thm:AG2}
Let $p$ be an odd prime. Then the binder $\cB$ of $\cG(p,p)$ is precisely $\widetilde{AG}(2,p^2)$.
\end{thm}
\begin{proof}
Let $\zeta_p$ be the primitive $p$th root of unity used in the construction of $\cG(p,p)$. 

We first show that $\widetilde{AG}(2,p^2) \subseteq \cB$.  By construction, each block in $\cB_0$ lies in $\cB$. It follows from Proposition~\ref{prop:Ab} that $\cB_1 \subseteq \cB$.  Let $\beta_{a,b,c_0}^{(2)} \in \cB_2$.  Then for distinct $(k, \kappa), (\tilde{k}, \tilde{\kappa}), (\widehat{k},  \widehat{\kappa}) \in \beta_{a,b,c_0}^{(2)}$,
\[
TP((k, \kappa), (\tilde{k}, \tilde{\kappa}), (\widehat{k},  \widehat{\kappa}))= -\zeta_p^{\textrm{EXP}},
\]
where
%\begin{align*}
%\textrm{EXP} &= {\frac{1}{2}  \left( k_0 (\tilde{\kappa}_0 - \widehat{\kappa}_0) + \tilde{k}_0(\widehat{\kappa}_0 - \kappa_0) + \widehat{k}_0(\kappa_0 - \tilde{\kappa}_0) + k_1 (\tilde{\kappa}_1 - \widehat{\kappa}_1) + \tilde{k}_1(\widehat{\kappa}_1 - \kappa_1) + \widehat{k}_1(\kappa_1- \tilde{\kappa}_1) \right)}\\
%&= \frac{1}{2} \left( c_0 (\tilde{\kappa}_0 - \widehat{\kappa}_0) + c_0(\widehat{\kappa}_0 - \kappa_0) + c_0(\kappa_0 - \tilde{\kappa}_0) + k_1 a(\tilde{k}_1 - \widehat{k}_1) + \tilde{k}_1 a(\widehat{k}_1 - k_1) + \widehat{k}_1a(k_1- \tilde{k}_1) \right) = 0.
%\end{align*}
\begin{align*}
\lefteqn{\textrm{EXP}}
\\ &= {\frac{1}{2}  \left( k_0 (\widehat{\kappa}_0-\tilde{\kappa}_0 ) + \tilde{k}_0(\kappa_0-\widehat{\kappa}_0 ) + \widehat{k}_0( \tilde{\kappa}_0-\kappa_0) + k_1 ( \widehat{\kappa}_1-\tilde{\kappa}_1) + \tilde{k}_1( \kappa_1-\widehat{\kappa}_1) + \widehat{k}_1( \tilde{\kappa}_1-\kappa_1) \right)}\\
&= \frac{1}{2} \left( c_0 ( \widehat{\kappa}_0-\tilde{\kappa}_0 ) + c_0( \kappa_0-\widehat{\kappa}_0) + c_0( \tilde{\kappa}_0-\kappa_0) + k_1 a(\widehat{k}_1-\tilde{k}_1 ) + \tilde{k}_1 a( k_1-\widehat{k}_1) + \widehat{k}_1a(\tilde{k}_1-k_1) \right) \\
&= 0.
\end{align*}
Thus, $\beta_{a,b,c_0}^{(2)} \in \cB$.  Similarly each $\beta_{a,b,c_1}^{(3)} \in \cB_3$ lies $\cB$.  Now let $\beta_{\alpha, \beta, a,b}^{(4)} \in \cB_4$.  Then for distinct $(k, \kappa), (\tilde{k}, \tilde{\kappa}), (\widehat{k},  \widehat{\kappa}) \in \beta_{\alpha, \beta, a,b}^{(4)}$,
\[
TP((k, \kappa), (\tilde{k}, \tilde{\kappa}), (\widehat{k},  \widehat{\kappa}))= -\zeta_p^{\textrm{EXP}},
\]
where
%\begin{align*}
%\textrm{EXP} &= {\frac{1}{2}  \big( k_0 (\tilde{\kappa}_0 - \widehat{\kappa}_0) + \tilde{k}_0(\widehat{\kappa}_0 - \kappa_0) + \widehat{k}_0(\kappa_0 - \tilde{\kappa}_0) + (\alpha k_0 + \beta) ((-1/\alpha) (\tilde{\kappa}_0 - \widehat{\kappa}_0)+ a(\tilde{k}_0 - \widehat{k}_0)) } \\  
%&\,\,\,+ (\alpha \tilde{k}_0 + \beta)((-1/\alpha)(\widehat{\kappa}_0 - \kappa_0)+a (\widehat{k}_0 - k_0)) + (\alpha\widehat{k}_0 + \beta)((-1/\alpha)(\kappa_0- \tilde{\kappa}_0)+ a(k_0- \tilde{k}_0)) \big)\\
%&=\frac{1}{2}\big( k_0 (\tilde{\kappa}_0 - \widehat{\kappa}_0) + \tilde{k}_0(\widehat{\kappa}_0 - \kappa_0) + \widehat{k}_0(\kappa_0 - \tilde{\kappa}_0) -(k_0 (\tilde{\kappa}_0 - \widehat{\kappa}_0) + \tilde{k}_0(\widehat{\kappa}_0 - \kappa_0) + \widehat{k}_0(\kappa_0 - \tilde{\kappa}_0)) \big)\\
%&= 0.
%\end{align*}
\begin{align*}
\lefteqn{\textrm{EXP}}\\
 &= {\frac{1}{2}  \big( k_0 ( \widehat{\kappa}_0-\tilde{\kappa}_0) + \tilde{k}_0(\kappa_0-\widehat{\kappa}_0 ) + \widehat{k}_0( \tilde{\kappa}_0-\kappa_0 ) + (\alpha k_0 + \beta) ((-1/\alpha) (\widehat{\kappa}_0-\tilde{\kappa}_0 )+ a(\widehat{k}_0-\tilde{k}_0 )) } \\  
&\,\,\,+ (\alpha \tilde{k}_0 + \beta)((-1/\alpha)( \kappa_0-\widehat{\kappa}_0)+a ( k_0-\widehat{k}_0)) + (\alpha\widehat{k}_0 + \beta)((-1/\alpha)(\tilde{\kappa}_0-\kappa_0)+ a(\tilde{k}_0-k_0)) \big)\\
&=\frac{1}{2}\big( k_0 ( \widehat{\kappa}_0-\tilde{\kappa}_0) + \tilde{k}_0( \kappa_0-\widehat{\kappa}_0) + \widehat{k}_0( \tilde{\kappa}_0-\kappa_0) -(k_0 ( \widehat{\kappa}_0-\tilde{\kappa}_0 ) + \tilde{k}_0(\kappa_0-\widehat{\kappa}_0 ) + \widehat{k}_0( \tilde{\kappa}_0-\kappa_0)) \big)\\
&= 0.
\end{align*}
Thus $\widetilde{AG}(2,p^2) \subseteq \cB$.
 
For the other inclusion, let $(k, \kappa), (\tilde{k}, \tilde{\kappa}), (\widehat{k},  \widehat{\kappa}) \in \beta \in \cB$ be distinct points in a block $\beta$ in the simplex binder.  We will show that they lie in a block in $\widetilde{AG}(2,p^2)$. Then
%\beq\label{eqn:TPp2}
%0=k_0 (\tilde{\kappa}_0 - \widehat{\kappa}_0) + \tilde{k}_0(\widehat{\kappa}_0 - \kappa_0) + \widehat{k}_0(\kappa_0 - \tilde{\kappa}_0) + k_1 (\tilde{\kappa}_1 - \widehat{\kappa}_1) + \tilde{k}_1(\widehat{\kappa}_1 - \kappa_1) + \widehat{k}_1(\kappa_1- \tilde{\kappa}_1).
%\eeq 
\beq\label{eqn:TPp2}
0=k_0 (\widehat{\kappa}_0-\tilde{\kappa}_0 ) + \tilde{k}_0( \kappa_0-\widehat{\kappa}_0) + \widehat{k}_0(\tilde{\kappa}_0-\kappa_0 ) + k_1 ( \widehat{\kappa}_1-\tilde{\kappa}_1) + \tilde{k}_1(\kappa_1-\widehat{\kappa}_1 ) + \widehat{k}_1( \tilde{\kappa}_1-\kappa_1).
\eeq 
We consider a few cases. \\
\textbf{Case 1}: If $(k_0,k_1) = (\tilde{k}_0, \tilde{k}_1)$, then the triple does not lie in a block in $\cB_1$. Also
%\beq\label{eqn:TPp3}
%\eqref{eqn:TPp2} = 0 = (k_0 - \widehat{k}_0)(\tilde{\kappa}_0 - \kappa_0) + (k_1 - \widehat{k}_1)(\tilde{\kappa}_1 -\kappa_1).
%\eeq
\beq\label{eqn:TPp3}
\eqref{eqn:TPp2} = 0 = (k_0 - \widehat{k}_0)( \kappa_0-\tilde{\kappa}_0 ) + (k_1 - \widehat{k}_1)(\kappa_1-\tilde{\kappa}_1).
\eeq
We assume that the points to not lie in a simplex in $\cB_0$, $\cB_2$, or $\cB_3$.  Then $k_0 - \widehat{k}_0, \kappa_0-\tilde{\kappa}_0 , k_1 - \widehat{k}_1,  \kappa_1-\tilde{\kappa}_1 \neq 0$.  Set $(\alpha, \beta) \in SI ((k_0,k_1),(\widehat{k}_0,\widehat{k}_1))$, where $\alpha$ must be nonzero. Plugging this into \eqref{eqn:TPp3}, we obtain
\[
\kappa_1 + (1/\alpha) \kappa_0 = \tilde{\kappa}_1 + (1/\alpha) \tilde{\kappa}_0.
\]
If we set $(a,b) = SI((k_0,\kappa_1 + (1/\alpha) \kappa_0),(\widehat{k}_0,\widehat{\kappa}_1 + (1/\alpha) \widehat{\kappa}_0))$, then the triple lies in $\beta_{\alpha, \beta, a, b}^{(4)}$.\\
\textbf{Case 2}: If $k_0=\tilde{k}_0 ,k_1 \neq \tilde{k}_1$, then the triple does not lie in a block in $\cB_0$, $\cB_3$, or $\cB_4$.  Let $(a_1,b^\ast)= SI((k_1,\kappa_0),(\tilde{k}_1,\tilde{\kappa}_0))$.  Then let $(a_2,b_1) = SI((k_1,\kappa_1 -a_1k_0),(\tilde{k}_1,\tilde{\kappa}_1 -a_1k_0))$.  We plug this into \eqref{eqn:TPp2}:
%\begin{align*}%\label{eqn:TPp4}
%\eqref{eqn:TPp2} &= 0 \\
%&= (k_0 - \widehat{k}_0)(\tilde{\kappa}_0 - \kappa_0) + k_1 (a_1 k_0 + a_2 \tilde{k}_1 + b_1 - \widehat{\kappa}_1) + \tilde{k}_1(\widehat{\kappa}_1 - a_1 k_0 - a_2 k_1 -b_1) + \widehat{k}_1 a_2(k_1- \tilde{k}_1)\\
%&= (a_1 \widehat{k}_0 + a_2\hat{k}_1 +b_1 - \widehat{\kappa}_1)(k_1 - \tilde{k}_1).
%\end{align*}
\begin{align*}%\label{eqn:TPp4}
\eqref{eqn:TPp2}&= 0 \\
&= (k_0 - \widehat{k}_0)(\kappa_0-\tilde{\kappa}_0 ) + k_1 (\widehat{\kappa}_1-a_1 k_0 - a_2 \tilde{k}_1 - b_1 ) \\
& \quad+ \tilde{k}_1(a_1 k_0 + a_2 k_1 +b_1-\widehat{\kappa}_1 ) + \widehat{k}_1 a_2( \tilde{k}_1-k_1)\\
&= (\widehat{\kappa}_1-a_1 \widehat{k}_0 - a_2\hat{k}_1 -b_1 )(k_1 - \tilde{k}_1).
\end{align*}
Thus $\widehat{\kappa}_1 = a_1 \widehat{k}_0 + a_2 \widehat{k}_1 + b_1$. If $\widehat{k}_0 = k_0$, then $\widehat{\kappa}_1 - a_1 \widehat{k}_0 =a_2 \widehat{k}_1 + b_1$, and the triple lies in $\beta_{a_2,b_1+a_1k_0,k_0}^{(2)}$.  Otherwise, if $\widehat{k}_0 \neq k_0$, then choose $(a_0,b_0)=SI((k_0,\kappa_0-a_1k_1),(\widehat{k}_0,\widehat{\kappa}_0-a_1\widehat{k}_1))$. Thus, 
\[
\tilde{\kappa}_0 =a_1\tilde{k}_1 + b^\ast = a_1 \tilde{k}_1 +\kappa_0 -a_1k_1 = a_1 \tilde{k}_1 + a_0\tilde{k}_0 + b_0,
\]
and the triple lies in $\beta_{A,b}^{(1)}$ with $A = \Sym(a_0,a_1,a_2)$.\\
%\[
%A = \left( \begin{array}{cc} a_0 & a_1 \\ a_1 & a_2\end{array}\right).
%\]
\textbf{Case 3}: By symmetry, when $k_0 \neq \tilde{k}_0 ,k_1 = \tilde{k}_1$, all possible triples lie in a block in $\widetilde{AG}(2,p^2)$.\\
\textbf{Case 4}: We now assume that $k_0 \neq \tilde{k}_0 ,k_1 \neq \tilde{k}_1$.  In this case, the triple cannot lie in a block in $\cB_0$, $\cB_2$, or $\cB_3$.  Let $(\alpha,\beta)= SI((k_0,k_1),(\tilde{k}_0,\tilde{k}_1))$, with $\alpha$ necessarily nonzero.  Further set $(e,f)=SI((k_0,\kappa_0),(\tilde{k}_0,\tilde{\kappa}_0))$ and $(g,h)=SI((k_0,\kappa_1),(\tilde{k}_0,\tilde{\kappa}_1))$.  Plugging these into 
%%% Switched version had a number of problems.
\begin{align*}%\label{eqn:TPp5}
\eqref{eqn:TPp2} &= 0\\
& =k_0 (\widehat{\kappa}_0 - e\tilde{k}_0 -f) + \tilde{k}_0 (ek_0 +f -\widehat{\kappa}_0) + \widehat{k}_0 e(\tilde{k}_0 - k_0) \\
&\quad +(\alpha k_0 + \beta)(\widehat{\kappa}_1 -g\tilde{k}_0 -h) + (\alpha \tilde{k}_0 + \beta)(gk_0 + h -\widehat{\kappa}_1) + \widehat{k}_1 g (\tilde{k}_0-k_0)\\
&= (\tilde{k}_0 - k_0)\left(e\widehat{k}_0 + f - \widehat{\kappa}_0 + g\widehat{k}_1 +\alpha h - \beta g - \alpha \widehat{\kappa}_1\right) \nonumber\\
& = (\tilde{k}_0 - k_0)\bigg(e\widehat{k}_0 + f - \widehat{\kappa}_0 + g\widehat{k}_1 +\alpha h - \beta g - \alpha \widehat{\kappa}_1 \\
&\quad+ a_1(\widehat{k}_1 - \alpha \widehat{k}_0 - \beta)-a_1 (\widehat{k}_1 - \alpha \widehat{k}_0 - \beta)\bigg)\nonumber\\
&= (\tilde{k}_0 - k_0)\bigg(\left( (e-a_1 \alpha)\widehat{k}_0 + a_1 \widehat{k}_1 + (f-a_1 \beta) - \widehat{\kappa}_0 \right) \\
&\quad + \left( \alpha a_1 \widehat{k}_0 + (g-a_1)\widehat{k}_1 + \alpha h - \beta (g-a_1) - \alpha \widehat{\kappa}_1 \right)\bigg).
\end{align*}
Let $a_1$ be the unique value in $\bZ_p$ such that
\[
0 =  (e-a_1 \alpha)\widehat{k}_0 + a_1 \widehat{k}_1 + (f-a_1 \beta) - \widehat{\kappa}_0.
\]
Thus,
\[
0 = a_1 \widehat{k}_0 +\frac{g-a_1}{\alpha}\widehat{k}_1 +  h - \frac{\beta}{\alpha} (g-a_1) -  \widehat{\kappa}_1.
\]
We set $a_0 = e-a_1\alpha$, $a_2=\frac{g-a_1}{\alpha}$, $b_0 = f - a_1\beta$, $b_1 = h -\frac{\beta}{\alpha} (g-a_1)$, which yields \ejk{$\widehat{\kappa} = \Sym(a_0,a_1,a_2) \widehat{k} + (\begin{array}{cc} b_0 & b_1 \end{array})^\top$},
%\[
%\left( \begin{array}{c}\widehat{\kappa}_0\\\widehat{\kappa}_1 \end{array} \right) = \left( \begin{array}{cc}a_0 & a_1 \\ a_1 & a_2 \end{array} \right) \left( \begin{array}{c} \widehat{k}_0 \\ \widehat{k}_1\end{array} \right) + \left( \begin{array}{c}b_0 \\ b_1 \end{array} \right),
%\]
with the same relationship for $(k,\kappa)$ and $(\tilde{k},\tilde{\kappa})$. Thus the triple lies in $\beta_{A,b}^{(1)}$. Thus $\cB \subseteq \widetilde{AG}(2,p^2)$.
\end{proof}

\ejk{There seems to be evidence that a stronger result than Theorem~\ref{thm:bibddecomp} holds, namely that for $p$ an odd prime, $\widetilde{AG}(2,p^2)$ contains no copies of $AG(2,p^2)$.  We note that by applying \cite[Theorem 5.2]{FJKM17} to \cite[Theorem 5.1]{FJMPW17} over a finite field of order $p^2$, we may obtain an equiangular tight frame of $p^4$ vectors spanning a space of dimension $p^2(p^2-1)/2$ which has a binder which contains $AG(2,p^2)$. However, one may verify via triple products that these equiangular tight frames are not switching equivalent.  Further, it is easy to find triples of points in simplices corresponding to blocks of  $AG(2,p^2)$ in that binder (e.g., from the line with slope $\alpha$  for $\alpha^2 + 1 = 0$ and intercept $0$) which do not lie in blocks in $\widetilde{AG}(2,p^2)$ under the mapping $\bF_{p^2} \rar \bZ_p \oplus \bZ_p$ that sends $k_0 + \alpha k_1 \rar (k_0, k_1)$ for $\alpha^2 +1 =0$.
}

\begin{cor}
Let $m$ be an odd integer $\geq 3$.   Then the binder $\cB$ of $\cG(m,m)$ contains $\cup_{i=0}^4 \cB_i$ (as defined in Theorem~\ref{thm:bibdsets}).  Thus $\cB(\cG(m,m))$ contains at least $(m+1)m^2(m^2+1)$ simplices.
\end{cor}
\begin{proof}
Simply replace $p$ with $m$ in the first half of the proof of Theorem~\ref{thm:AG2}.
\end{proof}
We note that when $m$ is not prime, $\cup_{i=0}^4 \cB_i$ is not a BIBD and $\cB \neq \cup_{i=0}^4 \cB_i$. We now generalize Theorems~\ref{thm:AG} and~\ref{thm:AG2} as the following conjecture.
\begin{conj}
Let $p$ be an odd prime and $(p,\hdots,p)$ length $s+1$. Then the binder of $\cG(p,\hdots,p)$ is a $(p^{2(s+1)}, p^{s+1}, (p^{s+1}-1)/(p-1))$-BIBD.
\end{conj}

\ejk{In the preceding results, we showed that Gabor-Steiner equiangular tight frames over any type of finite abelian group have rich binders. It is proven in~\cite{King19} that $\cG(p, p, \hdots, p)$ for odd prime $p$ has a high level of symmetry. We believe that $\cG(m)$  \bgb{could provide a set of examples of equiangular tight frames with a variety of symmetry groups for more general choices of the vector} $m$.}

%\ejk{Maybe move?  Maybe put in its own subsection?}
\bgb{We conclude this section by commenting on the relation between these results and related work.}
One may show with help of triple products that the Gabor-Steiner equiangular tight frames are equivalent to certain indirect constructions of equiangular tight frames.  Let $p$ be an odd prime and $\zeta_p$ be a primitive $p$th root of unity. By \cite[Theorem 5.2]{BoEl10b}, there is a $p^2 \times p^2$ Butson Hadamard such that the off-diagonal elements are the negatives of the phases of the Gram matrix $\Phi^\ast \Phi$ of an equiangular tight frame $\Phi$ of $p^2$ vectors in $\bC^{p(p-1)/2}$. %The entries are parameterized by $\bZ_p^2$ and correspond to the scaled inner products of the vectors. %; namely, $\Phi = \{\varphi_{k,\kappa}\}_{(k,\kappa) \in \bZ_p^2}$.  
Let $(k,\kappa), (\tilde{k},\tilde{\kappa}) \in \bZ_p^2$ with $(k,\kappa)\neq (\tilde{k},\tilde{\kappa})$, then $\ip{\varphi_{k,\kappa}}{\varphi_{\tilde{k},\tilde{\kappa}}} = -\zeta_p^{-\tilde{k}\kappa + k \tilde{\kappa}}$.  Thus, 
\[
TP((k,\kappa),(\tilde{k},\tilde{\kappa}),(\widehat{k},\widehat{\kappa}))=-\zeta_p^{k (\tilde{\kappa}- \widehat{\kappa}) + \tilde{k}_\ell(\widehat{\kappa}- \kappa) + \widehat{k}(\kappa - \tilde{\kappa})}.
\]
We note that $\zeta_p^{-2}$ is still a primitive $p$th root of unity. Thus the triple products arising from the construction in \cite{BoEl10a,BoEl10b} (or, more precisely, the Naimark complement of the equiangular tight frames constructed in those papers) with $\zeta_p$ are the same as the triple products for the Gabor-Steiner equiangular tight frame $\cG(p)$ constructed in Definition~\ref{defn:psi} over $\bZ_p$ with primitive root $\zeta_p^{-2}$; hence by Theorem~\ref{thm:TP} they are switching equivalent. Another equivalent construction appears in \cite[Theorem 6.4]{IJM17}.  Namely, let $(m_0,\hdots,m_s)$ be a vector of odd integers and $G = \oplus_{\ell=0}^s \bZ_{m_\ell}$.  Then the  average of the $\abs{m} \times \abs{m}$ identity $I_{\abs{m}}$ and the \emph{parity operator} $\frac{1}{\abs{m}} \sum_{(k,\kappa) \in G \times G} D_m^{(k,\kappa)}$ yields a rank $(\abs{m}-1)/2$ projection  $P_O$ in $\SYM_{\abs{m}}(\bC)$ \cite{ABDF17}.  The orbit of $P_O$ under the complete set of displacement operators $\{ D_m^{(k,\kappa)} : (k,\kappa) \in G \times G\}$ yields an equiangular (with respect to the Hilbert-Schmidt inner product) tight frame of $\abs{m}^2$ matrices in an $\abs{m}(\abs{m}-1)/2$-dimensional subspace of operators on $L^2(G) \rightarrow L^2(G)$.  Let $(k,\kappa), (\tilde{k},\tilde{\kappa}) \in G \times G$ with $(k,\kappa)\neq (\tilde{k},\tilde{\kappa})$.  Up to normalization of the operators and selection of primitive $m_\ell$th roots of unity, the inner products are 
\[
\ip{D_m^{(k,\kappa)}P_O}{D_m^{(\tilde{k},\tilde{\kappa})}P_O}_{H.S.} = -\prod_{\ell=0}^s \zeta_{m_\ell}^{(-\tilde{k}_{\ell}\kappa_\ell + k_\ell \tilde{\kappa}_{\ell})/2}
\]
and thus
\[
\TP((k,\kappa),(\tilde{k},\tilde{\kappa}),(\widehat{k},\widehat{\kappa})) = -\prod_{\ell=0}^s \zeta_{m}^{[k_\ell (\widehat{\kappa}_\ell-\tilde{\kappa}_\ell ) + \tilde{k}_\ell( \kappa_\ell-\widehat{\kappa}_\ell) + \widehat{k}_\ell(\tilde{\kappa}_\ell-\kappa_\ell )]/2},
\]
as desired. 

\ejk{In \bgb{another related paper}~\cite{FJMPW17}, incidence matrices of certain balanced incomplete block designs are phased in a particular way to yield equiangular tight frames for their spans.  These are called \emph{phased BIBD ETFs}.  We analyze the Naimark complement of the specific construction in~\cite[Theorem 5.1]{FJMPW17} (arising from a so-called \emph{$(q^2,q,q)$-polyphase BIBD ETF}), which has the same parameters as Gabor-Steiner ETFs for groups of size odd prime power $q=p^{s+1}$. Let $\chi$ be a non-trivial character over the additive group of $GF(q)$ and let $(k,\kappa), (\tilde{k},\tilde{\kappa}) \in GF(q)^2$.  Then the triple products of the Naimark complement of the phased BIBD ETF are
\[
TP((k,\kappa),(\tilde{k},\tilde{\kappa}),(\widehat{k},\widehat{\kappa}))=-\chi\left(k (\tilde{\kappa}- \widehat{\kappa}) + \tilde{k}_\ell(\widehat{\kappa}- \kappa) + \widehat{k}(\kappa - \tilde{\kappa})\right),
\]
with multiplication defined over the finite field.  Thus, when $q = p$, the Naimark complement of the corresponding phased BIBD ETF with $p^2$ vectors is switching equivalent (up to choice of primitive $p$th root) to $\cG(p)$.  However, for $q = p^2$, the triple products of the Naimark complement of the corresponding phased BIBD ETF are of the form
\[
TP((k,\kappa),(\tilde{k},\tilde{\kappa}),(\widehat{k},\widehat{\kappa}))=-\zeta_p^{\textrm{EXP}},
\]
where 
\begin{align*}
\textrm{EXP} = &k_0 \tilde{\kappa}_0 - k_1\tilde{\kappa}_1 - \tilde{k}_0\kappa_0 + \tilde{k}_1\kappa_1 + a\left(k_0 \tilde{\kappa}_1 + k_1\tilde{\kappa}_0 - \tilde{k}_0\kappa_1 - \tilde{k}_1\kappa_0\right)\\
&+\tilde{k}_0 \hat{\kappa}_0 - \tilde{k}_1\hat{\kappa}_1 - \hat{k}_0\tilde{\kappa}_0 + \hat{k}_1\tilde{\kappa}_1 + a\left(\tilde{k}_0 \hat{\kappa}_1 + \tilde{k}_1\hat{\kappa}_0 - \hat{k}_0\tilde{\kappa}_1 - \hat{k}_1\tilde{\kappa}_0\right)\\
&+\hat{k}_0 {\kappa}_0 - \hat{k}_1 {\kappa}_1 - {k}_0\hat{\kappa}_0 + {k}_1\hat{\kappa}_1 + a\left(\hat{k}_0 {\kappa}_1 + \hat{k}_1{\kappa}_0 - {k}_0\hat{\kappa}_1 - {k}_1\hat{\kappa}_0\right)
\end{align*}
for some choice $a \in \bZ_p \setminus \{0\}$ depending on the character $\chi$, which are unequal to the triple products of $\cG(p,p)$ which have exponent
\[
\textrm{EXP}  = k_0 \tilde{\kappa}_0 - k_1\tilde{\kappa}_1 + \tilde{k}_0\kappa_0 - \tilde{k}_1\kappa_1 + \tilde{k}_0 \hat{\kappa}_0 - \tilde{k}_1\hat{\kappa}_1 + \hat{k}_0\tilde{\kappa}_0 - \hat{k}_1\tilde{\kappa}_1+\hat{k}_0 {\kappa}_0 - \hat{k}_1 {\kappa}_1 + {k}_0\hat{\kappa}_0 - {k}_1\hat{\kappa}_1. 
\]
Since further some of the triple products of $\cG(p^2)$ involve primitive $p^2$th roots of unity not raised to a multiple of $p$, the phased BIBD ETF with $q^2 = p^4$ vectors is not switching equivalent to any Gabor-Steiner ETF.}

%\ejk{START NEW SECTION????}
\section{Spectrahedral arrangements from binders of equiangular tight frames}\label{sec:SAfromB}

As in Section~\ref{sec:MUB}, we can use subsets of vectors in $\Phi$, here with help from $\cG(m)$, to generate spectrahedral arrangements which correspond to subspace packings.  Namely, for each simplex $\beta$ in a binder $\cB$ we consider the collection of appropriately scaled orthogonal projections onto the subspaces spanned by these vectors.  The results are due to the following lemma. 
\begin{lem}\label{overlappack}
Let $\cB$ be a binder of an equiangular tight frame $\Phi$ of $n$ vectors in $\bF^d$.  Then for any two $\beta, \beta' \in \cB$ with corresponding orthogonal projections $P_\beta, P_{\beta'}$, 
\[
\tr(P_\beta P_{\beta'}) = 1 + \abs{\beta \cap \beta'} \frac{\abs{\beta}-2}{\abs{\beta}}.
\]
\end{lem}
\begin{proof}
Let $\norm{\varphi}$ denote the norm of any of the vectors in $\Phi$, which is necessarily constant. For an arbitrary $\beta \in \cB$, all elements of $\cB$ have size $\abs{\beta}$ and for any distinct $\varphi_k, \varphi_{k'} \in \Phi$, 
\[
\absip{\varphi_k}{\varphi_{k'}}^2 = \frac{\norm{\varphi}^2}{(\abs{\beta}-1)^2}.
\]
We further note that if $\beta\in \cB$, then the orthogonal projection onto the subspace spanned by the corresponding vectors is $P_\beta = \frac{\abs{\beta}-1}{\abs{\beta}\norm{\varphi}}\Phi_{\beta}\Phi_{\beta}^\ast$. 
Let $\beta, \beta' \in \cB$.  Then
\begin{align*}
\tr(P_\beta P_{\beta'}) &= \tr\left( \frac{\abs{\beta}-1}{\abs{\beta}\norm{\varphi}}\Phi_{\beta}\Phi_{\beta}^\ast\frac{\abs{\beta'}-1}{\abs{\beta'}\norm{\varphi}}\Phi_{\beta'}\Phi_{\beta'}^\ast\right) \\
&= \frac{(\abs{\beta}-1)^2}{\abs{\beta}^2\norm{\varphi}^2} \sum_{k \in \beta} \sum_{k' \in \beta'} \absip{\varphi_k}{\varphi_{k'}}^2\\
&= \frac{(\abs{\beta}-1)^2}{\abs{\beta}^2\norm{\varphi}^2} \left(\abs{\beta\cap\beta'} \norm{\varphi}^2 + (\abs{\beta}^2 - \abs{\beta \cap \beta'})\frac{\norm{\varphi}^2}{(\abs{\beta}-1)^2} \right)\\
&= 1 + \abs{\beta \cap \beta'} \frac{\abs{\beta}-2}{\abs{\beta}}.
\end{align*}
\end{proof}
The following theorem is a specific application of \cite[Theorem 6.2]{FJKM17}, which generalizes \cite[Section 2.4]{Zauner2011}, but we will prove it here directly with Lemma~\ref{overlappack} and state it within the context of spectrahedron arrangements.

\begin{thm}\label{prop:ff1}
Let $m=(m_0, \hdots, m_s)$ be a vector of odd integers $\geq 3$ with $\abs{m} = \prod_{\ell = 0}^s m_\ell$ and construct $\cG(m)$ with binder $\cB$.  There exists at least one $\cA\subset \cB$ which is a partition of $\left(\bigoplus_{\ell=0}^s \bZ_{m_\ell}\right)^2$.  Then for $\beta, \beta' \in \cA$ with $\beta \neq \beta'$ 
\[
\tr(P_{\beta} P_{\beta'}) = 1, 
\]
where $P_{\beta}$ is the orthogonal projection onto the subspace spanned by the vectors in $\beta$.  Thus $\cA$ generates an $(\abs{m}, \abs{m}(\abs{m}-1)/2)$-spectrahedron arrangement of purity $1/(\abs{m}-1)$ which saturates the first Rankin bound, that is, an equichordal arrangement of $\abs{m}$  $(\abs{m}-1)$-dimensional subspaces of $\bC^{\abs{m}(\abs{m}-1)/2}$.
\end{thm}
\begin{proof}
We note that $\cA=\{ \beta_k : k \in \bigoplus_{\ell=0}^s \bZ_{m_\ell}\}$ is one such subset of the binder which is a partition of $\left(\bigoplus_{\ell=0}^s \bZ_{m_\ell}\right)^2$.  For each $\beta \in \cA$, we define $W_\beta = P_\beta/\tr(P_\beta)$, which have trace $1$. Further for each $\beta \in \cA$, \[
\tr(W_\beta^2) = \tr(P_\beta^2)/\tr(P_\beta)^2 = \frac{1}{\abs{m}-1};
\]
thus the purity is $\gamma =1/(\abs{m}-1)$.

%Let $\beta \in \cB$.  Then the vectors $\{\pi(k, \kappa) \psi : (k,\kappa) \in \beta\}$ are a simplex and hence $P_\beta = \frac{1}{\abs{m}}\Phi_\beta \Phi_\beta^\ast$, where $\Phi_\beta$ is the matrix with columns $\{\pi(k, \kappa) \psi : (k,\kappa) \in \beta\}$ (cf.~Definition~\ref{defn:simpl}).

Let $\beta, \beta' \in \cA$ with $\beta \neq \beta'$. Then
\[
\frac{1}{\gamma^2} tr(W_{\beta} W_{\beta'}) = \tr(P_{\beta} P_{\beta'}) =1 = (\abs{m}-1)^2 \frac{2\abs{m} - \abs{m}}{\abs{m}(\abs{m}-1)^2},%=\frac{\abs{m} (\abs{m}-1)^2 - (\abs{m} (\abs{m}-1)/2) (\abs{m}-1)}{ (\abs{m} (\abs{m}-1)/2) (\abs{m}-1)},
\]
which saturates the first Rankin  bound \eqref{eqn:Rankin1}.
%\begin{align*}
%\tr(P_{\beta} P_{\beta'}) &= \frac{1}{\abs{m}^2}\tr(\Phi_{\beta} \Phi_{\beta}^\ast \Phi_{\beta'} \Phi_{\beta'}^\ast) =  \frac{1}{\abs{m}^2}\norm{\Phi_{\beta}^\ast \Phi_{\beta'}}_{H.S.}^2 \\
%&=  \frac{1}{\abs{m}^2}\sum_{(k,\kappa) \in \beta} \sum_{(k',\kappa') \in \beta'} \absip{\pi(k,\kappa)\psi}{\pi(k',\kappa')\psi}^2\\
%&=  \frac{1}{\abs{m}^2} \abs{\beta}\abs{\beta'}  =1 =\frac{\abs{m} (\abs{m}-1)^2 - (\abs{m} (\abs{m}-1)/2) (\abs{m}-1)}{ (\abs{m} (\abs{m}-1)/2) (\abs{m}-1)},
%\end{align*}

\end{proof}

We note that in \cite{BoPa15}, Gabor frames are formed from the orbit of $\mathbbm{1}_D$ under $\sigma(\bZ_m \times \bZ_m)$, for some $m \in \bN$, where $D$ is a so-called \emph{difference set} in $\bZ_m$ and $\mathbbm{1}_D$ is the vector that is $1$ on $D$ and $0$ on $\bZ_m \backslash D$.  The resulting frames are not equiangular but rather \emph{biangular}; however, $\{\beta_k = \{ \sigma(k, \kappa) \mathbbm{1}_D : \kappa \in  \bZ_{m} \}: k \in \bZ_m\}$ forms an optimal (equichordal) $(m, m)$-spectrahedron arrangement of purity $1/\abs{D}$, where in general $\abs{D} < m-1$.

\begin{thm}\label{thm:affrank}
Let $\Phi$ be an equiangular tight frame with binder $\cB$ which is a simple affine $(v,k,\lambda)$-BIBD.  Then the collection of orthogonal projections onto the simplices $\{P_\beta\}_{\beta \in \cB}$ generates a $\left(\frac{v(v-1)}{k(k-1)}\lambda,\frac{(k-1)^2 v}{k^2-2k+v}\right)$-spectrahedron arrangement with purity $1/(k-1)$ which saturates the second Rankin spectrahedron bound, that is, an arrangement of $\frac{v(v-1)}{k(k-1)}\lambda$  scaled projections of rank $k-1$ on the Hilbert space $\bC^{\frac{(k-1)^2 v}{k^2-2k+v}}$.
\end{thm}
\begin{proof}
Since $\cB$ is a $(v,k,\lambda)$-BIBD, we know that $\Phi$ has $v$ vectors, and  it follows from Corollary~\ref{cor:affint} that the binder has $\frac{v(v-1)}{k(k-1)}\lambda$ simplices spanning subspaces of dimension $k-1$. Further, the evaluation of the first Rankin spectrahedron bound \eqref{eqn:Rankin1} on both the simplices and the entire equiangular tight frame yields
\[
\frac{1}{k-1} = \sqrt{\frac{v-d}{d(v-1)}} \quad \Rightarrow \quad d=\frac{(k-1)^2v}{k^2 -2k+v},
\]
where $d$ is the dimension of the span of $\Phi$. Thus, for each $\beta \in \cB$
$W_\beta = P_\beta/(k-1)$ satisfies
\[
\tr W_j = \frac{k-1}{k-1} =1 \quad \textrm{and} \quad \tr W_j^2 = \frac{k-1}{(k-1)^2} = \frac{1}{k-1},
\]
yielding the parameters of the spectrahedron arrangement.

Since $\cB$ is simple, it does not have any repeated blocks, and we may apply Corollary~\ref{cor:affint} to determine that any two distinct blocks $\beta, \beta' \in \cB$ intersect in a set of size $1$ or $k^2/v$. It follows from Lemma~\ref{overlappack} that in the former case, $\tr(P_\beta P_{\beta'})=1$ and in the latter
\[
(k-1)^2\tr(W_{\beta} W_{\beta'})= \tr(P_\beta P_{\beta'}) = \frac{(k-1)^2}{\frac{(k-1)^2v}{k^2-2k+v}} = (k-1)^2 \frac{1}{\frac{(k-1)^2v}{k^2-2k+v}},
%= 1+ \frac{k^2}{v}\frac{k-2}{k} = \frac{v+k(k-2)}{v} = \frac{(k-1)^2}{\frac{(k-1)^2v}{k^2-2k+v}}.
\]
which saturates the second Rankin bound \eqref{eqn:Rankin2}.
\end{proof}

\begin{cor}\label{cor:ff2}
Let $p$ be an odd prime and construct $\cG(p)$ with binder $\cB$.  Then for $\beta, \beta' \in \cB$ with $\beta \neq \beta'$
\[
\tr(P_{\beta} P_{\beta'}) \in \left\{1, \frac{2(p-1)}{p} \right\}
\]
where $P_{\beta}$ is the orthogonal projection onto the subspace spanned by the vectors in $\beta$.  Thus $\cB$ generates a $(p(p+1),p(p-1)/2)$-spectrahedron arrangement with purity $1/(p-1)$ which saturates the second Rankin bound, that is, an arrangement of $p(p+1)$ 
states of rank $p-1$ on the Hilbert space $\bC^{p(p-1)/2}$. This spectrahedron arrangement is optimal in the spectrahedron $\cS$ defined by
the trace-normalized positive semidefinite operators in the real span of
$ \{ \pi(k,\kappa) \psi \otimes \psi^\ast \pi(k,\kappa)^\ast: (k,\kappa) \in  \bZ_p \times \bZ_p\} \, .
$
%or any other $(p^2-1)$-dimensional spectrahedron containing the arrangement.
\end{cor}
\begin{proof}
We know from Theorem~\ref{thm:AG} that $\cB = AG(2,p)$, which is a simple, affine $(p^2,p,1)$-BIBD. Thus the first part of the claim follows immediately from Theorem~\ref{thm:affrank}.
We note that there are $p(p+1)$ elements in the packing and $p(p+1)> p^2-1 = D_{\cS}$.  Thus the arrangement saturates the second Rankin spectrahedron bound \eqref{eqn:Rankin2} with sufficiently many elements and is optimal.
\end{proof}
The evidence suggests that the spectrahedron arrangement in Corollary~\ref{cor:ff2} is not an optimal spectraplex arrangement.  Initially, for $p > 3$ 
\[
p(p+1) < \frac{p^2(p^2+1)}{4} = \DSYM = D_{\textrm{spectraplex}} + 1,
\]
meaning that saturation of the second Rankin bound does not guarantee optimality as a packing of the spectraplex.  Furthermore, numerical testing using a Matlab implementation of \cite{DHST08} suggests that there exists an equichordal $(m(m+1),m(m-1)/2)$-spectraplex arrangement with purity $1/(m-1)$ that saturates the first Rankin spectrahedron bound for $m=4,5,6,7$.  In any case, the numerical experiments yield arrangements which are more spread with respect to chordal distance.

%
%Unfortunately, for $p \geq 5$, $p(p+1) < \frac{p^2 (p-1)^2}{4} = \DSYM$.  In this case the subspace arrangement generated in Corollary~\ref{cor:ff2} is not orthoplex-bound achieving.   Indeed, numerical testing using a Matlab implementation of \cite{DHST08} seems to suggest that there exists a $(m(m+1), m-1, m(m-1)/2)$-subspace packing which is equichordal and saturates the first Rankin bound $(m-1)(2m+1)/(m^2+m-1)$ (see Theorem~\ref{thm:rankin}) for $m \in \{4, 5, 6, 7\}$. \ejk{Add new packing shit}

%However, we note that the first Rankin bound of $(p(p+1),p-1,p(p-1)/2)$-subspace arrangements is $(p-1)(2p+1)/(p^2+p-1)$; thus, in particular for larger values of $p$, $\max_{\beta \neq \beta'} \tr(P_\beta P_\beta')$ is close to the first Rankin bound. 

\begin{cor}\label{cor:ff3}
Let $p$ be an odd prime and construct $\cG(p,p)$ with binder $\cB$.  Then for $\beta, \beta' \in \cB$ with $\beta \neq \beta'$
\[
\tr(P_{\beta} P_{\beta'}) \in \left\{1, \frac{2(p^2-1)}{p^2}, \frac{(p-1)(p+2)}{p} \right\}
\]
where $P_{\beta}$ is the orthogonal projection onto the subspace spanned by the vectors in $\beta$.  Further, $\cB$ generates a $(p^2(p^2+1)(p+1),p^2(p^2-1)/2)$-spectrahedron arrangement of purity $1/(p^2-1)$.
\end{cor}
\begin{proof}
The result follows from Theorems~\ref{thm:AG2} and~\ref{thm:bibddecomp} as well as Lemma~\ref{overlappack}.
\end{proof}
It is hard to ascertain the potential optimality of  the arrangement in Corollary~\ref{cor:ff3} since $\max_{\beta \neq \beta'} \tr (P_\beta P_\beta')$ does not saturate either Rankin bound.

\section{Mutually unbiased bases from a non-maximal equiangular tight frame} \label{sec:BinderMUB}
It is known that $AG(2,q)$ for $q$ a prime power induces maximal mutually unbiased bases in $\bC^q$ by assigning to each line in the affine geometry a vector  in $\bC^q$ \cite{Zauner1999,Zauner2011,GHW04,Woo06}.  In this set-up, each of the $q+1$ sets of $q$ parallel lines in $AG(2,q)$ are mapped to an orthonormal basis.  The mapping from the geometry over the finite field plane to complex space is not completely trivial, although constructions are known for each prime power $q$.  One such example arises from the Hesse SIC.
\begin{ex}\label{ex:MUB}
One can use the binder of the Hesse SIC $\cG(3)$ ($\Phi_2$ in Example~\ref{ex:hesse}) to construct a maximal set of mutually unbiased bases in $\bC^3$~\cite{DBBA2013}.  Namely, for each $\beta \in \cB(\cG(3))$, we chose $\eta^{(\beta)} \in \bC^3$ to be a unit vector in $\left( \lspan_{(k,\kappa)\in\beta} M_3^\kappa T_3^k \psi\right)^\perp$.  Let $\alpha = \frac{1}{\sqrt{3}}$.   A convenient choice of the $\eta^{(\beta)}$ is
\[
\begin{array}{c||ccc|ccc|ccc|ccc} 
\beta_\ast & 0 &  1 & 2 & 0,0 & 0,1&  0,2 &  1,0 &  1,1 & 1,2 & 2,0 &  2,1&  2,2\\\hline
& 0 & 0 & 1 & \alpha & \alpha & \alpha & \alpha & \alpha & \alpha &\alpha & \alpha & \alpha  \\ 
\eta^{(\beta)} & 1 & 0 & 0 & \alpha & \alpha \zeta_3^2 & \alpha  \zeta_3 & \alpha \zeta_3^2 & \alpha \zeta_3 & \alpha &\alpha \zeta_3 & \alpha & \alpha \zeta_3^2 \\
& 0 & 1 & 0 & \alpha & \alpha \zeta_3 & \alpha \zeta_3^2 & \alpha & \alpha \zeta_3 & \alpha \zeta_3^2 &\alpha & \alpha \zeta_3 & \alpha  \zeta_3^2 
\end{array}.
\]
In this case $\{ \eta^{(\beta_k)} \}_{k \in \bZ_3}$ is the standard orthonormal basis for $\bC^3$, $\{ \eta^{(\beta_{0,b}))} \}_{b \in \bZ_3}$ is the Fourier basis (i.e. normalized columns of $F_3$), and $\{ \eta^{(\beta_{1,b}))} \}_{b \in \bZ_3}$ and $\{ \eta^{(\beta_{2,b}))} \}_{b \in \bZ_3}$ are so-called \emph{chirp bases}.
\end{ex}
Example~\ref{ex:MUB} generalizes for odd prime $p$, where the binder of $\cG(p)$ induces a maximal set of mutually unbiased bases.  In particular, Theorem~\ref{thm:gabMUB} gives a previously unknown connection between non-maximal equiangular tight frames and maximal sets of mutually unbiased bases.
\begin{thm}\label{thm:gabMUB}
Let $p$ be an odd prime. Then the image of the mapping $\cB(\cG(p)) \rightarrow \bC^p$ defined by
\[
\beta \mapsto \eta^{(\beta)}, \quad \eta^{(\beta)}_j = \left\{ \begin{array}{lr}\delta_{k+(p-1)/2,j}; & \beta = \beta_k, k \in \bZ_p \\ \frac{1}{\sqrt{p}} \zeta_p^{(p+1)/2j (aj + 2b +a)}; & \beta=\beta_{a,b}, a,b \in \bZ_p\end{array} \right.
\]
is a maximal set of mutually unbiased bases, where each parallel class in $\cB(\cG(p))$ corresponds to the vectors in one basis.  Furthermore, each $\eta^{(\beta)}$ satisfies
\[
\lspan \eta^{(\beta)} = \left( \lspan_{(k,\kappa) \in \beta, i \in \mathcal{I}} \left\{ M_p^\kappa T_p^k \phi_i \right\}\right)^\perp,
\]
where the $\phi_i$ are the component vectors used to construct the Gabor-Steiner equiangular tight frames (Definition~\ref{defn:psi}).
\end{thm}
\begin{proof}
We will first show that $\{\eta^{(\beta)}\}_{\beta \in \cB(\cG(p))}\}$ forms a maximal set of mutually unbiased bases with parallel classes corresponding to the constituent orthonormal bases.  The vectors result from an affine re-parameterization of the mutually unbiased basis construction in \cite{WoFi89,Arc05}, but we will include a proof here for completeness.

We recall that $\cB(\cG(p))\cong AG(2,p)$ (Theorem~\ref{thm:AG}) consists of two types of simplices,
\[
\cB_0 = \left\{ \beta_k = \{ (k, \kappa): \kappa \in \bZ_p\} : k \in \bZ_p\right\},
\]
which correspond to the vertical lines in $AG(2,p)$ and
\[
\cB_1  = \left\{\beta_{a,b} = \{ (k, ak+b): k \in \bZ_p\}: a,b \in \bZ_p\right\},
\]
which correspond to the non-vertical lines.  From inspection, one can note that $\{ \eta^{(\beta_k)} \}_{k \in \bZ_p}$ corresponding to the parallel class $\cB_0$ is simply the standard orthonormal basis and for any $k, a, b \in \bZ_p$,
\[
\absip{\eta^{(\beta_k)}}{\eta^{(\beta_{a,b})}} = \abs{\frac{1}{\sqrt{p}} \zeta_p^{\frac{p+1}{2} (k+(p-1)/2) [ a (k+(p-1)/2) + 2b + a]}} = \frac{1}{\sqrt{p}},
\]
as desired.  We now seek to characterize the inner products of the $\eta^{(\beta)}$ corresponding to different $\beta$'s in $\cB_1$.  Namely, let $a,\tilde{a}, b, \tilde{b} \in \bZ_p$.  We compute
\begin{equation}\label{eqn:absum}
\ip{\eta^{(\beta_{a,b})}}{\eta^{(\beta_{\tilde{a},\tilde{b}})}} = \frac{1}{p} \sum_{j \in \bZ_p} \zeta_p^{\frac{p+1}{2}j [(a-\tilde{a})j + 2(b-\tilde{b})]}
\end{equation}
If $a = \tilde{a}$, then 
\[
\eqref{eqn:absum} = \frac{1}{p}\sum_{j \in \bZ_p} \zeta_p^{(b-\tilde{b})j} = \left\{ \begin{array}{lr} 1; & b=\tilde{b} \\ 0; &  b \neq \tilde{b} \end{array}\right..
\]
Thus each parallel class corresponding to a slope $a$ yields an orthonormal basis.  If $a \neq \tilde{a}$, then
\begin{align}
\eqref{eqn:absum} &= \frac{1}{p} \sum_{j \in \bZ_p} \zeta_p^{\frac{p+1}{2} \left[(a-\tilde{a})\left(j+\frac{b-\tilde{b}}{a-\tilde{a}}\right)^2 - \frac{(b-\tilde{b})^2}{a-\tilde{a}}\right]} =\frac{1}{p}  \zeta_p^{-\frac{p+1}{2} \frac{(b-\tilde{b})^2}{a-\tilde{a}}} \sum_{j \in \bZ_p} \zeta_p^{\frac{p+1}{2}(a-\tilde{a})j^2} \label{eqn:gauss1}\\
&= \frac{1}{p}  \zeta_p^{-\frac{p+1}{2} \frac{(b-\tilde{b})^2}{a-\tilde{a}}}  \left\{\begin{array}{lr}\sqrt{p}; & p\equiv 1 \mod 4 \\ \zeta_4\sqrt{p}; & p\equiv 3 \mod 4 \end{array} \right. \label{eqn:gauss2}\\
&\Rightarrow \absip{\eta^{(\beta_{a,b})}}{\eta^{(\beta_{\tilde{a},\tilde{b}})}} = \frac{1}{\sqrt{p}}, \nonumber
\end{align}
where \eqref{eqn:gauss2} is due to the sum on the right hand side of  \eqref{eqn:gauss1} being a quadratic Gauss sum and $\zeta_4$ is a primitive $4$th root of unity depending on the choice of the primitive $p$th root of unity $\zeta_p$.  Thus the $\{ \eta^{(\beta)}\}_{\beta \in \cB(\cG(p))}$ form a maximal set of mutually unbiased bases for $\bC^p$.

We now show that the $\eta^{(\beta)}$ have the desired relationship with the vectors in $\cG(p)$ by demonstrating two facts:
\begin{enumerate}
\item Claim 1: $ \eta^{(\beta)} \in \left( \lspan_{(k,\kappa) \in \beta, i \in \mathcal{I}} \left\{ M_p^\kappa T_p^k \phi_i \right\}\right)^\perp$ and
\item Claim 2: $\dim \lspan_{(k,\kappa) \in \beta, i \in \mathcal{I}} \left\{ M_p^\kappa T_p^k \phi_i \right\} \geq p-1$.
\end{enumerate}

To show the first claim, we consider two cases, $\beta \in \cB_0$ and $\beta \in \cB_1$.  For the former case, it follows from the definition of the $\phi_i$ that for a fixed $k \in \bZ_p$, $\lspan_{\kappa \in \bZ_p, i \in \mathcal{I}} \left\{ M_p^\kappa T_p^k \phi_i \right\}$ consists solely of vectors which are zero at $k+\frac{p-1}{2}$ and thus orthogonal to $\eta^{(\beta_k)}$.  For the latter case, we fix $a, b \in \bZ_p$ and let $i \in \mathcal{I}$, $k \in \bZ_p$ be arbitrary.  Then
\begin{align}
\lefteqn{\sqrt{p} \ip{\eta^{(\beta_{a,b})}}{M_p^{ak+b} T_p^k \phi_i} }\nonumber\\
&= \ip{\left(\zeta_p^{\frac{p+1}{2}j (aj + 2b +a) }\right)_j}{\left(\zeta_p^{(ak+b)(i+k)}\delta_{i+k,j} -\zeta_p^{(ak+b)(-i-1+k)} \delta_{-i-1+k,j}\right)} \nonumber\\
&= \zeta_p^{\frac{p+1}{2} (i+k)(ai+ak+2b+a) -(ak+b)(i+k)} -\zeta_p^{\frac{p+1}{2} (-i-1+k)(-ai-a+ak+2b+a) -(ak+b)(-i-1+k)}\nonumber\\
&=: \zeta_p^{\textrm{EXP1}} - \zeta_p^{\textrm{EXP2}}.\label{eqn:zero}
\end{align}
Computing over $\bZ_p$, we have
\begin{align*}
\textrm{EXP1} &= \frac{p+1}{2} a(i+k)(i-k+1) = \textrm{EXP2};
\end{align*}
thus, \eqref{eqn:zero} $= 0$, independent of choice of $i$ and $k$.

To prove the second claim, we will again consider cases based on whether $\beta$ lies in $\cB_0$ or $\cB_1$.  For $\beta_k \in \cB_0$, 
\[
\dim \lspan_{\kappa \in \bZ_p, i \in \cI} \left\{ M_p^\kappa T_p^k \phi_i\right\} = p-1
\]
by inspection.  Let $\beta_{a,b} \in \cB_1$.  Then $(0,b), (1,b) \in \beta_{a,b}$.  The $(p-1)/2$ vectors $\{ M_p^b \phi_i \}_{i \in \cI}$ have disjoint support and are thus orthogonal.  The same holds for $\{ M_p^{a+b} T_p \phi_i \}_{i \in \cI}$.   We assume by way of contradiction that
\[
\dim \lspan \left( \{ M_p^b \phi_i \}_{i \in \cI} \cup \{ M_p^{a+b} T_p \phi_i \}_{i \in \cI} \right) < p-1.
\]
By symmetry, we can assume that there exists an $\tilde{i} \in \cI$ such that $M_p^{a+b} T_p \phi_{\tilde{i}} \in \lspan_{i \in \cI} M_p^b \phi_i$.  By orthogonality, this implies
\begin{equation}\label{eqn:MTbwo}
M_p^{a+b} T_p \phi_{\tilde{i}} = \sum_{i \in \cI} \ip{M_p^{a+b} T_p \phi_{\tilde{i}}}{M_p^b \phi_i} \frac{M_p^b \phi_i}{\norm{M_p^b \phi_i}}.
\end{equation}
However, due to the support of the vectors, there are only one or two (depending on the exact value of $\tilde{i}$) $i \in \cI$ such that $\ip{M_p^{a+b} T_p \phi_{\tilde{i}}}{M_p^b \phi_i} \neq 0$, where the supports of the one or two $M_p^b \phi_i$ and $M_p^{a+b} T_p \phi_{\tilde{i}}$ are all necessarily different.  This, however, means that the support of the right hand side of \eqref{eqn:MTbwo} is necessarily different from the support of the left hand side, which is a contradiction.  Hence $\dim \lspan_{k \in \bZ_p, i \in \cI} \left\{ M_p^{ak+b} T_p^k \phi_i\right\} \geq p-1$, as desired.
\end{proof}
Such a mapping is called a \emph{quantum net} and yields not only a maximal set of mutually unbiased bases but also a construction of the Wootters-Wigner distribution \cite{GHW04}.  The correspondence of lines in $AG(2,q)$ to specially structured vectors (or, equivalently, pure states) generalizes the fact that marginal probabilities of the continuous (over $\bR$) Wigner function over parallel affine lines in phase space ($\bR \times \widehat{\bR}$) correspond to the probability distribution associated with the observable along that oriented axis \cite{BB87,GHW04}, which has significance for quantum tomography.  By using the trace over $GF(p^s)$ (Definition~\ref{def:trace}), one may generalize the definition of the $\eta^{(\beta)}$ above to $\bC^{p^s}$ \cite{WoFi89,Arc05}; however, the relationship with the vectors in a Gabor-Steiner equiangular tight frame would no longer hold.
\begin{cor}\label{cor:GaborMUB}
Let $m$ be an odd composite number with prime factorization $\prod_i p_i^{e_i}$ and $i^\ast = \argmin_i p_i$ satisfying $e_{i^\ast}=1$.  Further define $\left[ \bZ_m /\langle p_{i^\ast} \rangle \right]$ to be a set of coset representatives in $\bZ_m$ of the cyclic subgroup generated by $p_{i^\ast}$. Then the image of the mapping from an appropriate subset of $\cB(\cG(m))$ to  $\bC^m$ defined by
\[
\beta \mapsto \eta^{(\beta)}, \quad \eta^{(\beta)}_j = \left\{ \begin{array}{lr}\delta_{k+(m-1)/2,j}; & \beta = \beta_k, k \in \bZ_m \\ \frac{1}{\sqrt{m}} \zeta_m^{(m+1)/2j (aj + 2b +a)}; & \beta=\beta_{a,b}, a \in \left[ \bZ_m /  \langle p_{i^\ast} \rangle \right],b \in \bZ_m\end{array} \right.
\]
is a set of $p_{i^\ast}+1$ mutually unbiased bases.  Furthermore, each $\eta^{(\beta)}$ satisfies
\[
\lspan \eta^{(\beta)} = \left( \lspan_{(k,\kappa) \in \beta, i \in \mathcal{I}} \left\{ M_m^\kappa T_m^k \phi_i \right\}\right)^\perp,
\]
where the $\phi_i$ are the component vectors used to construct the Gabor-Steiner equiangular tight frames (Definition~\ref{defn:psi}).
\end{cor}
\begin{proof}
With the exception of the evaluation of the Gauss sum in \eqref{eqn:gauss2}, each argument in the proof of Theorem~\ref{thm:gabMUB} applies after replacing $p$ with $m$.  In order for the Gauss sum on the right hand side of \eqref{eqn:gauss1} to have the correct absolute value of $\sqrt{m}$, $m$ and $a-\tilde{a}$ must be relatively prime.  The largest subsets of $\bZ_m$ such that every pair $a, \tilde{a}$ has a difference relatively prime to $m$ are sets of coset representatives of the largest proper cyclic subgroup, i.e., the one generated by $p_{i^\ast}$ which is isomorphic to $\bZ_{m/p_{i^\ast}}$.
\end{proof}
In a sense, Corollary~\ref{cor:GaborMUB} represents the best one can do with this type of approach.  Namely, if one defines a map from a multiplicatively closed subset of a finite ring of composite order $N=\prod_i p_i^{e_i}$ to $\bC^N$ which is the composition of a two-variable polynomial and a component-wise group homomorphism, one can obtain at most $\min_i p_i^{e_i}$ mutually unbiased bases \cite{Arc05}.  By appending the standard orthonormal basis, one obtains $1 + \min_i p_i^{e_i}$.

%%%%%%%%%%%%%%%
\oneappendix % use \appendix if you have more than one appendix
\begin{acknowledgements}
Part of this work was completed at the Oberwolfach Mini-Workshop in Algebraic, Geometric, and Combinatorial Methods in Frame Theory.
The authors would like to thank Marcus Appleby, Ingemar Bengtsson, Matt Fickus, Markus Grassl, Joey Iverson, John Jasper, Romanos-Diogenes Malikiosis, Chris Manon, and Dustin Mixon for fruitful discussions, some of which took place at the Oberwolfach Mini-Workshop and others at the Summer of Frame Theory at the Air Force Institute of Technology and at the Workshop on Finite Weyl-Heisenberg Groups which was part of the Hausdorff Trimester Program Mathematics of Signal Processing. The authors would also like to thank Alexander Rosa, who sent a hard copy of \cite{MaRo85}, which was not available online. 
\end{acknowledgements}

%%% DO NOT REPLACE WHAT FOLLOWS WITH A BBL FILE!!!!!!!!!!!
%\bibliography{BodmannKingBib}{}
%\bibliographystyle{abbrv}
%%% DO NOT REPLACE WHAT FOLLOWS WITH A BBL FILE!!!!!!!!!!!
%%% A LOT OF MANUAL EDITING WENT INTO WHAT FOLLOWS!!!!

\affiliationone{% in this example, two authors share an institution
   B.\ G.\ Bodmann\\
Department of Mathematics\\
641 Philip G. Hoffman Hall\\
University of Houston\\     
Houston, Texas 77204-3008 \\
United States of America  
   \email{bgb@math.uh.edu}}
% Important: Do not put any empty line here.
\affiliationtwo{% in this example, one author has two addresses}
   E.\ J.\ King\\
   1874 Campus Delivery\\
   111 Weber Bldg\\
   Fort Collins, CO 80523-1874\\
   United States of America  
   \email{emily.king@colostate.edu}}
%% Important: Do not put any empty line here.
%% Use \affiliationthree{} for any address positioned under \affiliationone
%% Use \affiliationfour{}  for any address positioned under \affiliationtwo
%\affiliationthree{~} %inserts a space to make this field empty
%\affiliationfour{%
%   Current address:\\
%   Present long-term address\\
%   Country
%   \email{t.hird@institution.edu}}
%%

\end{document}